\documentclass[11pt]{article}

\usepackage{amssymb}
\usepackage{amssymb,latexsym,amsmath,epsfig,amsthm,graphicx,afterpage,fancyhdr,
lettrine,xcolor}

\usepackage[labelfont={small,bf}, font={small}]{caption}

\graphicspath{ {images/} }
\usepackage[linesnumbered, lined,boxed, commentsnumbered]{algorithm2e}

\makeatletter

\renewcommand\section{\@startsection {section}{1}{\z@}
{-30pt \@plus -1ex \@minus -.2ex}
{2.3ex \@plus.2ex}
{\normalfont\normalsize\bfseries\boldmath}}

\renewcommand\subsection{\@startsection{subsection}{2}{\z@}
{-3.25ex\@plus -1ex \@minus -.2ex}
{1.5ex \@plus .2ex}
{\normalfont\normalsize\bfseries\boldmath}}

\renewcommand{\@seccntformat}[1]{\csname the#1\endcsname. }

\makeatother

\SetAlCapSkip{1em}

\newtheorem{theorem}{Theorem}
\newtheorem{lemma}[theorem]{Lemma}

\theoremstyle{definition}

\newtheorem{definition}[theorem]{Definition}

\newtheorem*{notation}{Notation}

\newtheorem*{remark}{Remark}

\begin{document}

\begin{center}
{\uppercase{\bf On the distribution of monochromatic complete subgraphs and
arithmetic progressions}}

\vskip 10pt

{\bf Aaron Robertson}\footnote{Principal corresponding author; \texttt{arobertson@colgate.edu}},
{\bf William Cipolli}\footnote{\texttt{wcipolli@colgate.edu}}, and 
{\bf Maria Dasc\u{a}lu}\footnote{Undergraduate student, \texttt{mdascalu@colgate.edu}}

\textit{Department of Mathematics, Colgate University, Hamilton, New York}

\end{center}

\begin{abstract} 
We investigate the distributions of the number of: (1) monochromatic complete
subgraphs over edgewise 2-colorings of complete graphs; and
(2) monochromatic arithmetic progressions over 2-colorings of intervals, as
statistical Ramsey theory questions. We present convincing
evidence that both distributions are very well-approximated
by the   Delaporte distribution.
\end{abstract}

\section{Introduction}
Ramsey theory deals with finding order among chaos, two  fundamental results espousing this being Ramsey's Theorem and van der Waerden's Theorem. Ramsey's Theorem,
in particular, proves the existence, for any $k \in \mathbb{Z}^+$, of a minimal positive integer $R(k)$ such that every 2-coloring of the edges of a complete graph on $R(k)$ vertices contains a monochromatic complete subgraph on $k$ vertices. Van der Waerden's Theorem, in particular, states that there exists a least positive integer $w(k)$ such that every $2$-coloring of $\{1, 2, \dots, w(k)\}$ contains a monochromatic arithmetic progression of length $k$.

While the definitions of both Ramsey and van der Waerden numbers are simple, the computations of both, especially Ramsey numbers, are notoriously difficult. 
For examples, the most recent Ramsey number was determined by McKay and Radziszowksi \cite{MR}, who used almost 10 years of cpu time;
Kouril \cite{K} used over 200 processors and 250 days
to show that $w(6)=1132$. Given the exponential nature of these numbers, the remaining unknown numbers seem intractable at the present time.
Given the difficulty of computing Ramsey numbers exactly we explore the
potential value of a statistical approach with an ultimate goal of gaining some insight into $R(k)$ and $w(k)$. Starting with Ramsey numbers, let all edgewise $2$-colorings of the
complete graphs on $n$
vertices be equally likely and define $X_k=X_k(n)$ as the random variable giving the total number of monochromatic subgraphs on $k$ vertices (i.e, $K_k$).
Our goal is to find a very good approximation for the probability mass function (pmf) of
$X_k$.

In \cite{G}, it is shown that
$X_k$ is asymptotically Poisson as $k \rightarrow \infty$ (with certain conditions on $n$ and $k$); that is, with an appropriate restriction on $n$, as $k \rightarrow \infty$ we have
$\mathbb{P}(X_k =j) \approx \frac{\lambda^j e^{-\lambda}}{j!}$, where
$\lambda = \frac{{n \choose k}}{2^{{k \choose 2}-1}}$.
However, since this is an asymptotic (in $k$) result, using this for
small values of $k$ is not appropriate.  In this article, we present (hopefully very convincing) evidence
of what the distribution of $X_k$ for small $k$ may be.

\section{Sampling Algorithm for 2-Colored \boldmath $K_n$}
Given a user input of positive integers $n$, $k$, and $g$, our Python program {\tt GraphCount}\footnote[2]{Available at {\tt http://www.aaronrobertson.org}.} generates $g$ graphs, each on $n$ vertices, using an adjacency list. It colors the edges between pairs of vertices randomly using the Python random module. It then counts the total number of monochromatic complete subgraphs on
$k$ vertices of each given graph.  Compiling all results will give
us an empirical pmf.

Algorithm 1 is recursive  with a base case of $k=3$. For $k=3$, the Triangle-Counting Algorithm (Algorithm 2) is used. 

\IncMargin{1em}
\begin{algorithm} 
\BlankLine
\SetKwInOut{Input}{input}\SetKwInOut{Output}{output}
\Input{List of edge colorings of graph and $k$}
\Output{List $\mathcal{F}$ of  vertices of monochromatic  cliques of size $k$}
Run Algorithm 1/2 with $k-1$; call the output  subcliques\;
\For{every vertex $A$ in the graph, starting from $A=k$}
{\If{$A$ can support a clique of size $k$;}{
\For{each subclique $S$}{
ensure   $A > i$ for every vertex $i$ in $S$\;

\If{the subclique and $A$ form a monochromatic clique of size $k$}{
create a tuple of the subclique and $A$\;
add the tuple to $\mathcal{F}$\;
}
}
}}
return   $\mathcal{F}$\;
\BlankLine
\end{algorithm}
\vspace*{-20pt}
\centerline{\small {\bf Algorithm 1:} Recursive Counting Algorithm}

\vskip 20pt
{\tt GraphCount} has a run-time of  $O(n^2)$ for $k=3$ and $O(n^k)$ for $k \ge 4$. Table 1
 compiles a list of approximate run-times.

\begin{algorithm}[t!] 
\BlankLine

\SetKwInOut{Input}{input}\SetKwInOut{Output}{output}
\Input{$2$-colored graph}
\Output{List $\mathcal{F}$ of vertices of monochromatic triangles}
\For{every pair of vertices $A<B$ of the graph}{
determine the color $c$ connecting $A$ and $B$\;
let $\mathcal{N}$ be the set of neighbors common to $A$ and $B$\;
\For{every $C \in \mathcal{N}$}{
\If{color of edge between $A$ and $C$ is c and color of edge between $B$ and $C$ is $c$}{
add ($A$, $B$, $C$) to $\mathcal{F}$\;}
}}
return  $\mathcal{F}$\;
\BlankLine
\end{algorithm}

\newpage
\vspace*{-40pt}
 \centerline{\small   {\bf Algorithm 2:} Triangle-Counting Algorithm}
\vskip 20pt

Our goal is to run {\tt GraphCount} with $g \geq 1,000,000$ in order to
obtain an empirical probability mass function that is a fairly good approximation of the probability mass function.  However, this will only
allow us (given reasonable time constraints) to investigate $k=6$ for
a few values of $n$.  Furthermore,
as can be seen from the table of run-times (Table 1), gathering enough
samples to attempt distribution fitting for $k=7$ would require approximately
75 years on a single computer or a couple of years on the cluster
of 36 computers we have available to us (with dedicated use, which
we do not have).  So, at this
time, pursuit of $k=7$ is not realistic with our algorithm.

\begin{center} 
\begin{tabular} { |c|c|c|} 
\hline
Input $k$ & Input $n$ & Time per graph (sec)  \\
\hline \hline
3 & 6 & 0.0002 \\
4 & 18 & 0.00275 \\ 
5 & 43 & 0.15 \\
5 & 49 & 0.296 \\
6 & 102 & 22.59 \\
6 & 165 & 407 \\
7 & 205 & 2368 \\
\hline
\end{tabular}
\vskip 5pt 
{\small {\bf Table 1}: Run-times for {\tt GraphCount}}
\end{center}

In Figure 1, we present the empirical probability mass
functions for the number of monochromatic $K_k$ subgraphs over
$2$-colorings of the edges of $K_n$ for small $k$ and $n$.  You will notice
a similar shape for all presented.  This occurred in all histograms
we obtained (for sufficiently large sample sizes).

\vskip 30pt

\begin{figure}[h!]\tiny 
\begin{center} \hspace*{-.1in}
		\begin{tabular}{ccc}
		\includegraphics[scale=.21]{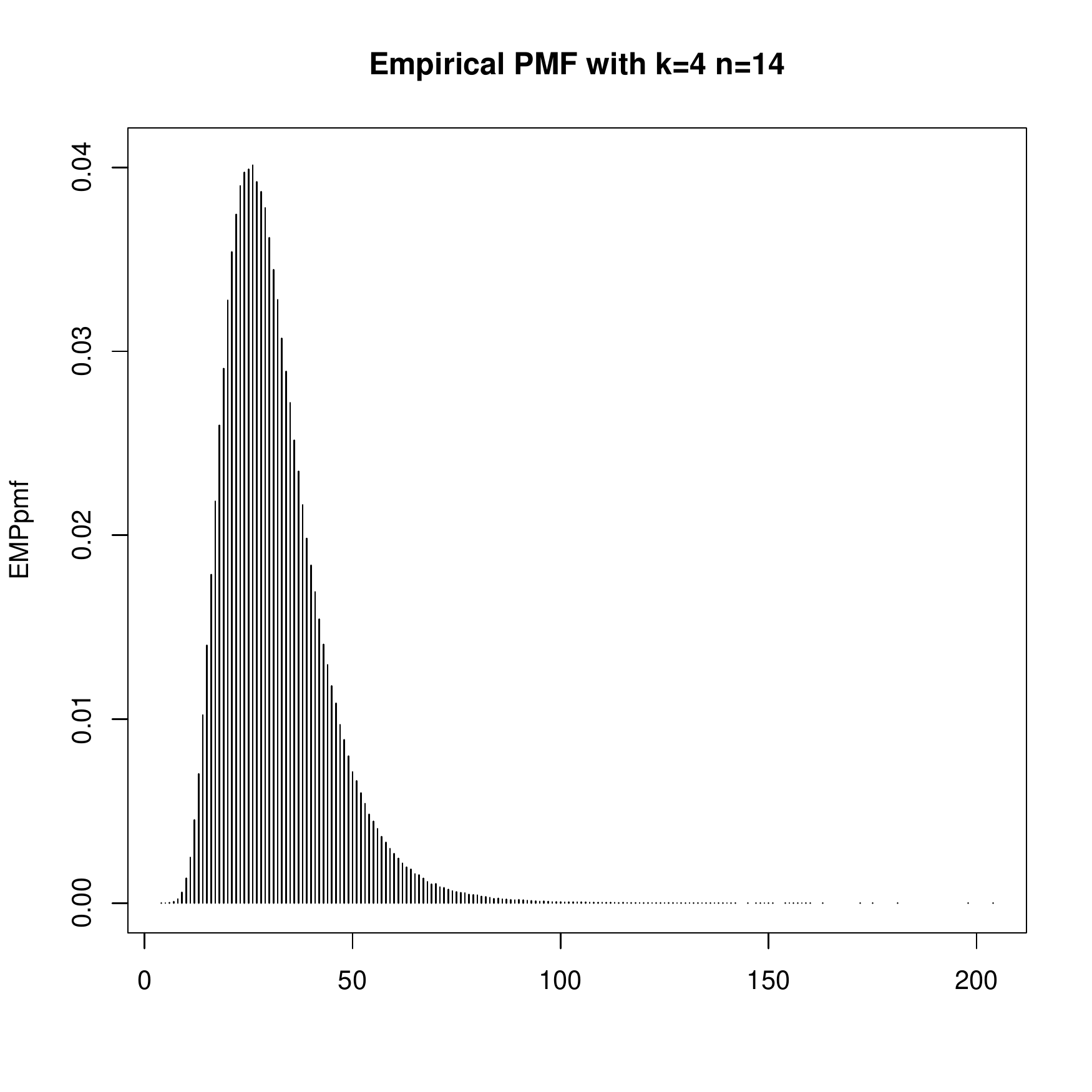} &  
		\includegraphics[scale=.21]{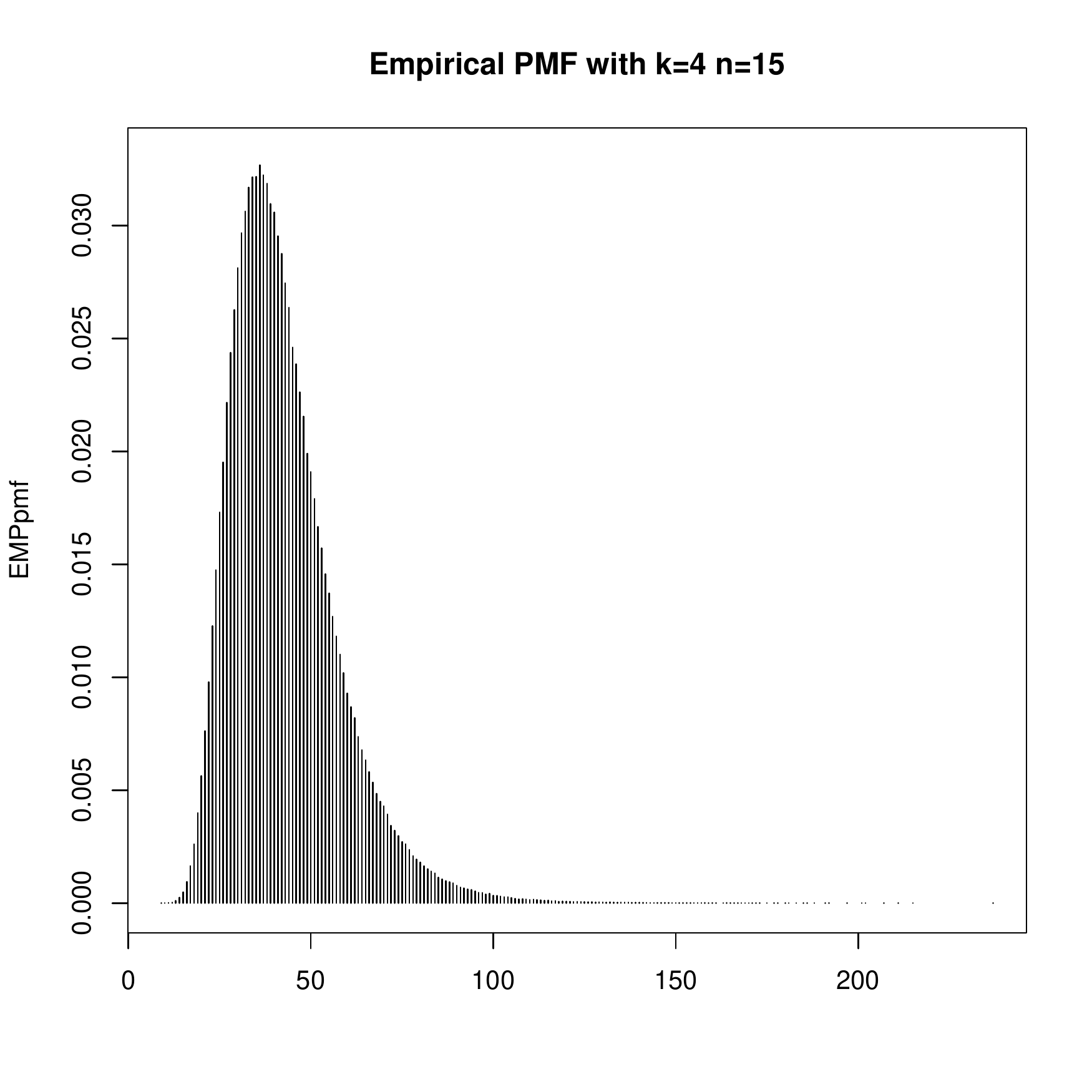} &   
		\includegraphics[scale=.21]{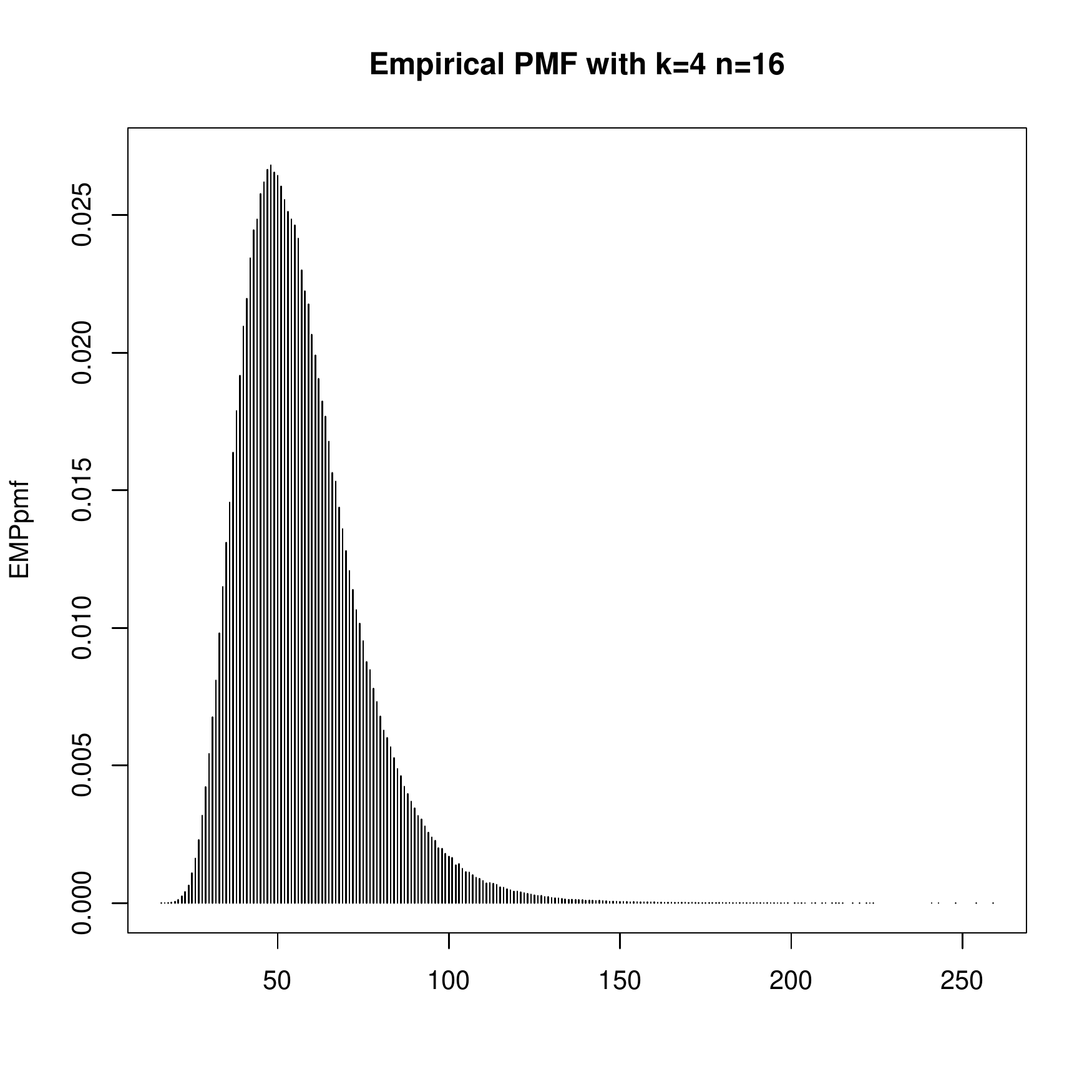} \vspace{-.4em}\\
			Sample size = 1M & Sample size = 1M & Sample size = 1M\\ 
		\includegraphics[scale=.21]{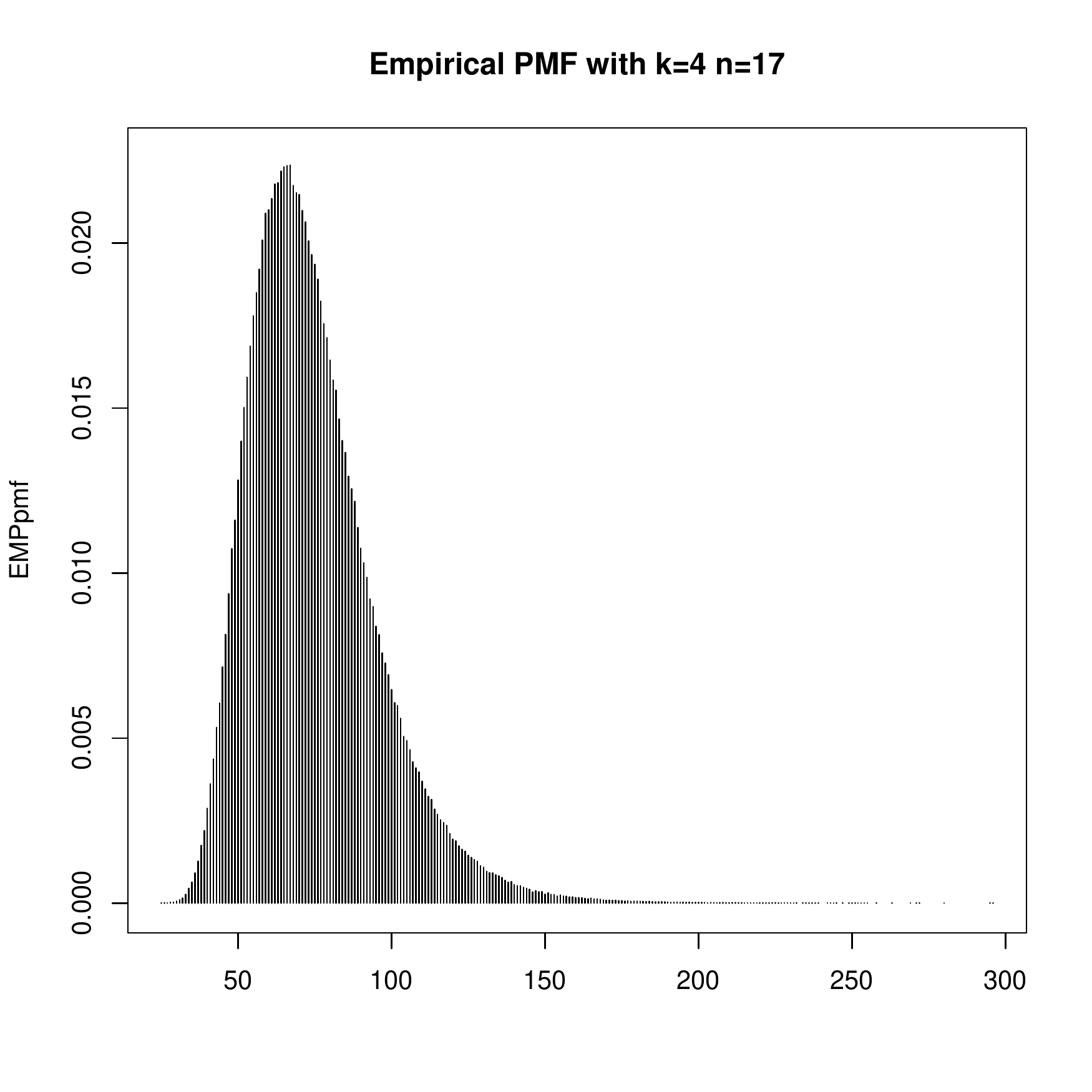} &  
		\includegraphics[scale=.21]{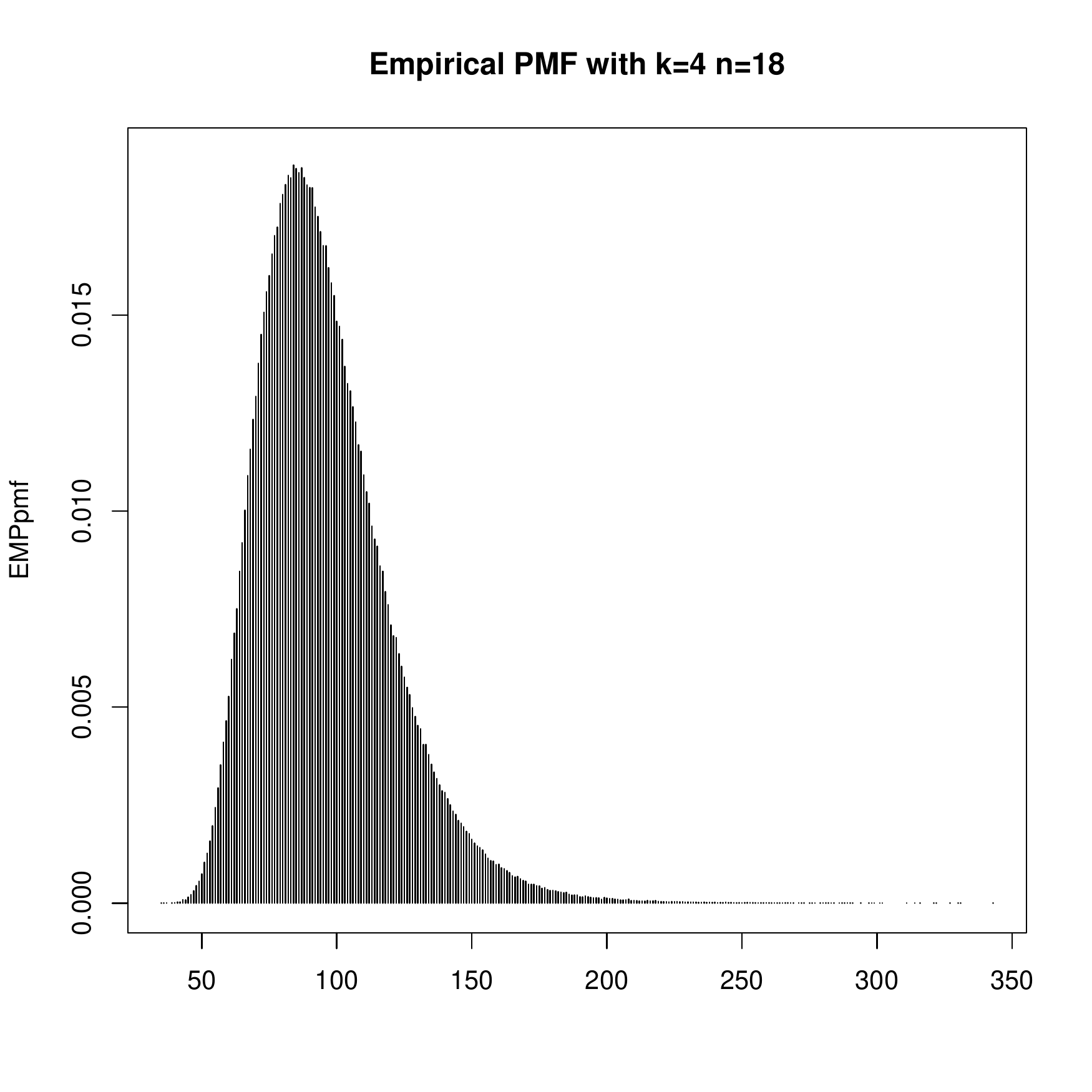} & 
		\includegraphics[scale=.21]{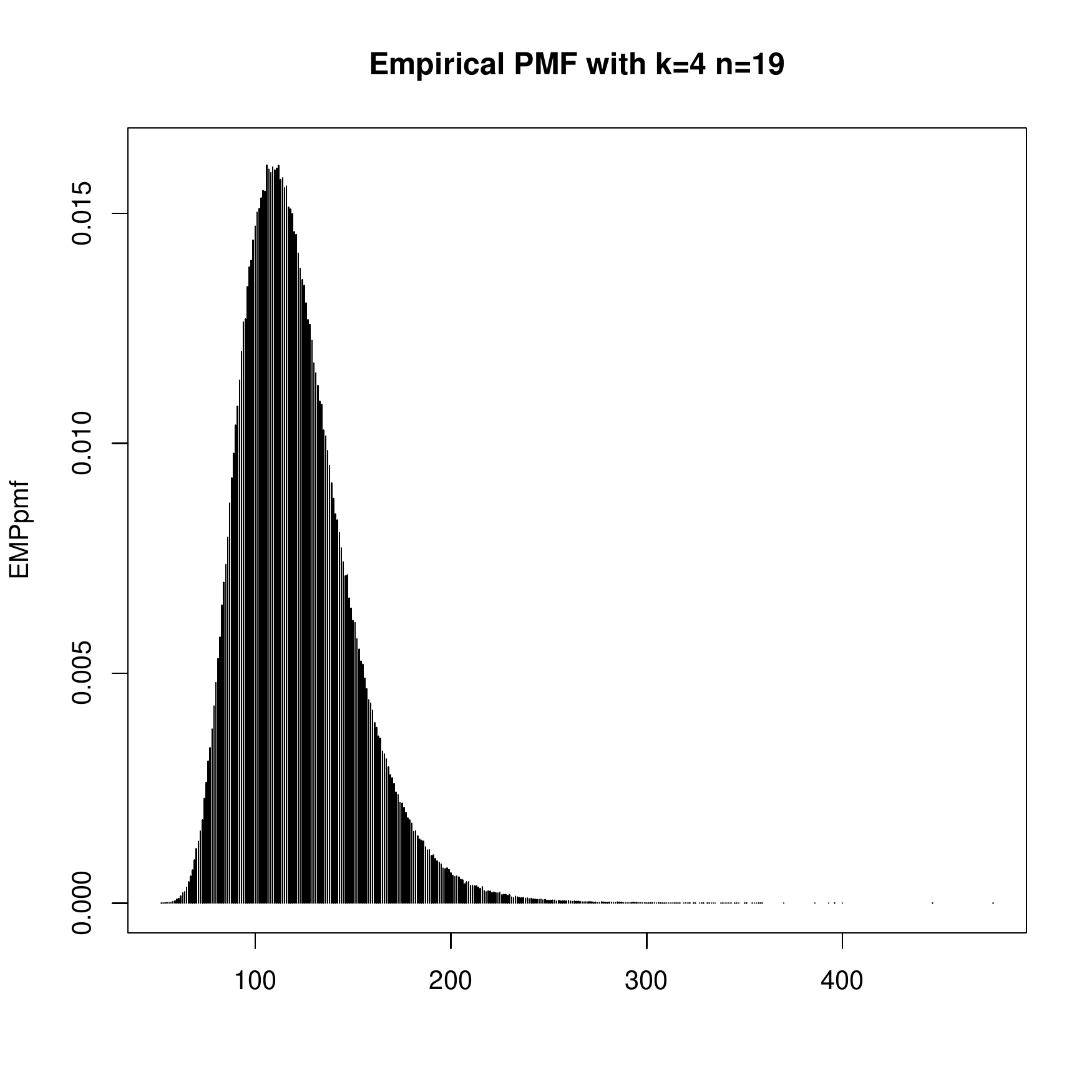} \vspace{-.4em}\\
		Sample size = 1M & Sample size = 1M & Sample size = 1M\\ 
		\includegraphics[scale=.21]{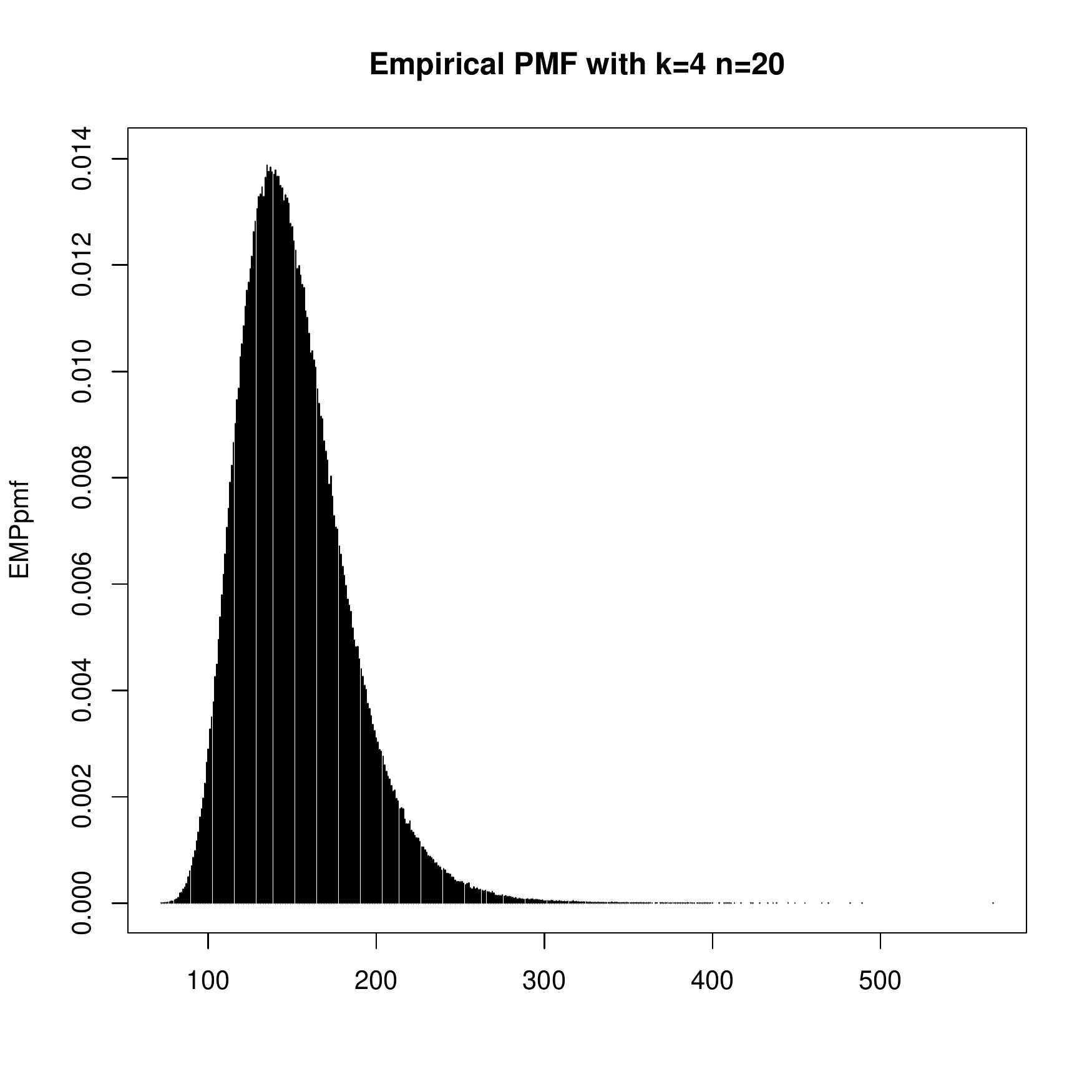} &  
		\includegraphics[scale=.21]{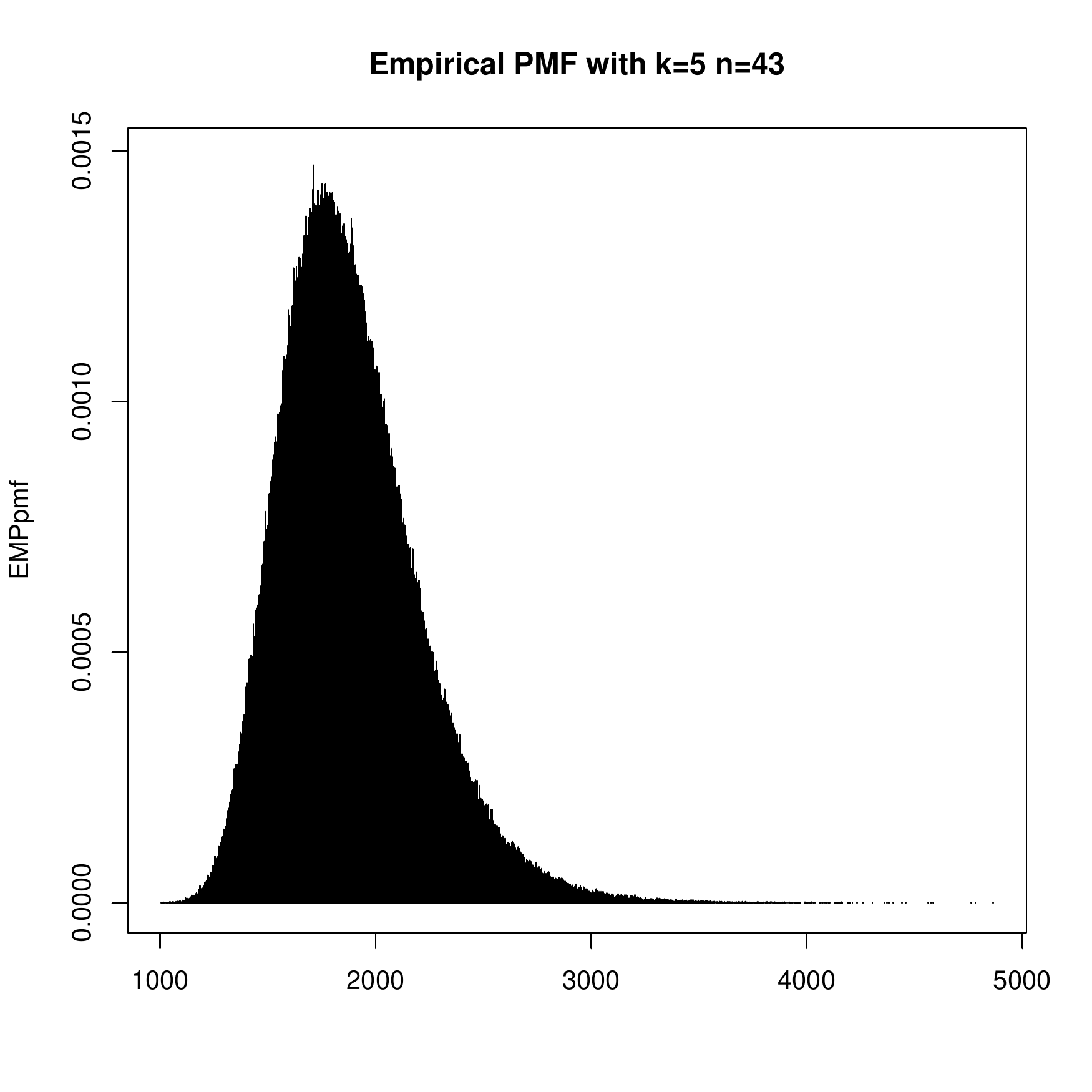} &   
		\includegraphics[scale=.21]{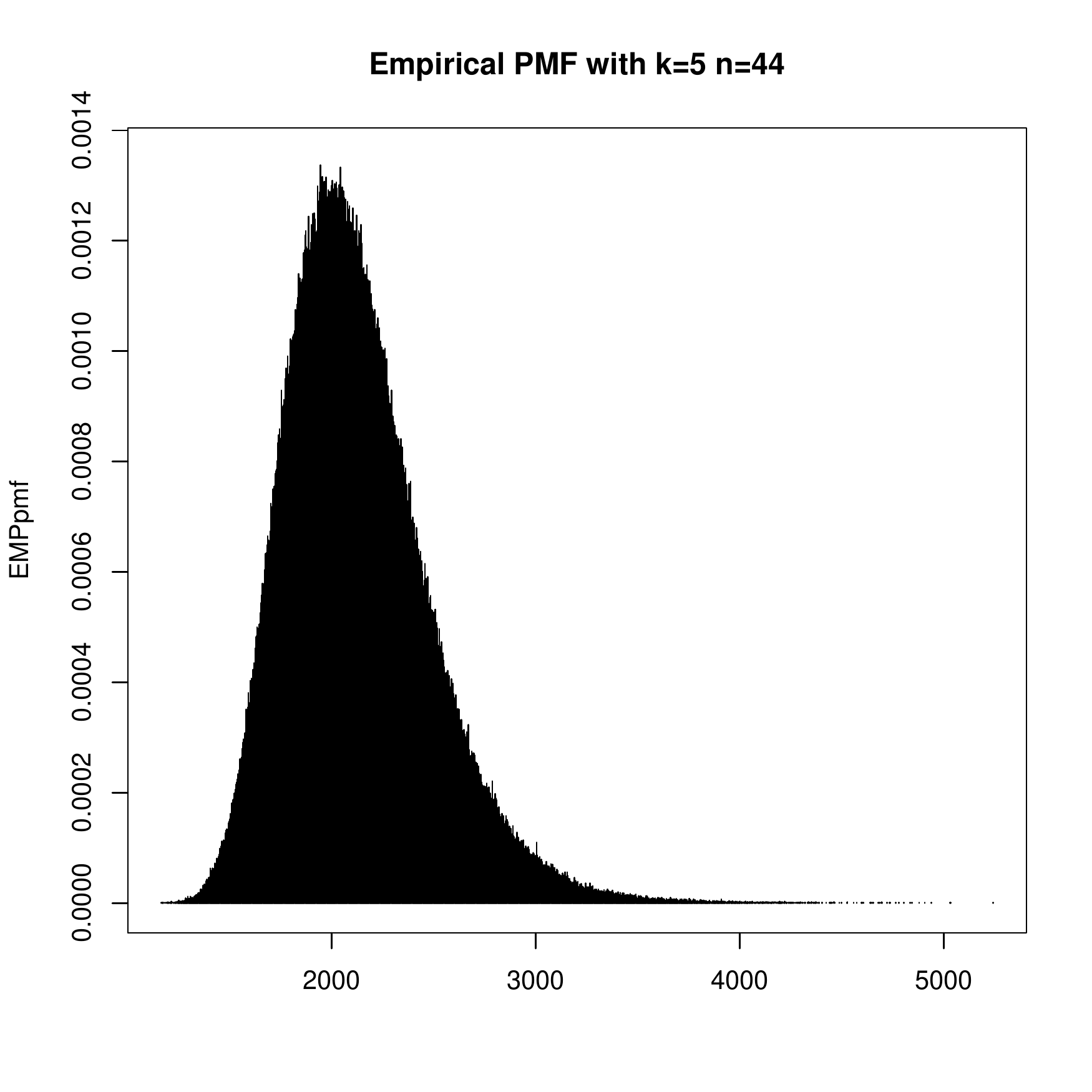} \vspace{-.4em}\\
		Sample size = 1M & Sample size = 1.1M & Sample size = 1M\\ 
		\includegraphics[scale=.21]{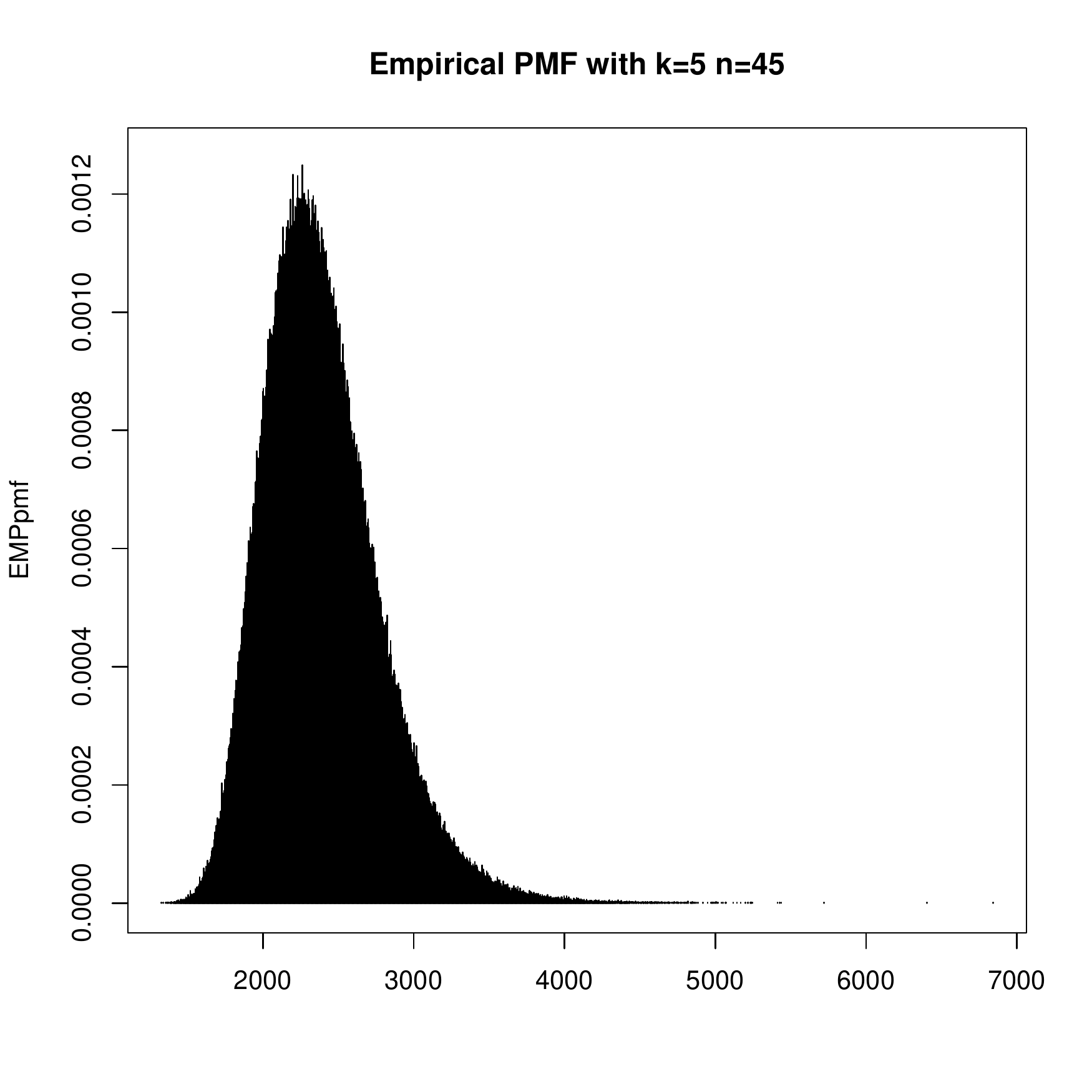} &  
		\includegraphics[scale=.21]{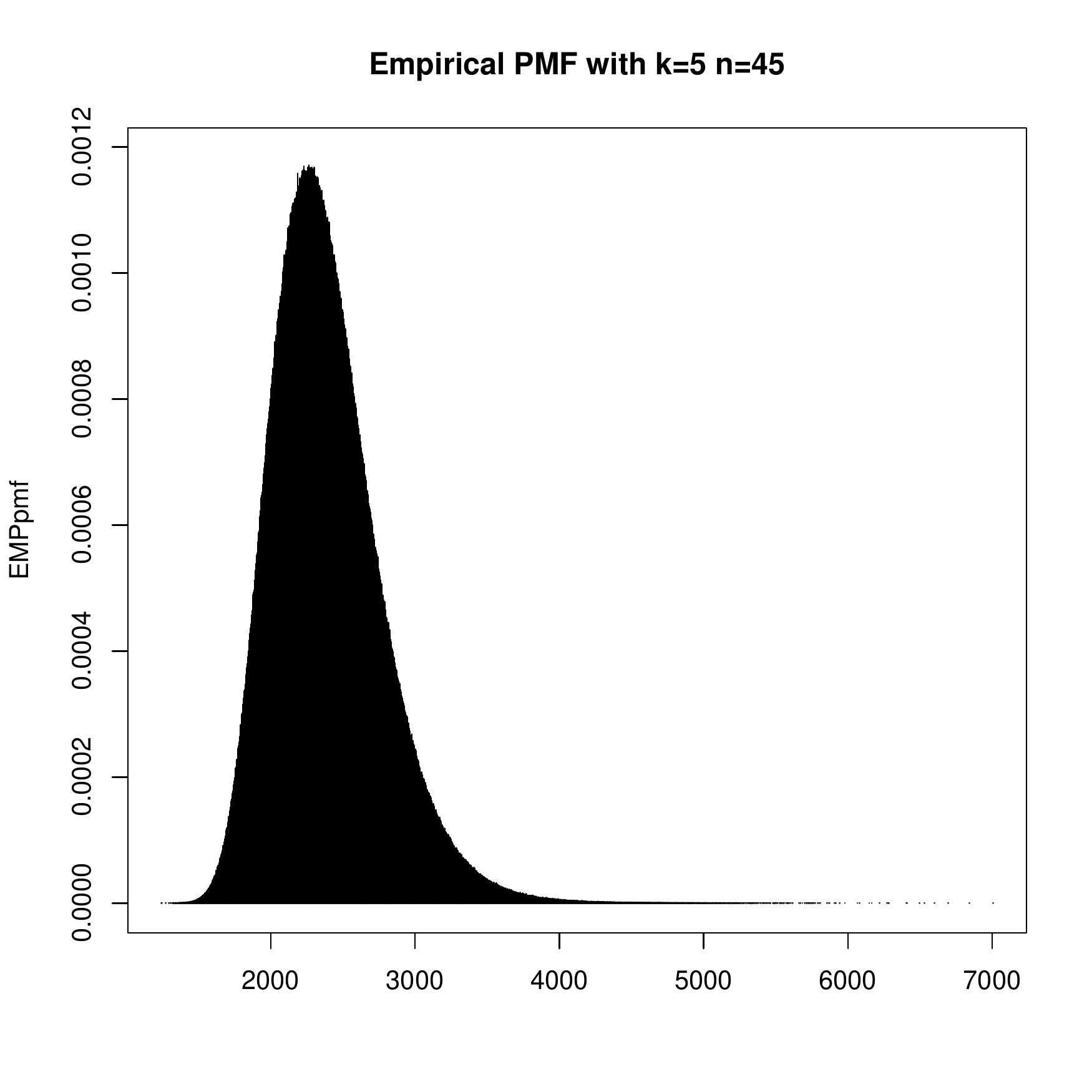} &
		\includegraphics[scale=.21]{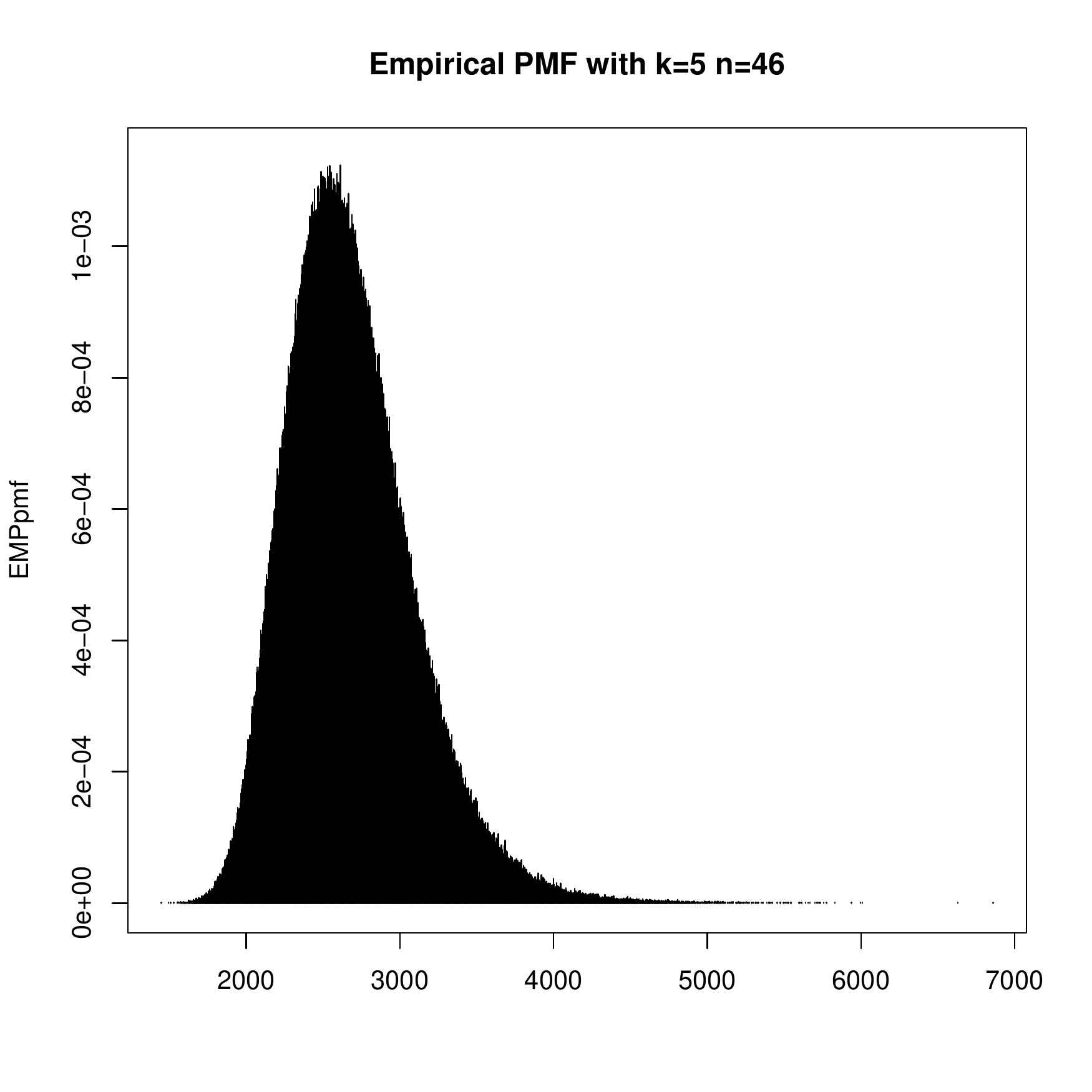}\vspace{-.4em}\\
		Sample size = 1M & Sample size = 17.4M & Sample size = 1.1M\\ 
	\end{tabular}\normalsize
	\caption{Empirical probability mass functions (pmfs) for various scenarios.}
\end{center}
\end{figure}

\section{Fitting the Empirical Probability Mass Function}
\noindent
We can view the random variable $X_k$ as a sum of indicator random variables
$Y_i$, where $Y_i = 1$ if  the $i^{\mathrm{th}}$ $K_k$ is monochromatic and
$Y_i=0$ otherwise.
Since $n$ is typically much larger than $k$, most pairs of $K_k$'s are
independent.  Hence, we can view $X_k$ as the sum of (somewhat) weakly
dependent indicator random variables, each of which have a small probability
of being $1$.  By weak dependence, we mean that the probability of
two randomly chosen $K_k$'s are dependent is near $0$.  To see this,
note that the probability that two such subgraphs are dependent requires them
to share at least $2$ vertices, so that this probability is
$$\frac{{k \choose 2}{n-2 \choose k-2}}{{n \choose k}} \approx \frac{k^4}{2n^2}.$$  
Noting that $n$ is typically much larger than $k^2$, we see that this probability is quite low.

Now, as $k$ tends to infinity, the proportion of dependent pairs of
subgraphs goes to 0.  Hence, asymptotically, we can view the set of all subgraphs
as almost entirely independent.  In the situation where the subgraphs
are completely independent, since the probability that any
given subgraph is monochromatic is small, through the Poisson process we get a Poisson distribution and the
result in \cite{G}; see \cite{J} for a general theorem about when we can
have a limiting Poisson distribution with weak overall dependence.

However, when we investigate fixed values of (small) $k$, as we
can see from Figure 2, the Poisson distribution is not a good fit.
This is to be expected since the Poisson distribution has a variance
equal to its expectation, while
we  know that $\mathbb{E}(X_k) \neq$ Var$(X_k)$ for $k \geq 4$ (see Lemma \ref{ExpVar} below). For the $k=3$ case,
we do have $\mathbb{E}(X_k) \approx $ Var$(X_k)$ for large $n$; see Table 2 in the next section.  We can also see from Figure 2 that the Poisson distribution appears under-dispersed (while Lemma \ref{ExpVar} below proves
this).
More fundamentally, for fixed values of $k$,
the dependence between some of the subgraphs is not
accounted for with Poisson modeling. 

\begin{lemma} For any $k \geq 4$, we have 
$\displaystyle \lim_{n \rightarrow \infty} \frac{\mathbb{E}(X_k(n))}{\mathrm{Var}(X_k(n))} =0.$
\label{ExpVar}
\end{lemma}

\begin{proof}
Lemma 3.5 in \cite{JLR} gives the asymptotic order of Var$(X_k)$: Define 
$\Phi(X_k) = \min (\mathbb{E}(X_H))$, where the minimum is taken over all
nontrivial subgraphs $H$ of $K_k$.  Then $\mbox{Var}(X_k) \asymp \frac{1}{2} \frac{\mathbb{E}^2(X_k)}{\Phi(X_k)}.$
Taking $H=K_3$ (where $K_2$ is the degenerate case) we see that 
$
\mathrm{Var}(X_k) \asymp \frac{1}{2} \frac{{n \choose k}^2 / 2^{{2k \choose 2}-2}}{{n \choose 3}/2^2}
= \Omega\left({n \choose k} \frac{n^{k-3}}{2^{2{k \choose 2}-2}}\right),
$ which agrees with the expressions given in Table 2 (in the next
section).
We know that 
$\mathbb{E}(X_k) = \frac{{n \choose k}}{2^{{k \choose 2}-1}},$ and by comparison with the 
above expression we see that the lemma's statement holds.
\end{proof}

To further illustrate the point, in Figure 2 we present overlays of the best-fitting (defined in the next paragraph) Poisson distributions over the empirical
pmfs presented in Figure 1. As you can see, the Poisson distribution is clearly not a good fit for small values of $k$.

Our measure of
best-fitting is via the $\ell_1$-distance between two
 probability mass functions $f(k)$ and $g(k)$:
$
\sum_{k \geq 0} |f(k)-g(k)|.
$

\begin{figure}[h!]\tiny  
	\begin{center} \hspace*{-0.1in}
		\begin{tabular}{ccc}
			\includegraphics[scale=.21]{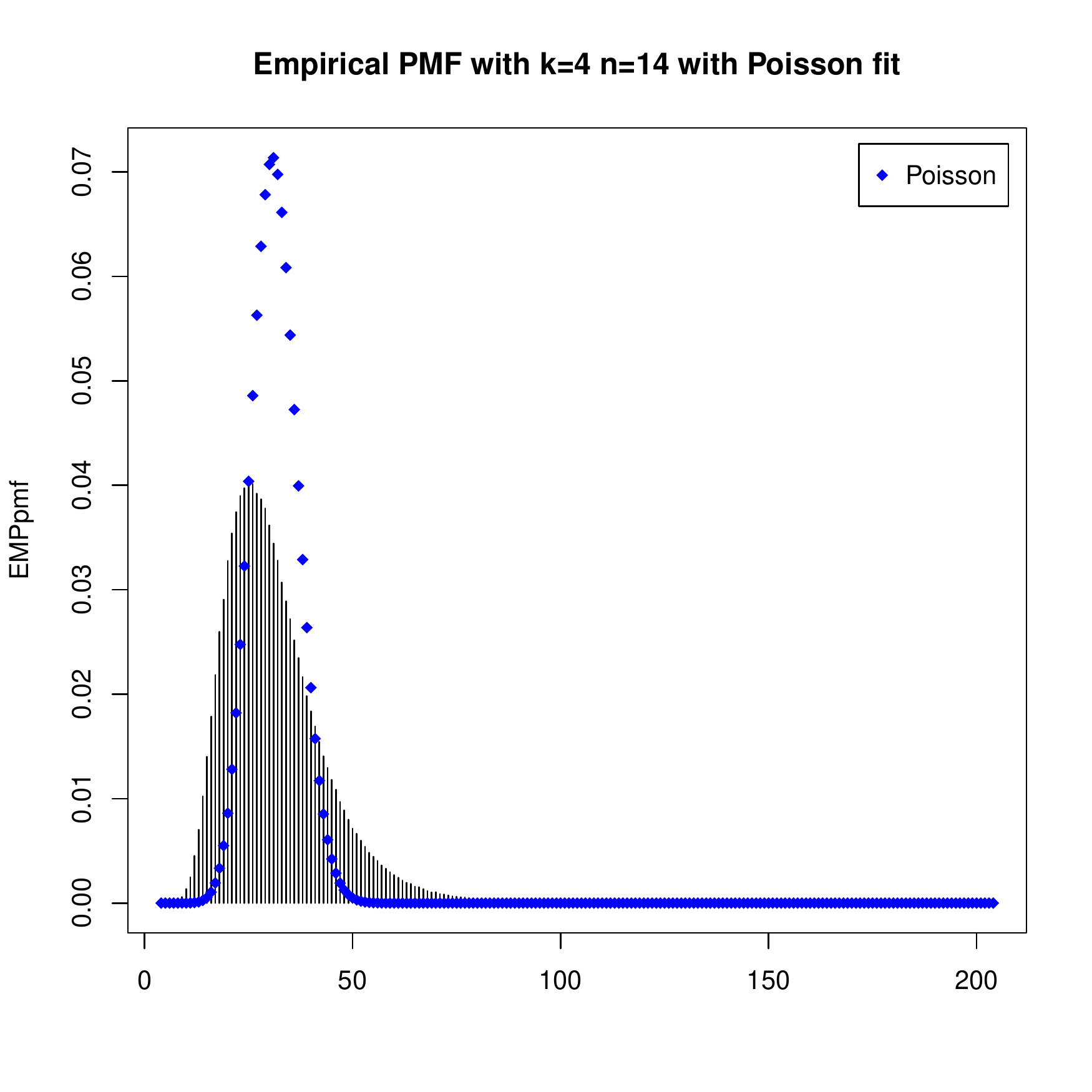} &  
			\includegraphics[scale=.21]{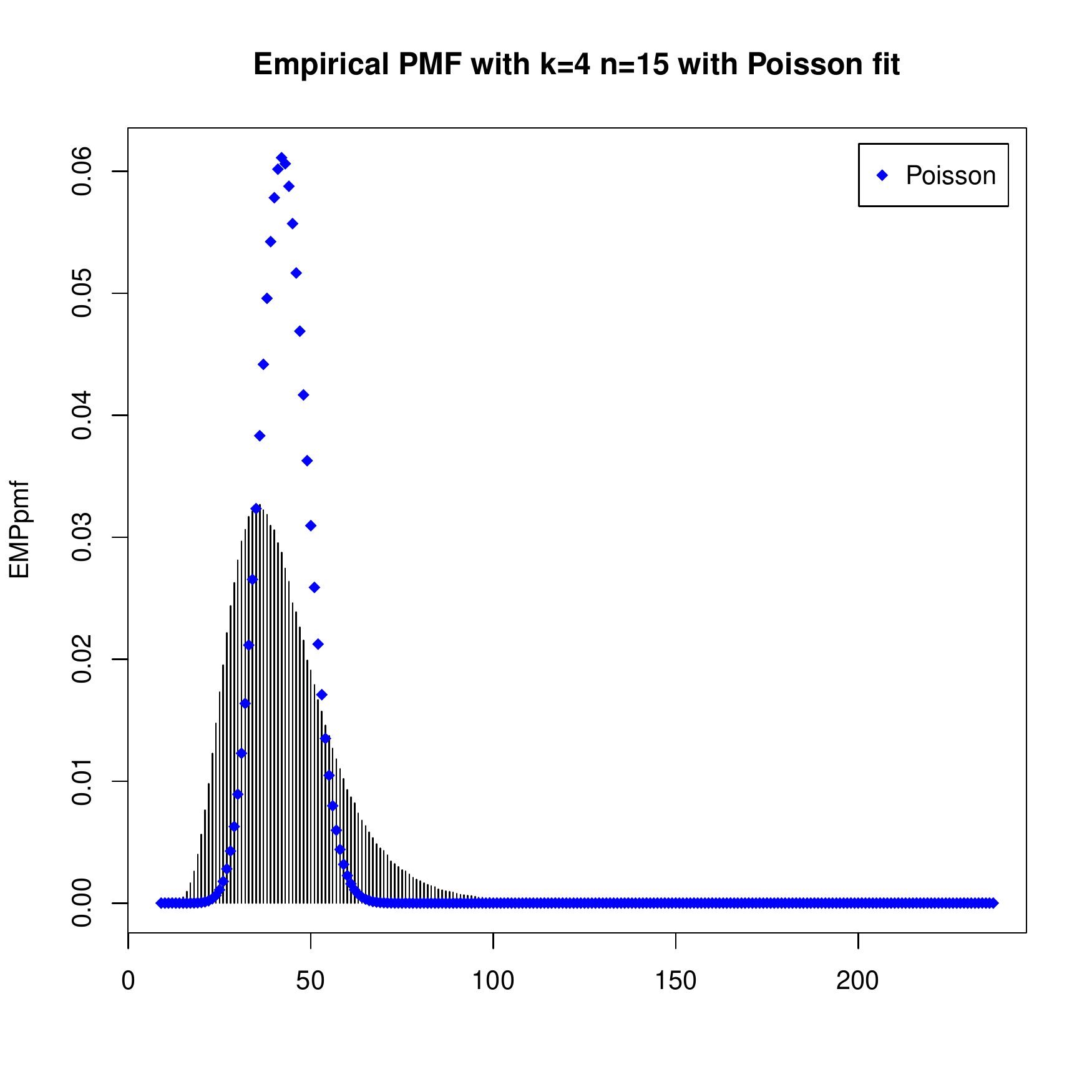} &   
			\includegraphics[scale=.21]{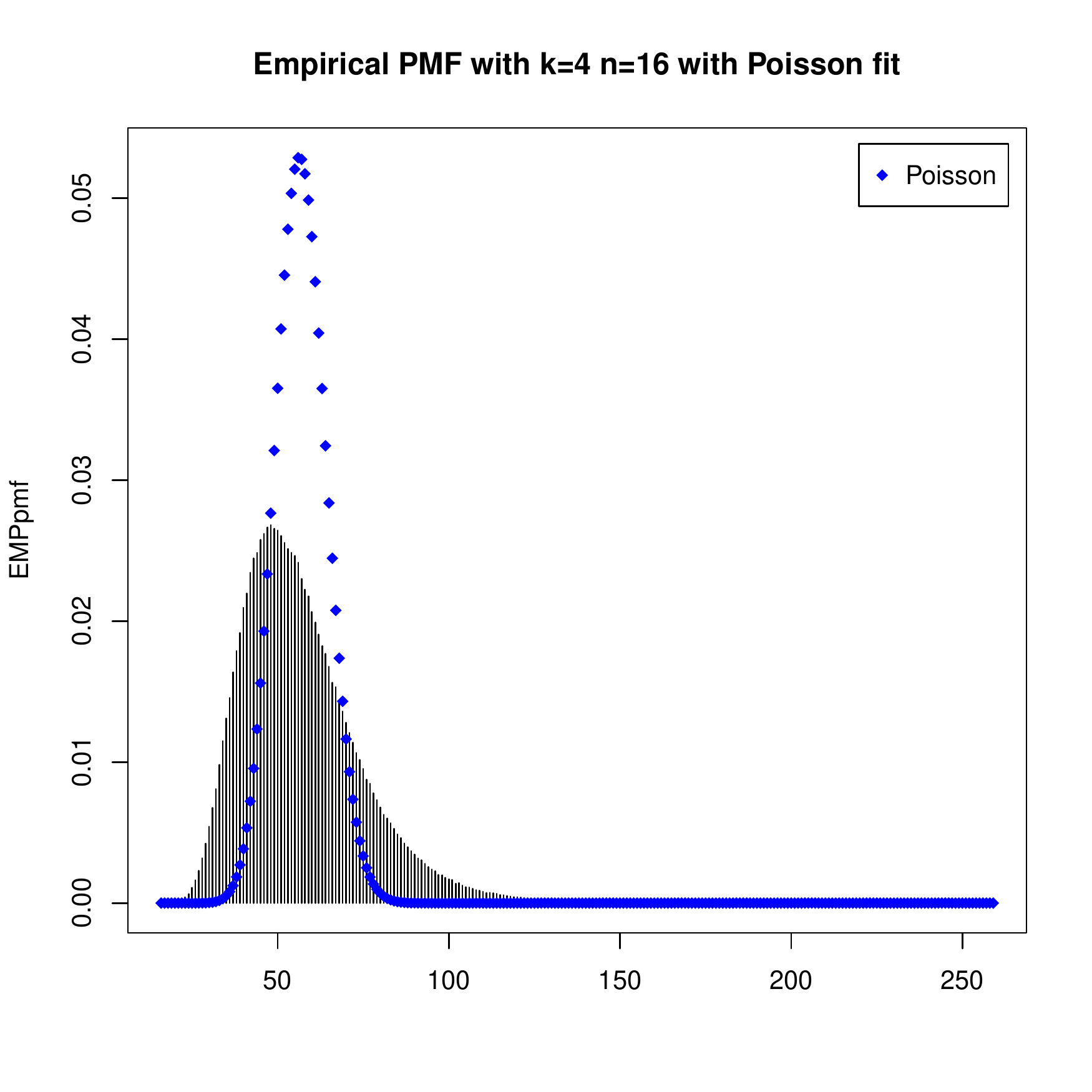} \vspace{-.4em}\\
			Sample size = 1M & Sample size = 1M & Sample size = 1M\\ 
			\includegraphics[scale=.21]{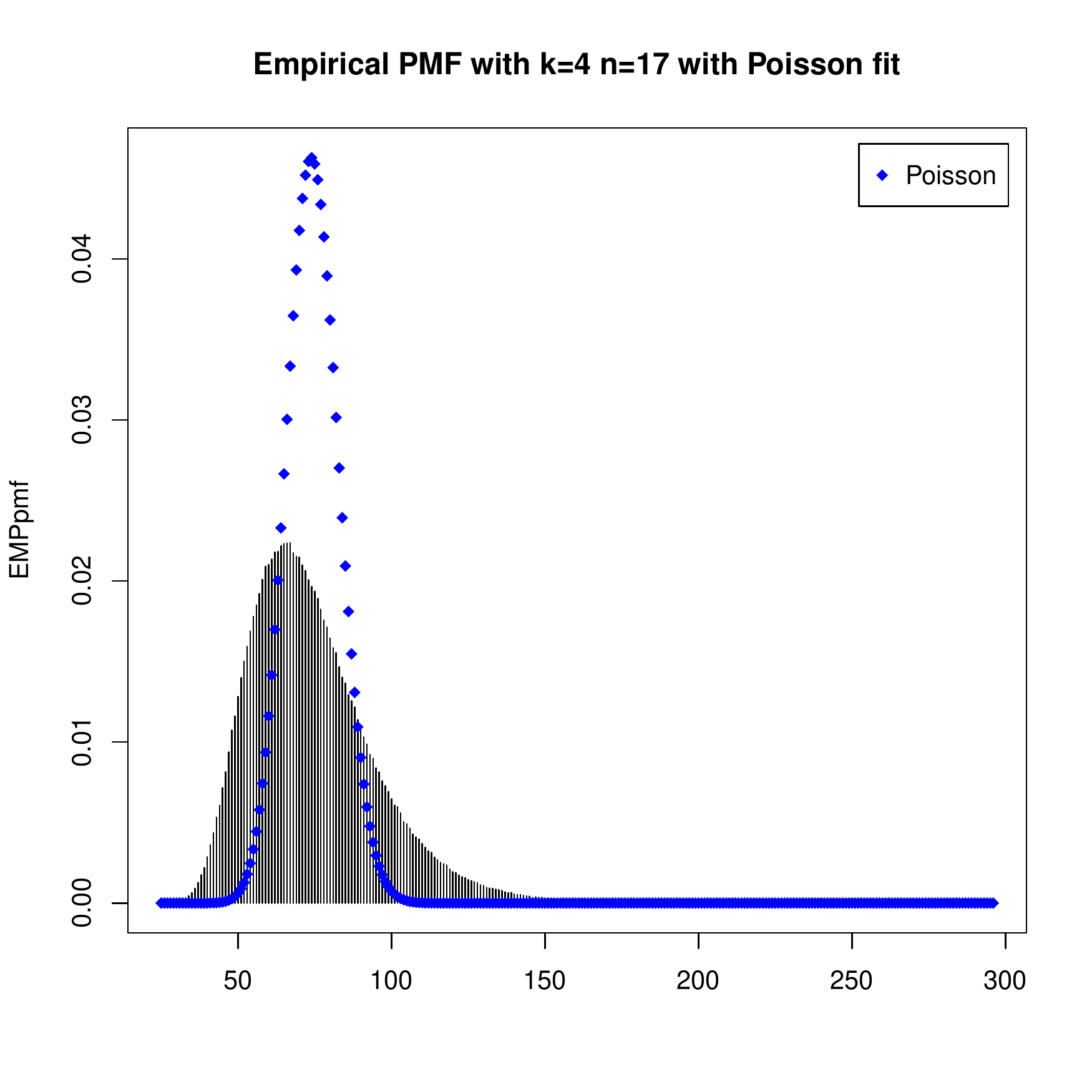} &  
			\includegraphics[scale=.21]{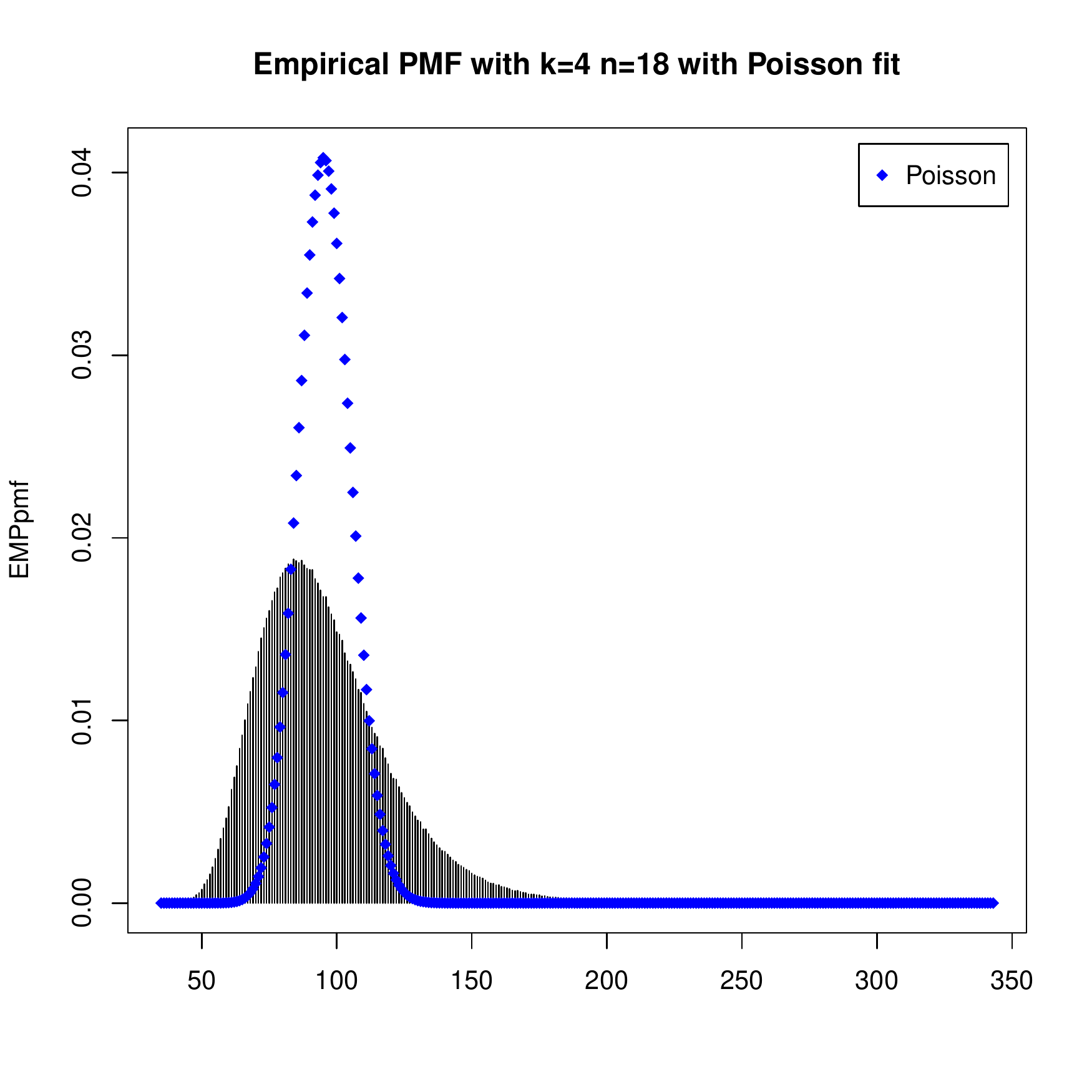} & 
			\includegraphics[scale=.21]{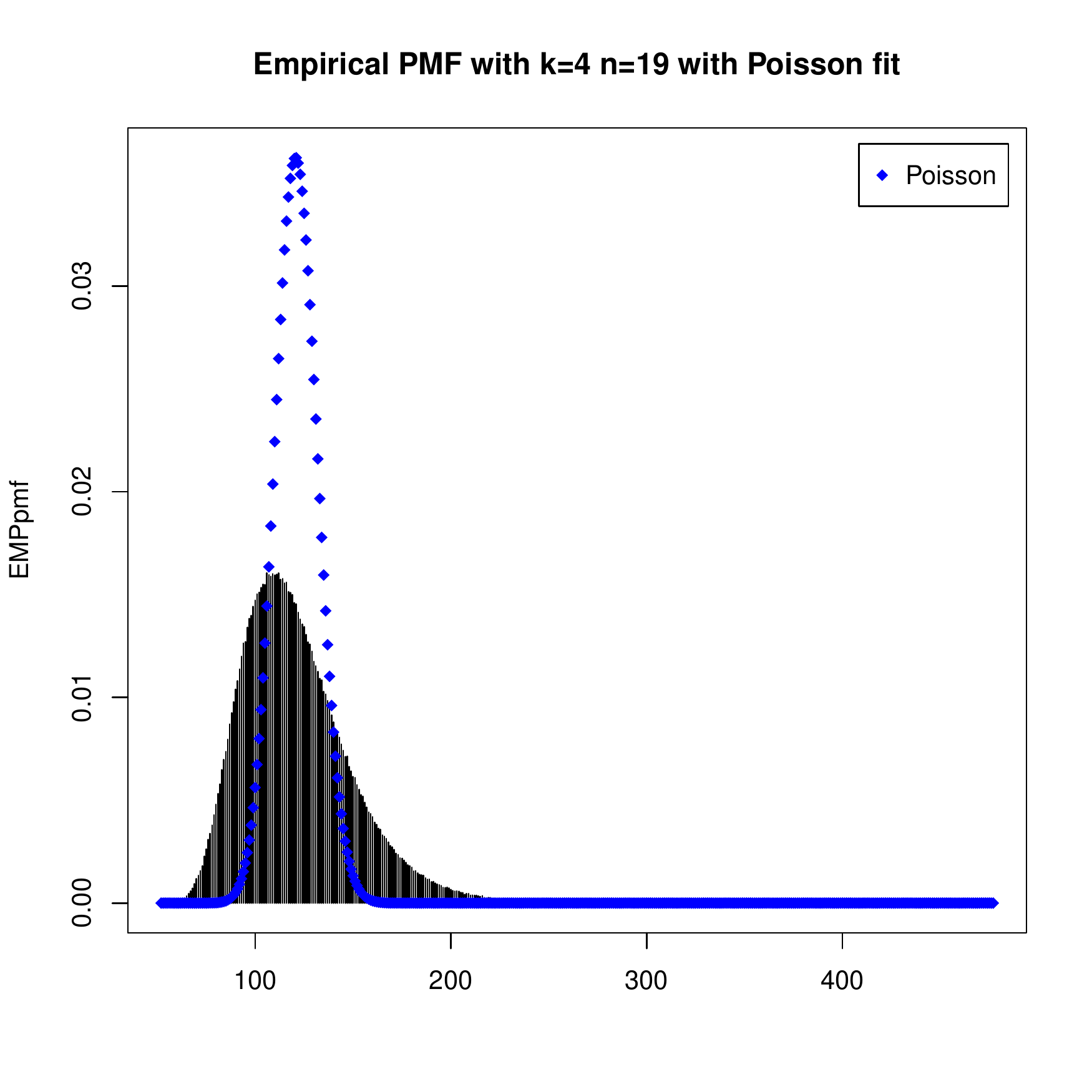} \vspace{-.4em}\\
			Sample size = 1M & Sample size = 1M & Sample size = 1M\\ 
			\includegraphics[scale=.21]{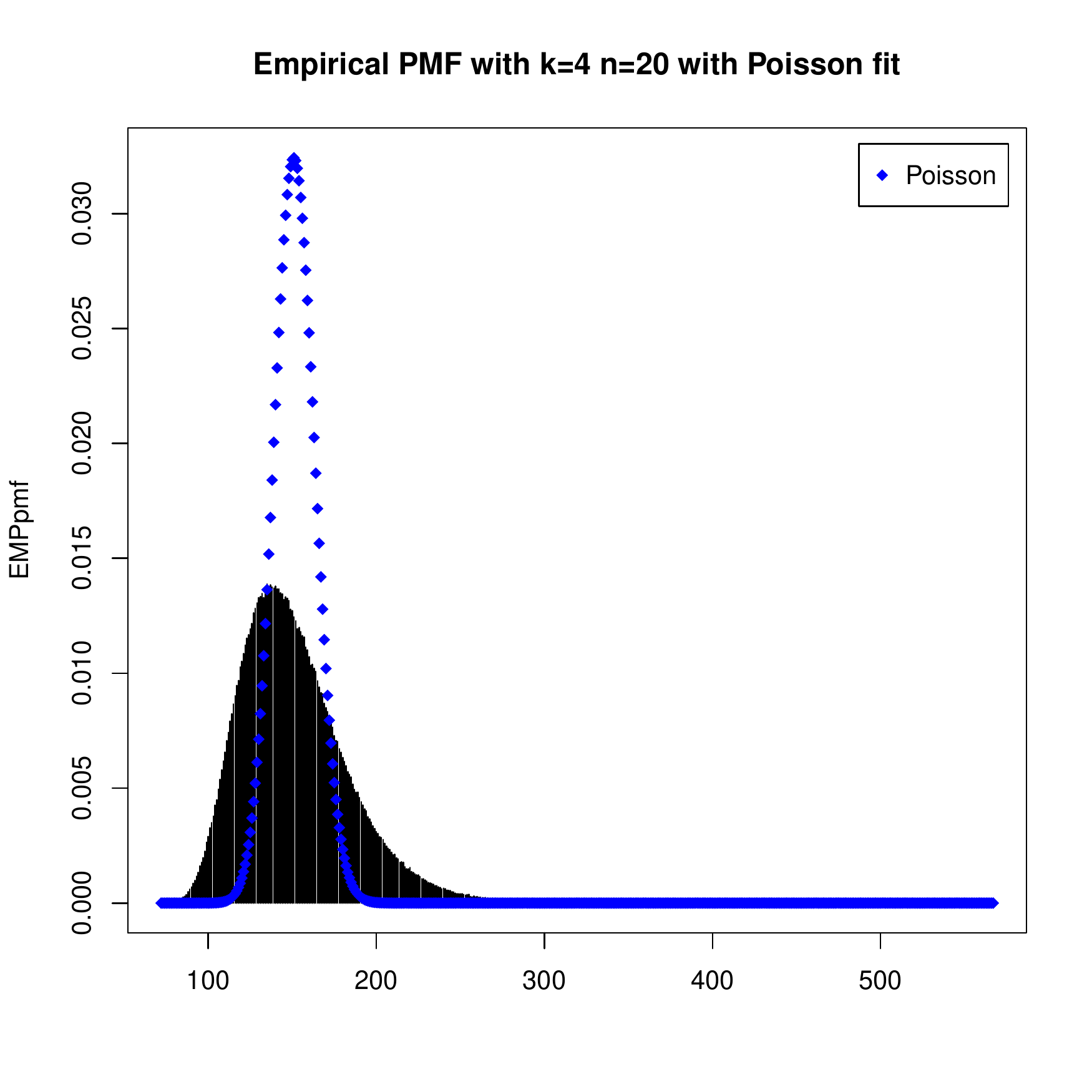} &  
			\includegraphics[scale=.21]{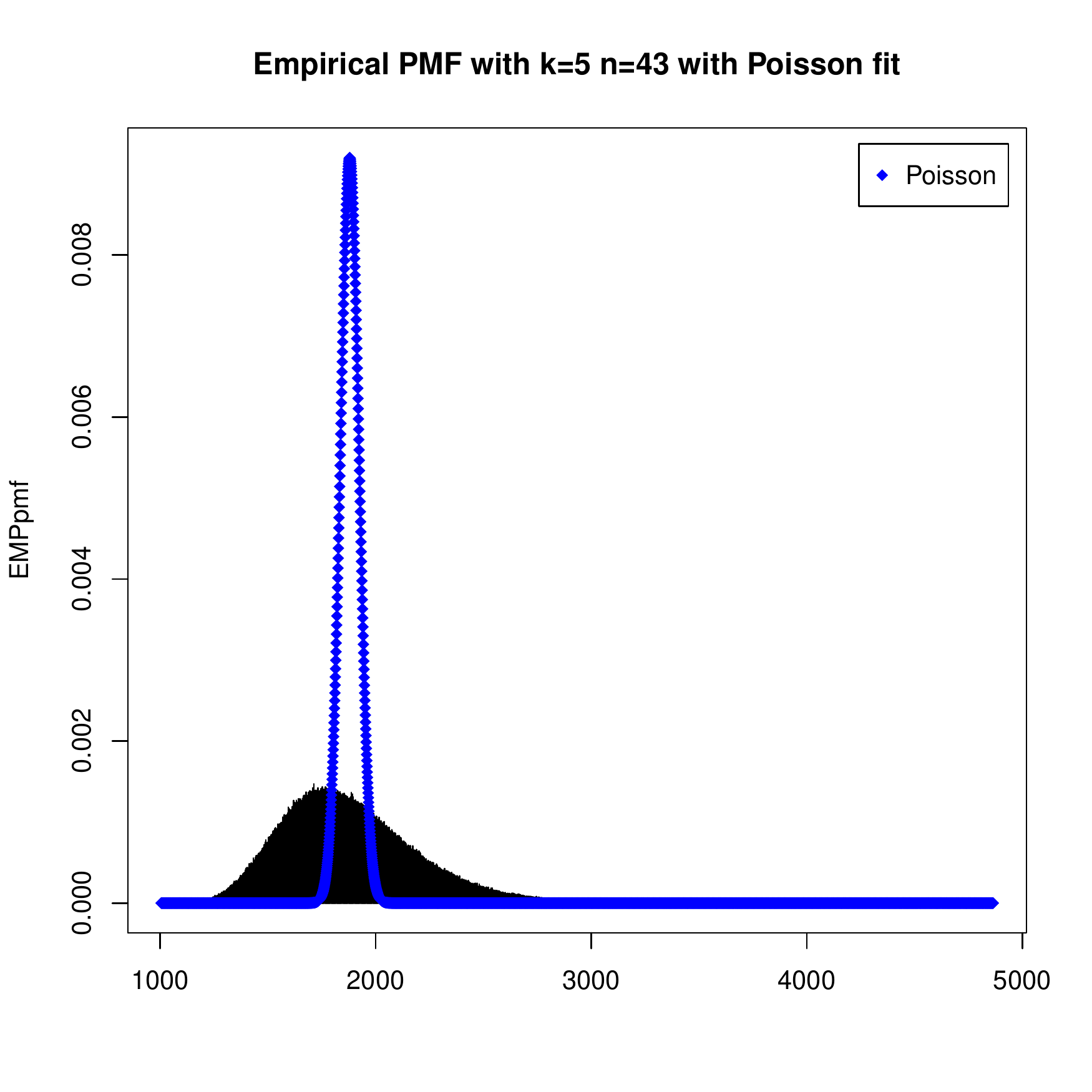} &   
			\includegraphics[scale=.21]{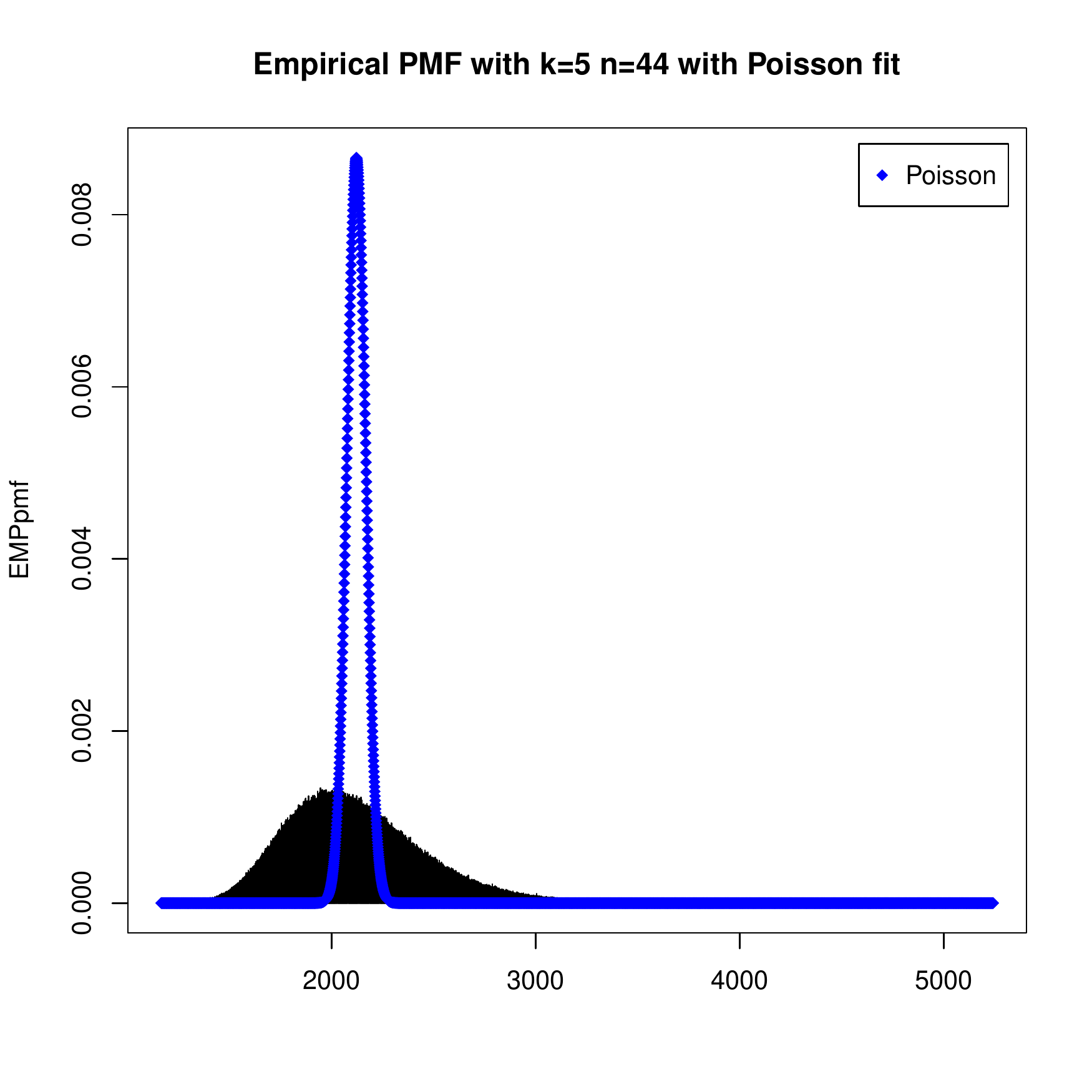} \vspace{-.4em}\\
			Sample size = 1M & Sample size = 1.1M & Sample size = 1M\\ 
			\includegraphics[scale=.21]{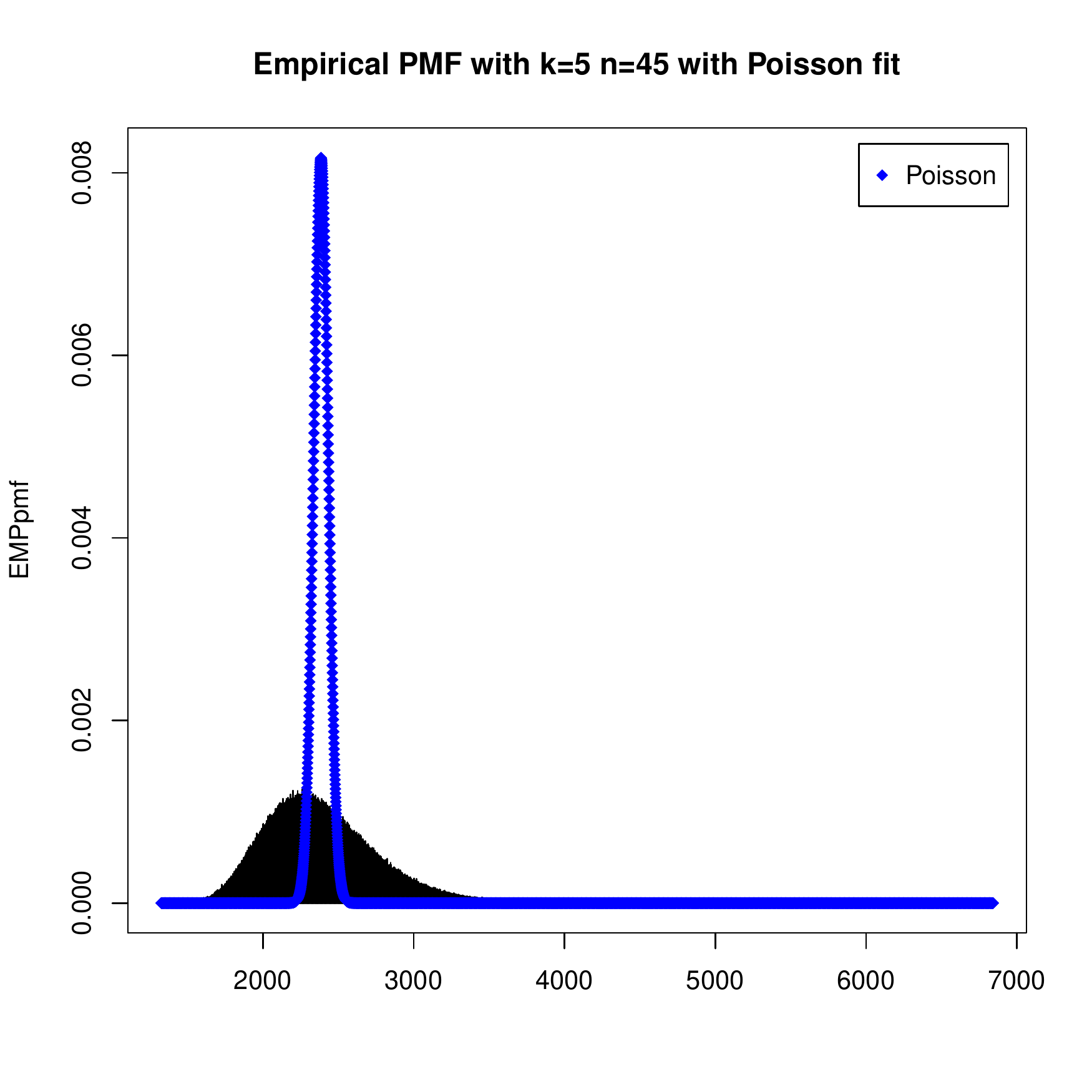} &  
			\includegraphics[scale=.21]{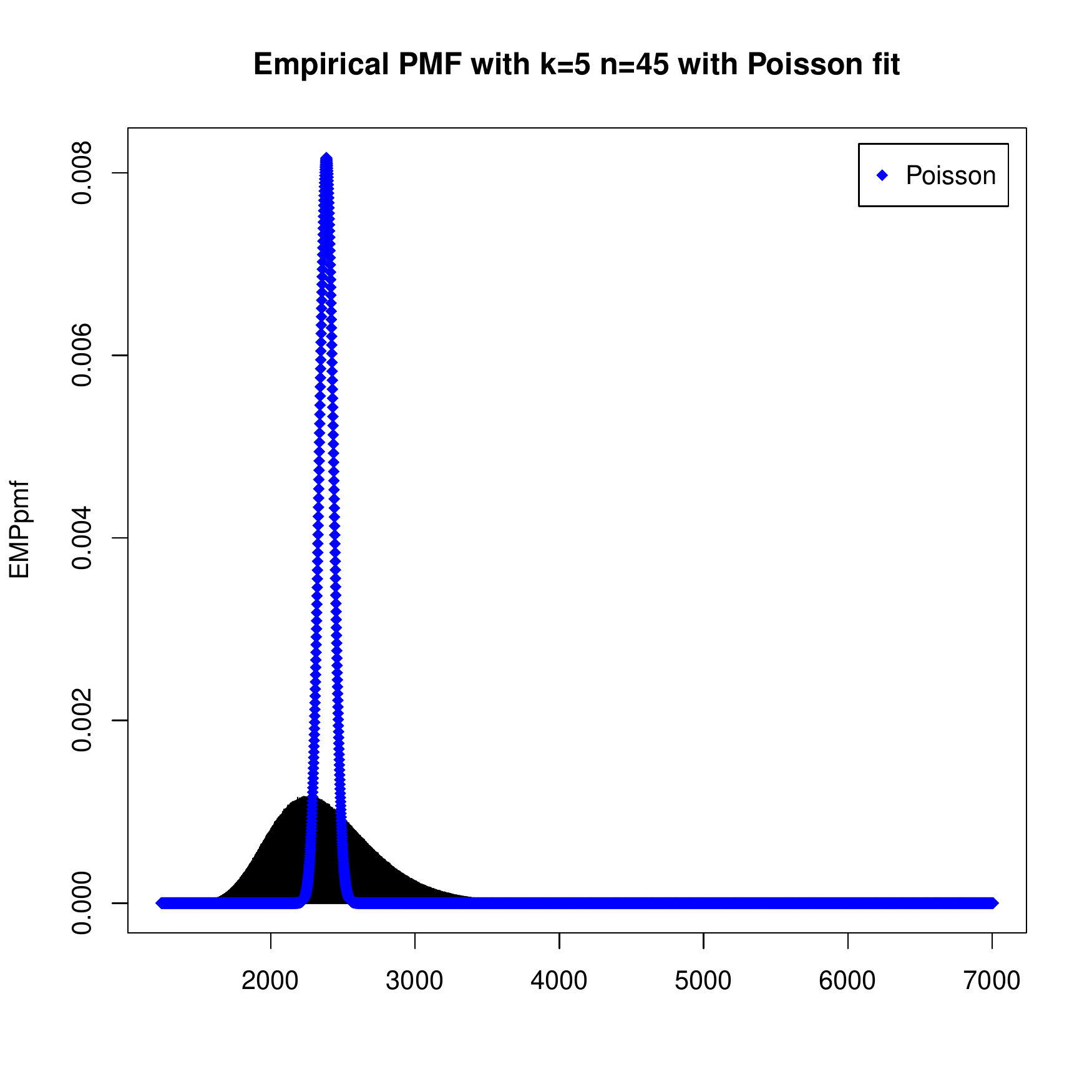} &
			\includegraphics[scale=.21]{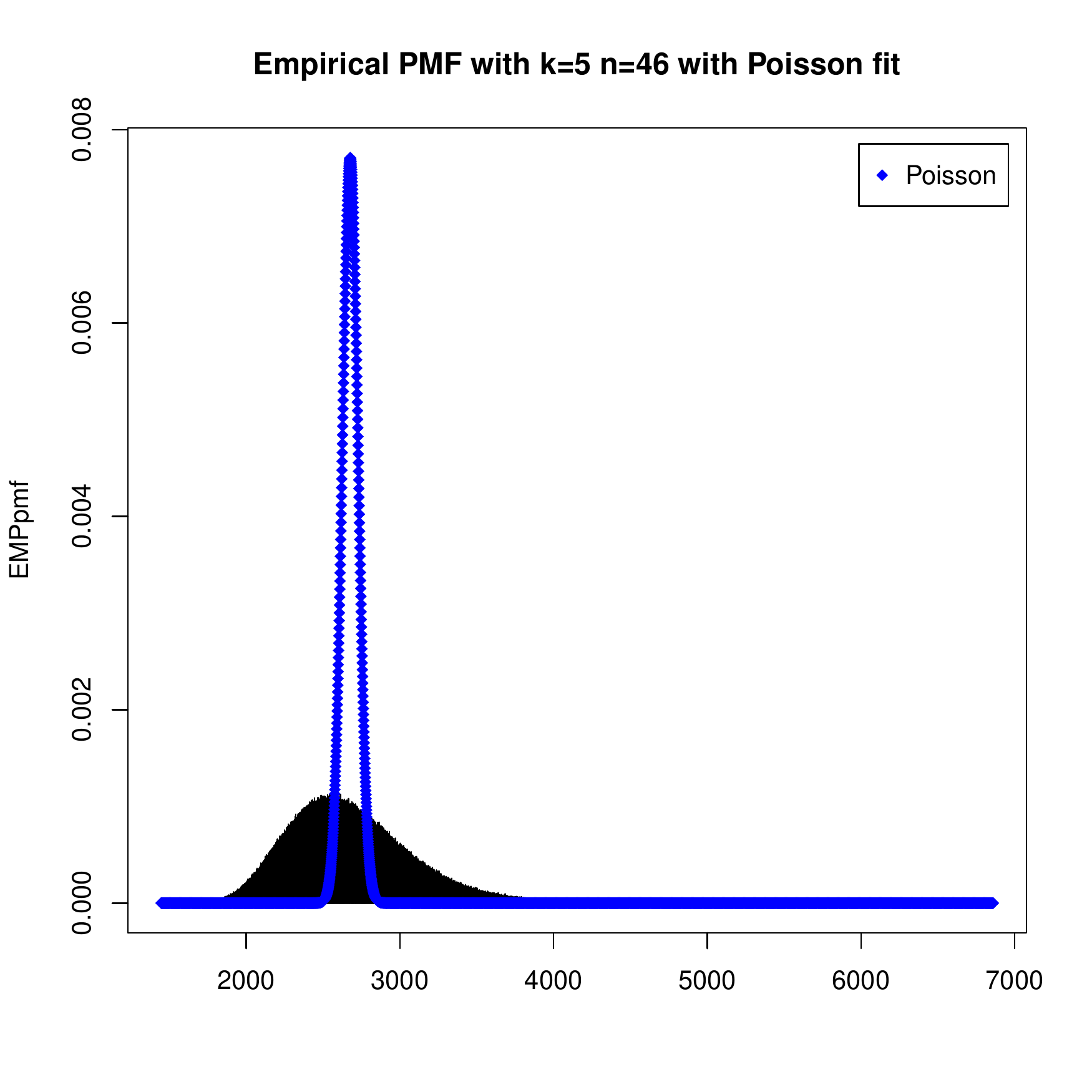} \vspace{-.4em}\\
			Sample size = 1M & Sample size = 17.4M & Sample size = 1.1M\\   
		\end{tabular} 
		\caption{Empirical pmfs for various scenarios with 
best-fit Poisson overlay}
	\end{center}
\end{figure}

To address the dependence and under-dispersion, we turn to mixed-Poisson processes, i.e.,
Poisson processes with a parameter $L$ that is
a random variable (as opposed to being fixed).  The result
will be a compound Poisson distribution. We need to maintain
the asymptotic Poisson nature and so, heuristically, having
the parameter contain a fixed portion $\lambda$ and a random portion $G$
addresses this.
Considering $L=\lambda+G$ gives our mixed-Poisson process a ``Poisson part" $\lambda$ (for the
independent subgraphs) and a
``local dependence corrector" $G$, which captures
unknown dependence.  While this addresses the dependence
issue, it also allows us to correct the under-dispersion of the
Poisson approximation.  This is to be expected since the under-dispersion
is linked with the failure to account for some dependence between
events.

A commonly used choice for mixing with Poisson is the Gamma distribution.
However, even though there is ample empirical evidence for
the use of Gamma as a mixing function, there is no real theoretic support for the choice of
Gamma \cite{V}.  It seems that the Gamma distribution is used because
of its flexibility and calculability when mixed with Poisson.
However, we can turn to the P\'olya-Eggenberger urn scheme
for some motivation.

As noted in \cite{M}, the P\'olya-Eggenberger urn models have been used
as contagion models.  In terms of our mixed-Poisson process, a contagion
model would have the property that the probabilities of future events occurring increase as
events occur (see \cite{Al}).  For us, this would model the change
in probabilities when dealing with dependent subgraphs.  Connecting this to the P\'olya-Eggenberger model, we find
that the negative binomial distribution is one of the limiting
distributions.  Since a Poisson distribution with random parameter
being Gamma results in a negative binomial distribution, we have some
(albeit, tangential)
relationship to the P\'olya-Eggenberger model as a contagion model.

Returning to our mixed-Poisson process, we can now describe
$L=\lambda+G$ as having a fixed ``Poisson part" $\lambda$ (for the
independent subgraphs) and a
``contagion driver" $G$, which helps to model the weak dependence.
This mixing process gives rise to the
following
convolution pmf called the Delaporte distribution.

\begin{definition}  A random variable $D = D(\lambda,\alpha,\beta)$ is
called a {\it Delaporte random variable} if it has probability mass
function
$$
\mathbb{P}(D=j) = \sum_{i=0}^j
\frac{\Gamma(\alpha+i)}{\Gamma(\alpha) i!} \!\left(\frac{\beta}{1+\beta}\right)^{\! i}
\!\left(\frac{1}{1+\beta}\right)^{\! \alpha} \frac{\lambda^{j-i}e^{-\lambda}}{(j-i)!}.
$$
Furthermore, $\mu=\mathbb{E}(D) = \lambda + \alpha\beta$, Var$(D) = \mathbb{E}((D-\mu)^2) = \lambda + \alpha\beta(1+\beta)$, and
$\mathbb{E}((D-\mu)^3) = \lambda + \alpha\beta(1+3\beta+2\beta^2)$. \label{DelDef}
\end{definition}

\noindent
{\bf Remark.} In the Delaporte pmf above, $\lambda$ is the parameter for the Poisson part while
$\alpha$ and $\beta$ are parameters for the Gamma part of our 
$L=\lambda + G$ model for the Poisson process rate, which leads to
a negative binomial distribution with parameters $\alpha$ and $\frac{\beta}{1+\beta}$.

\vskip 5pt

In Figure 3, we present the same empirical pmfs as in Figures 1 and 2 along with an overlay of the best-fit
Delaporte distribution.

\begin{figure}[h!] \tiny
	\begin{center} \hspace*{-0.1in}
		\begin{tabular}{ccc}
			\includegraphics[scale=.21]{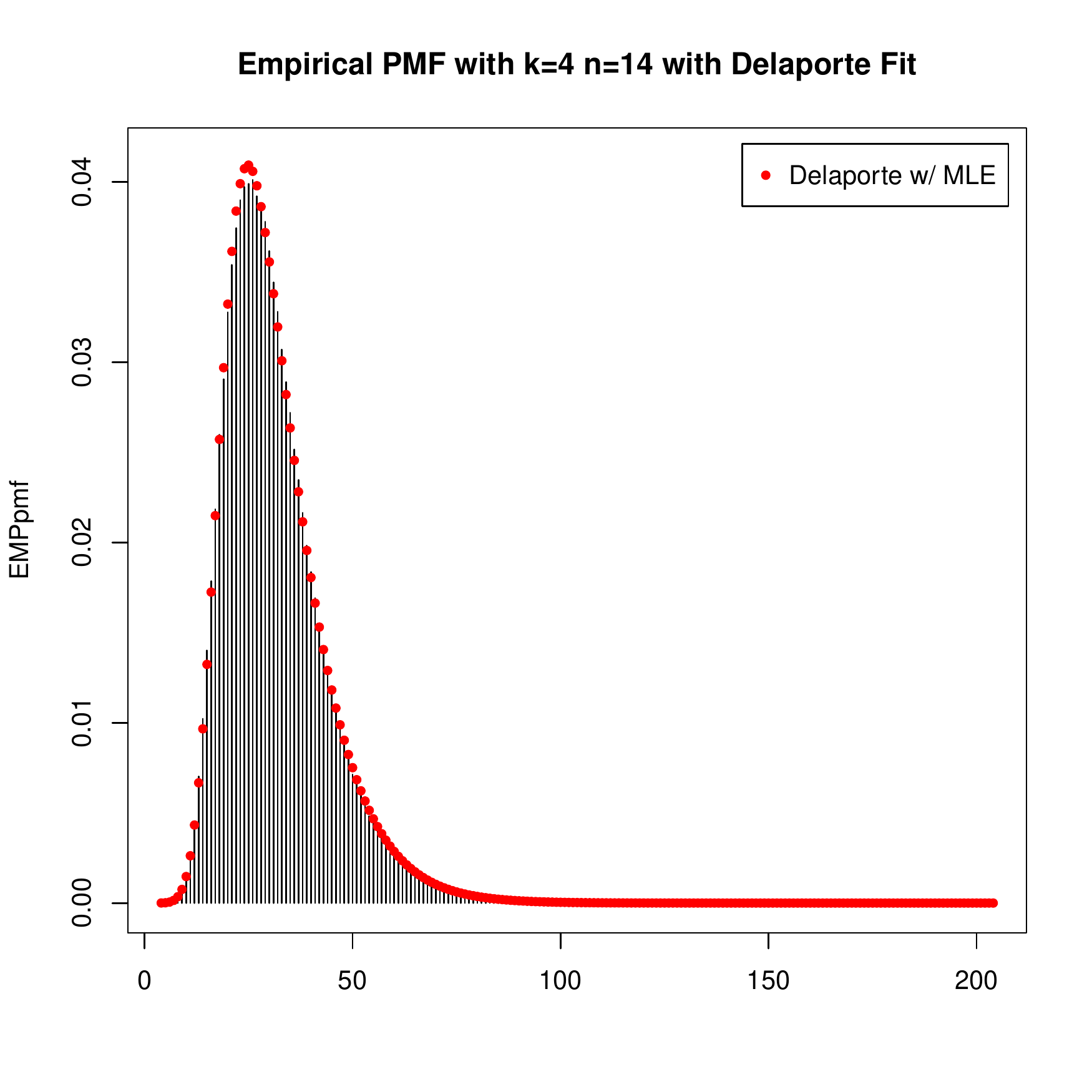} &  
			\includegraphics[scale=.21]{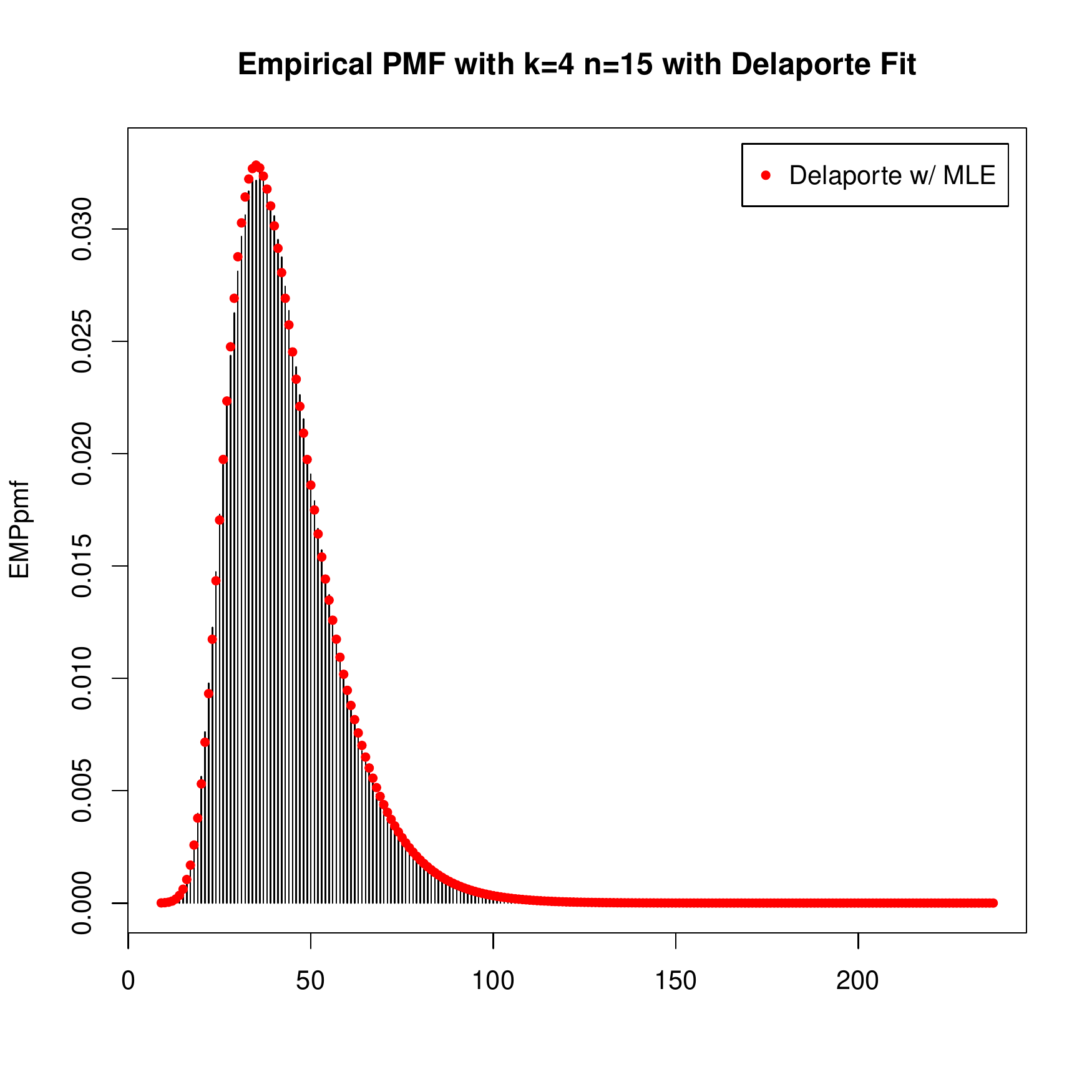} &   
			\includegraphics[scale=.21]{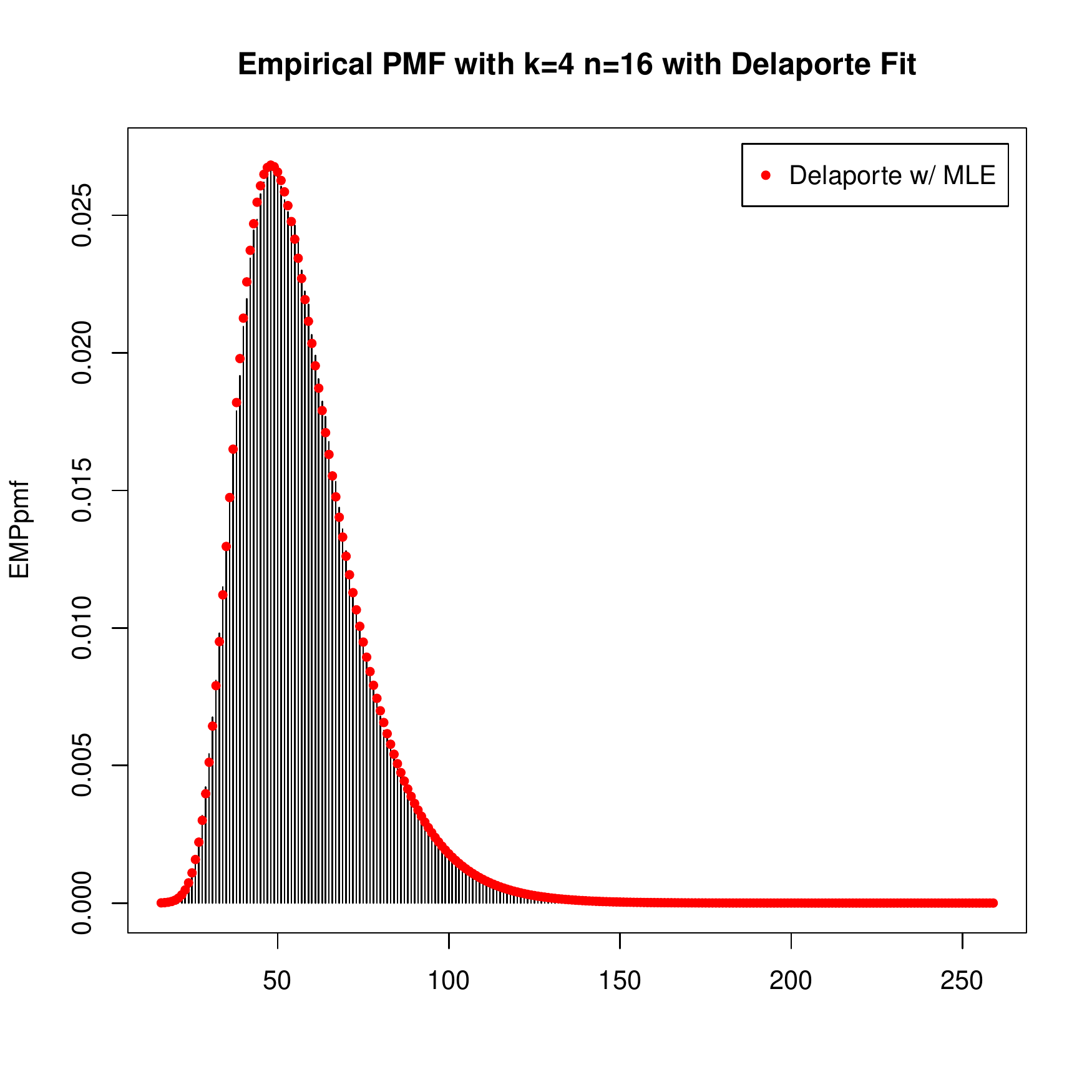} \vspace{-.4em}\\
			Sample size = 1M & Sample size = 1M & Sample size = 1M\\ 
			\includegraphics[scale=.21]{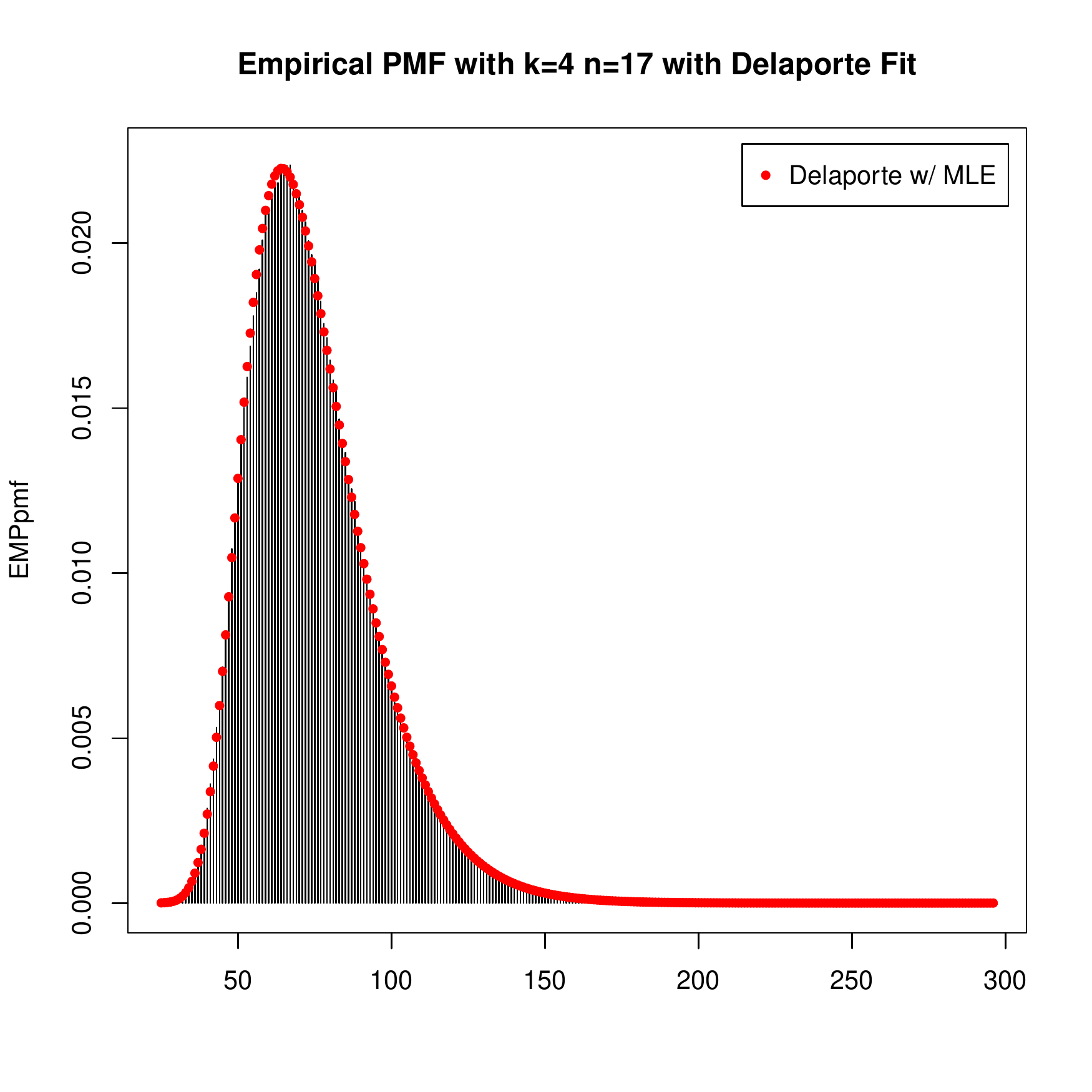} &  
			\includegraphics[scale=.21]{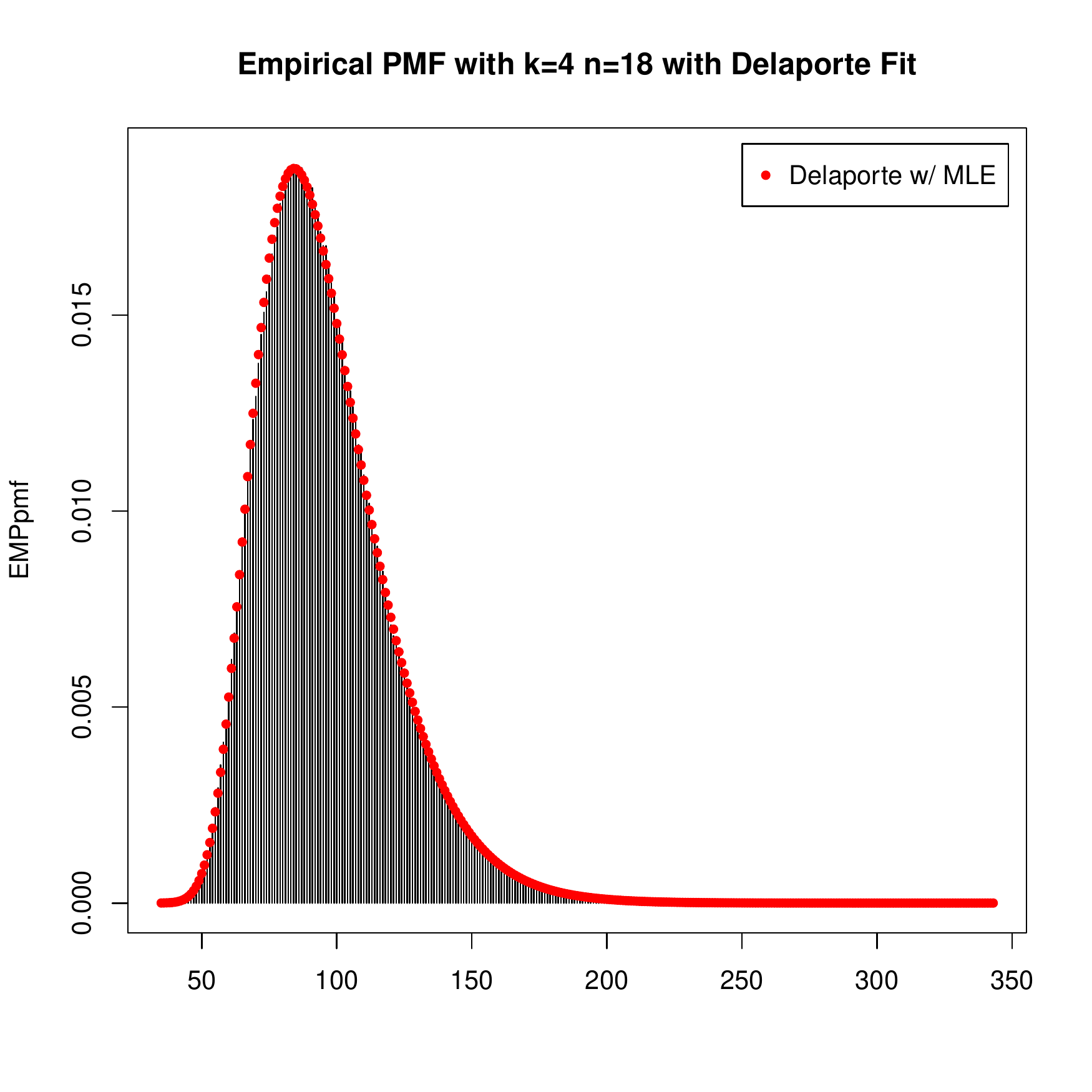} & 
			\includegraphics[scale=.21]{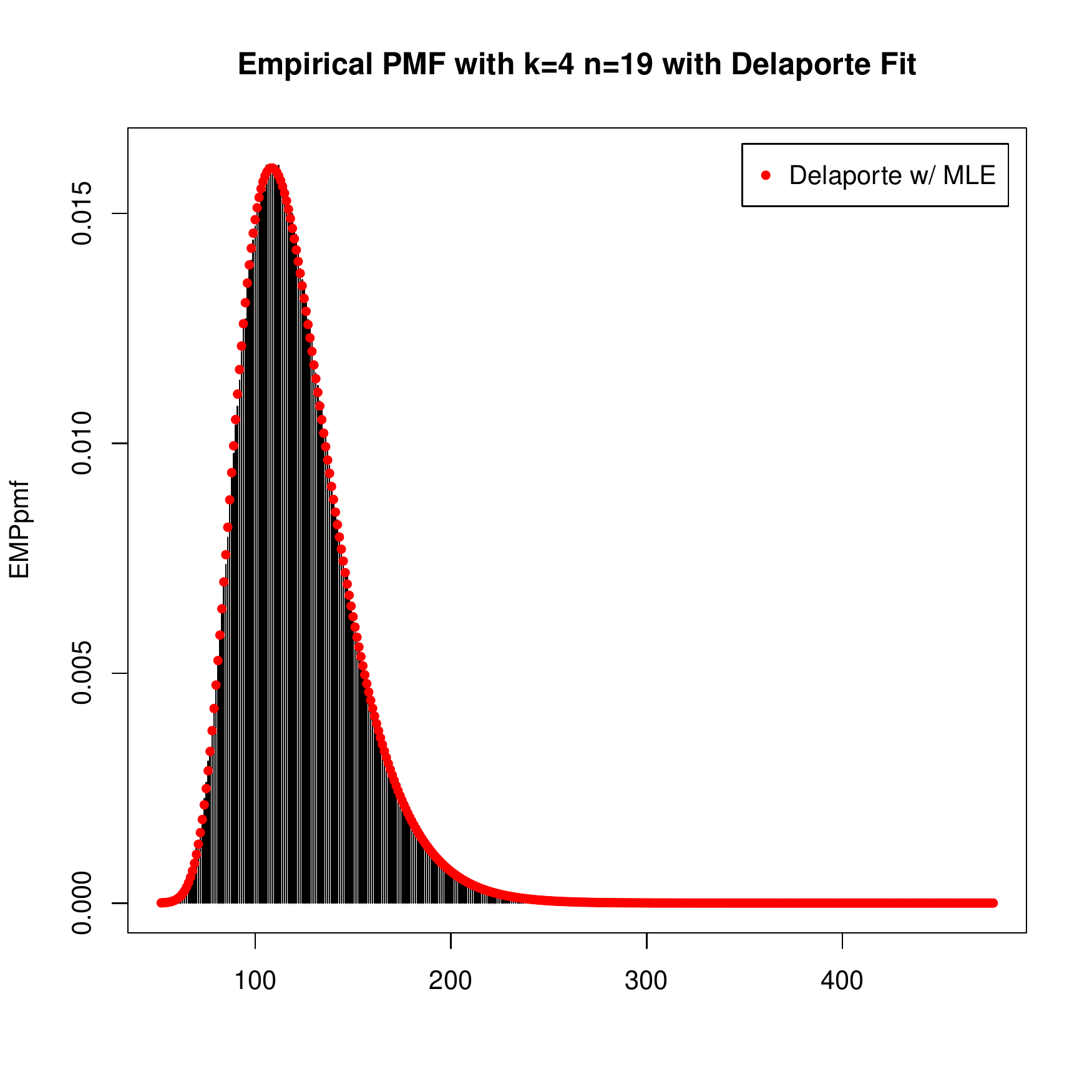} \vspace{-.4em}\\
			Sample size = 1M & Sample size = 1M & Sample size = 1M\\ 
			\includegraphics[scale=.21]{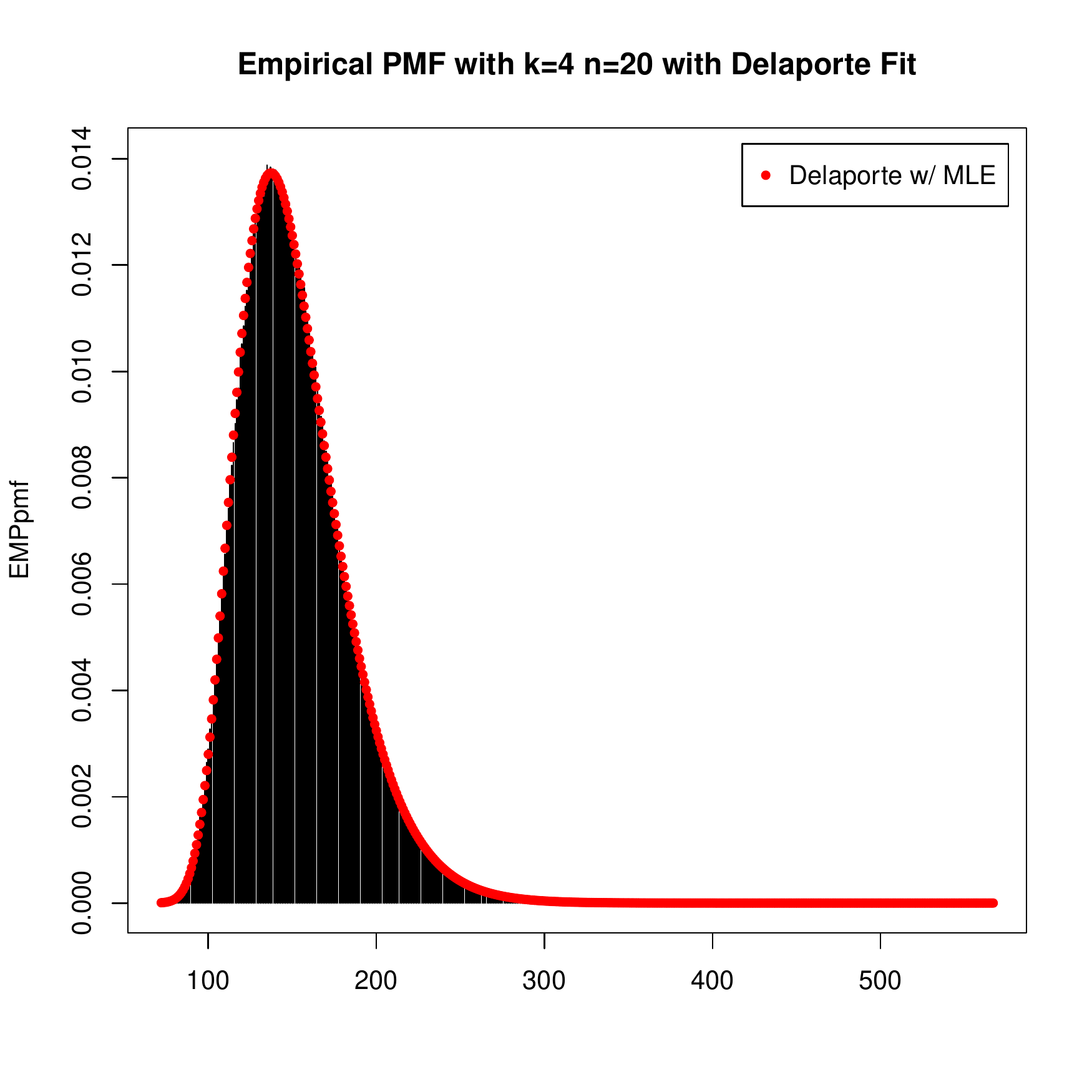} &  
			\includegraphics[scale=.21]{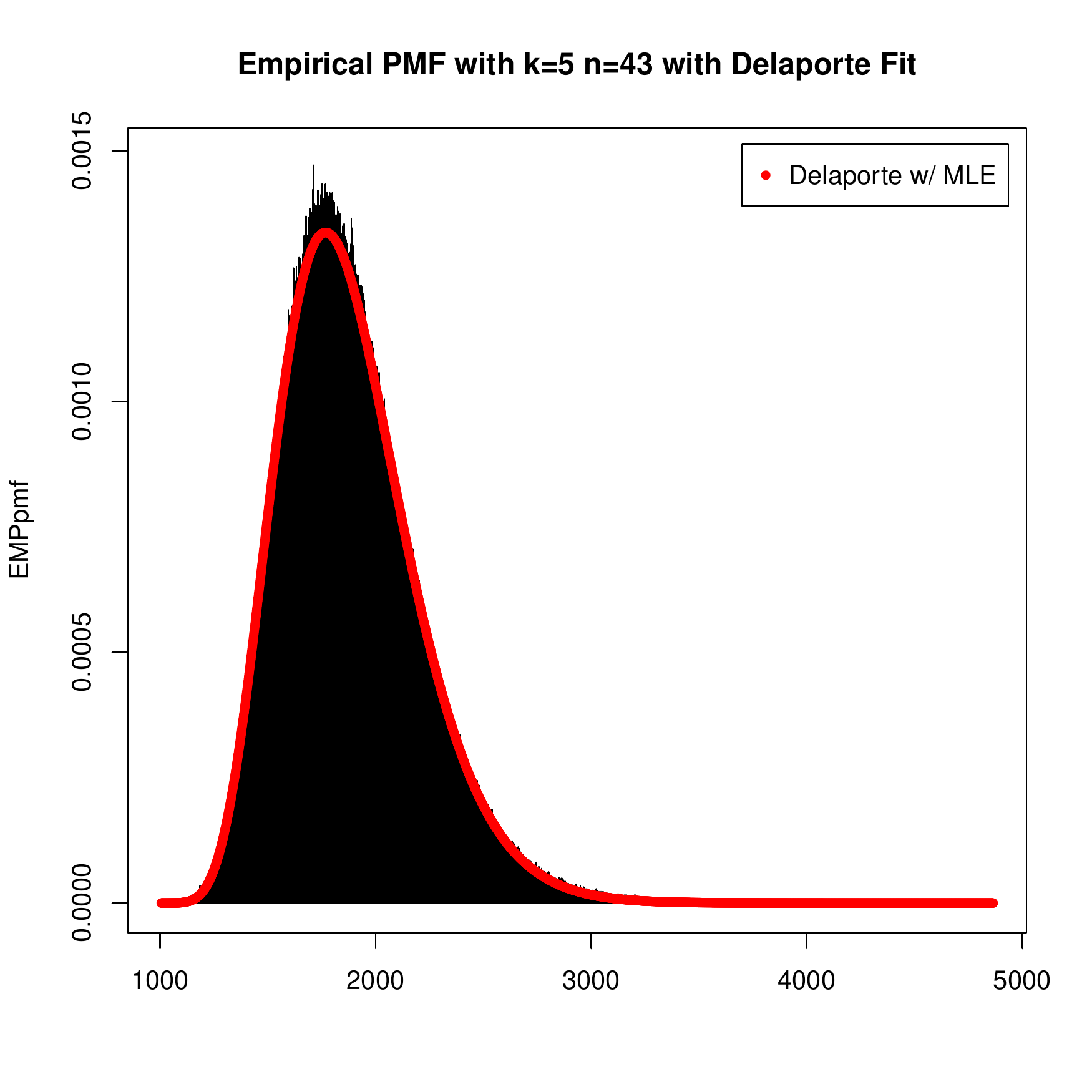} &   
			\includegraphics[scale=.21]{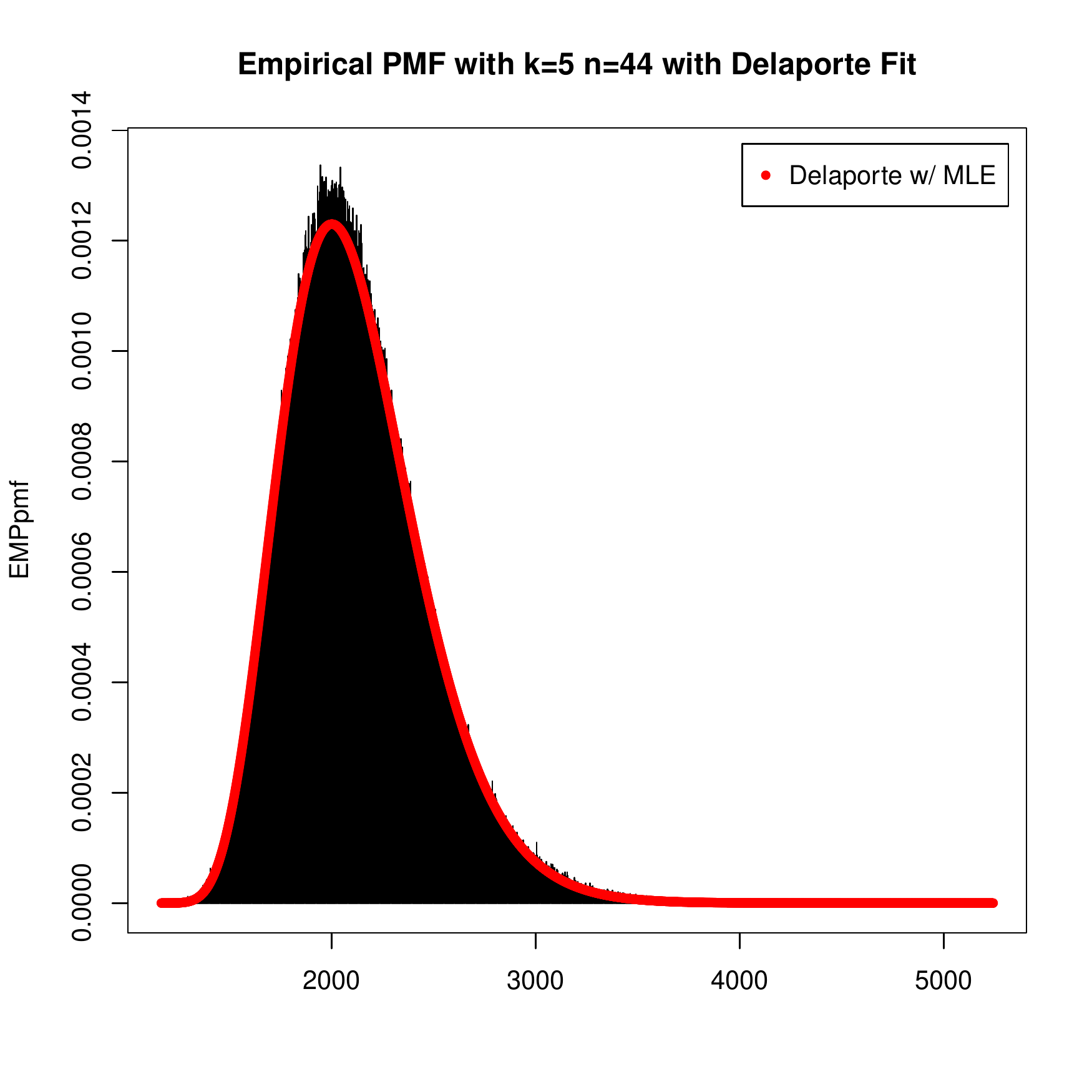} \vspace{-.4em}\\
			Sample size = 1M & Sample size = 1.1M & Sample size = 1M\\ 
			\includegraphics[scale=.21]{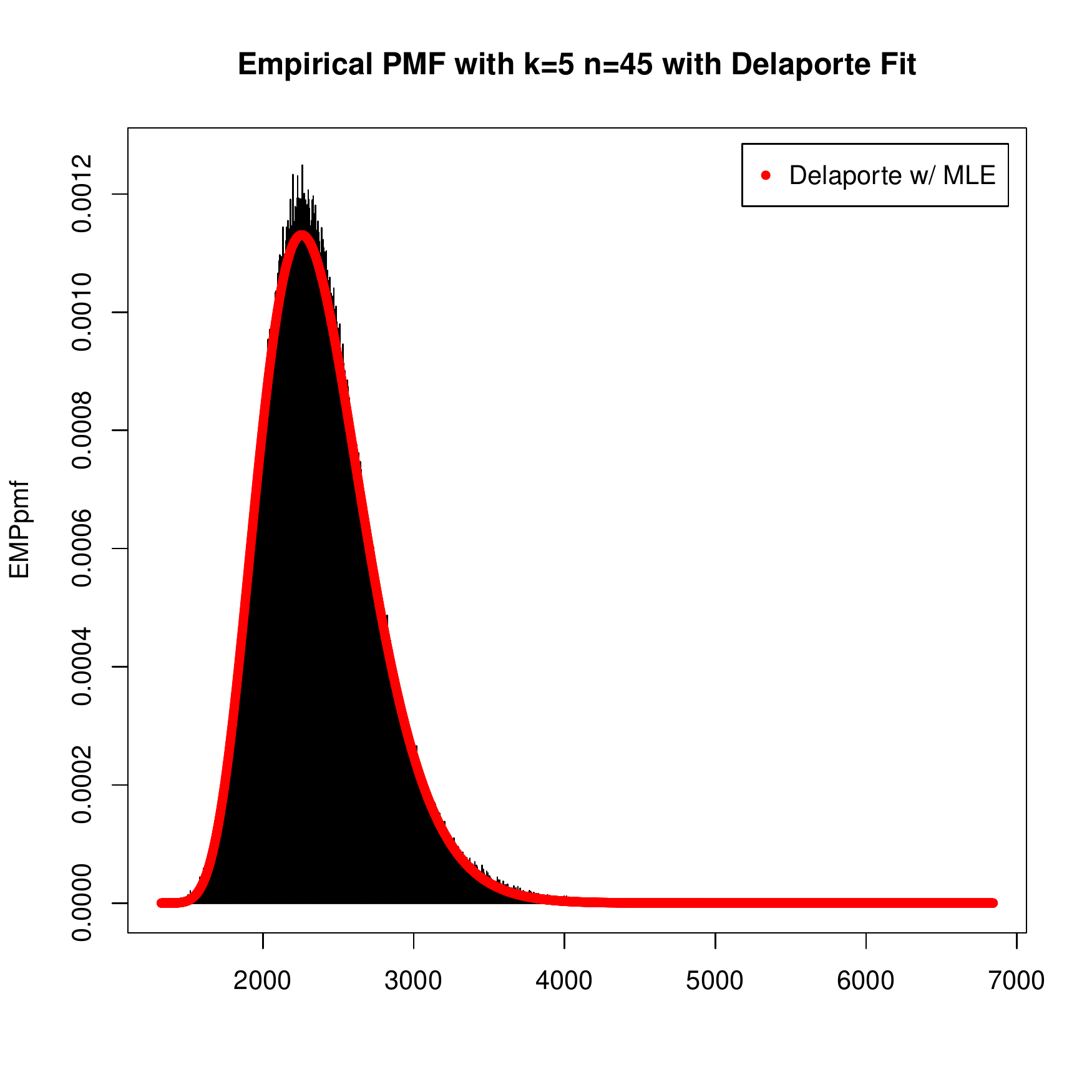} &  
			\includegraphics[scale=.21]{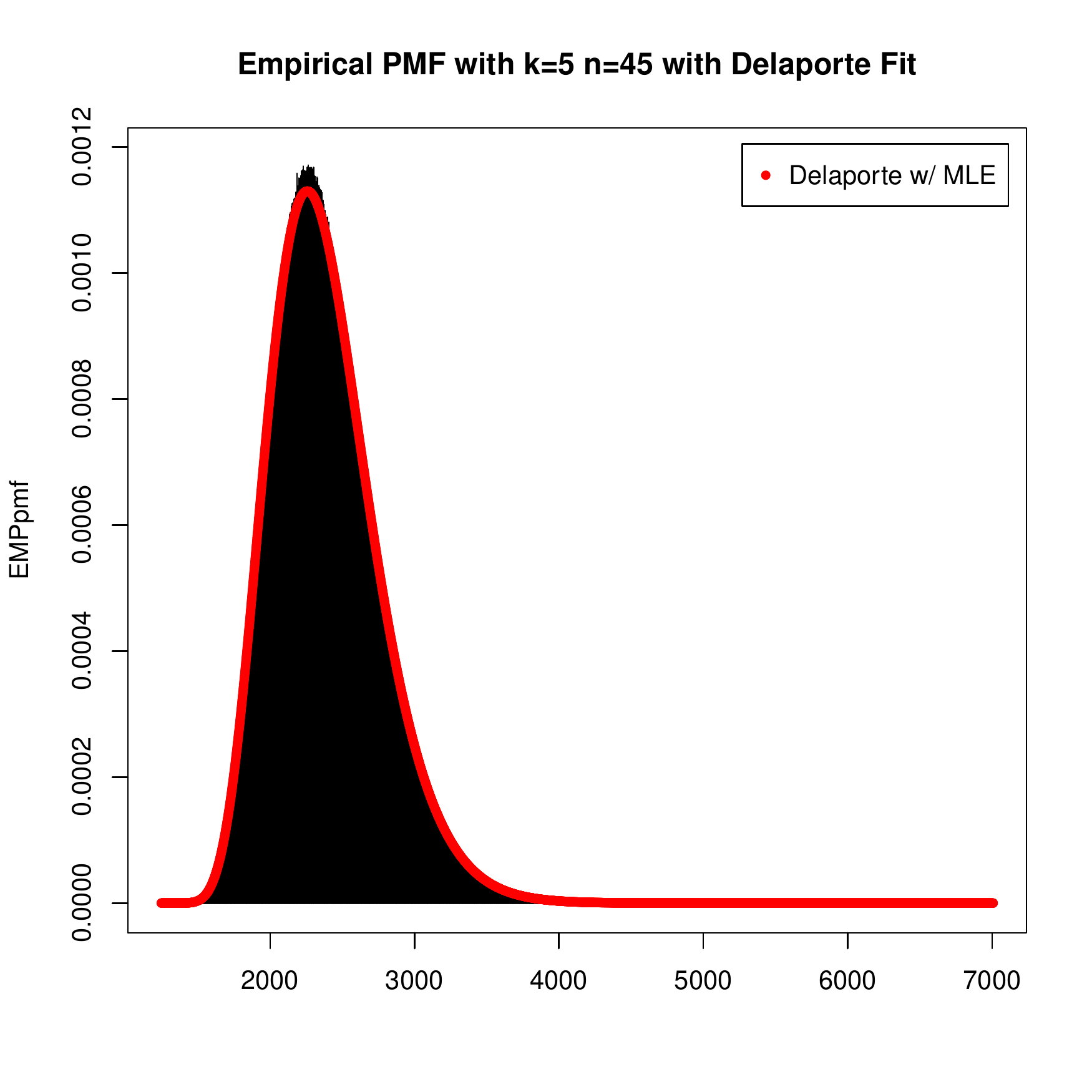} &   
			\includegraphics[scale=.21]{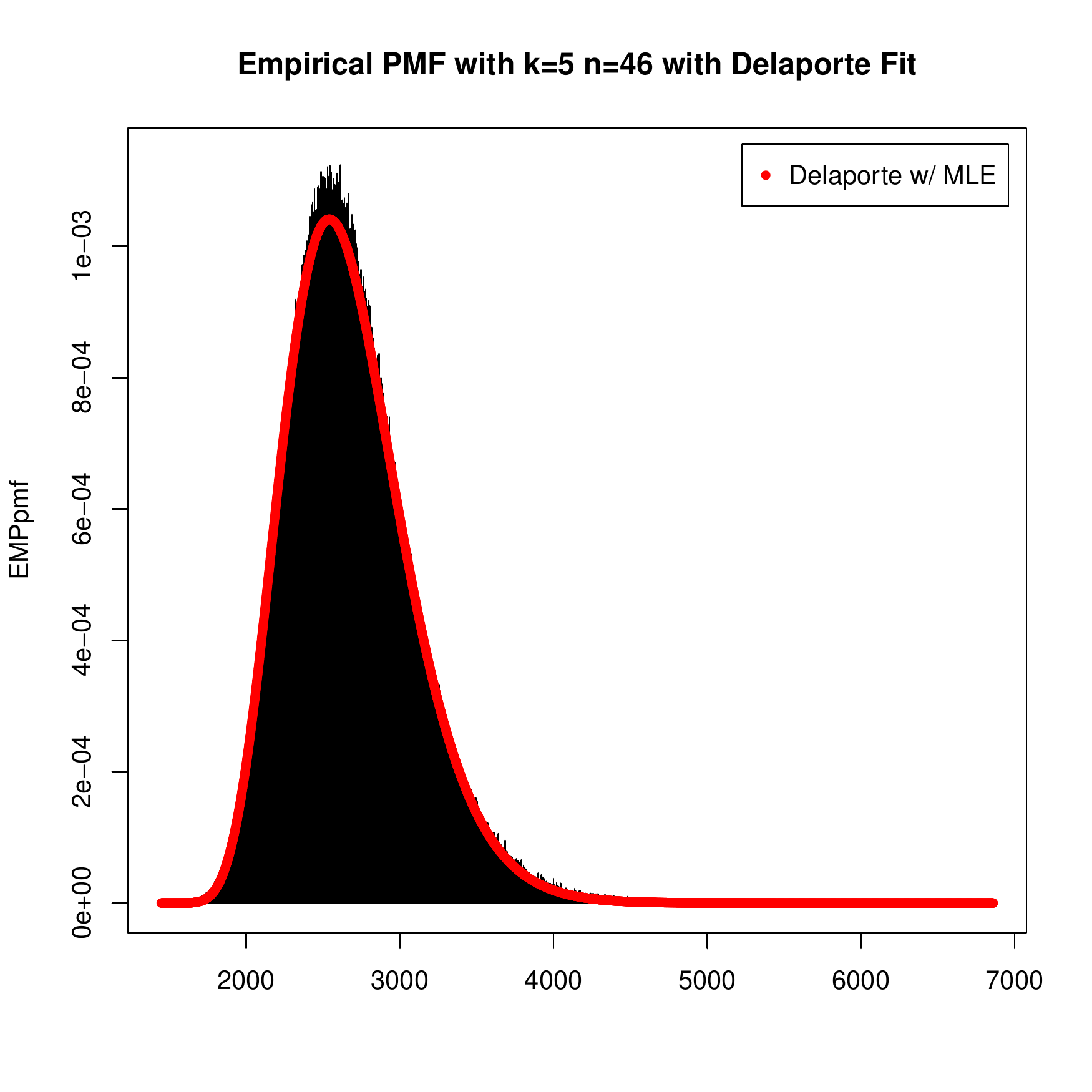} \vspace{-.4em}\\
			Sample size = 1M & Sample size = 17.4M & Sample size = 1.1M\\  
		\end{tabular}\normalsize
		\caption{Empirical pmfs for various scenarios with Delaporte Overlay}
	\end{center}
\end{figure}

In order to find the best-fitting such Delaporte distribution, we must
find good estimates for the parameters $\lambda$, $\alpha,$ and $\beta$.
The two main approaches are the maximum likelihood estimates (MLE) and
the method of moments estimates (MOM).

We have found that the MLEs consistently provide better results than the MOM estimates (this is generally true because likelihood methods
are more efficient).  In fact, these MLEs produce near-optimal results, i.e.,
a total $\ell_1$-distance between the empirical pmf and the 
MLE-estimated Delaporte distribution very near 0.  Unfortunately, closed-form formulas for the MLEs of
$\lambda, \alpha$, and $\beta$ do not exist in our situation (it should also
be noted that {\tt R}'s calculation of the MLEs for $k=5$ neared a day of
computation time for each $n$).  Hence, in the next section
we use MOM estimators for theoretical work. 
Although the MOM estimators do not provide the best fit, they are still very reasonable and relatively close to the MLEs as we show in our Simulation Study section.

A remark about the support of the Delaporte distribution as an approximation 
for the distribution of $X_k$ is in order.  We know that $X_k$ can only
take on values in $\{0,1,2,\dots,{n \choose k}\}$ while the Delaporte
distribution's support is the nonnegative integers.  The
$\ell_1$-distances  do include  the tail of the Delaporte distribution, i.e., the value
of $\sum_{j > {n \choose k}} \mathbb{P}(D=j)$.  Hence, over the
support of $X_k$, the total $\ell_1$-distance is smaller, although we will show that the difference is negligible.

Given $n,k \in \mathbb{Z}^+$, let $q = \mathbb{P}(D > {n \choose k}) = \sum_{j > {n \choose k}} \mathbb{P}(D=j)$.  Our goal is to show that
$q$ is negligible.  We will use the one-sided Chebyshev inequality:
 $\mathbb{P}(X - \mu \geq x) \leq \frac{\sigma^2}{\sigma^2+x^2},$
where $\mu = \mathbb{E}(X)$ and $\sigma^2 = $Var$(X)$.
In our situation we have $\mu = \frac{2}{2^{{k \choose 2}}}{n \choose k}$
so we let $x = \frac{2^{{k \choose 2}}-2}{2^{{k \choose 2}}}{n \choose k}+1$
to bound $P(D > {n \choose k})$.  We also know that Var$(D) = \mu + \alpha \beta^2$.  In the next section, we present evidence to suggest that 
$\sigma^2 \sim {n \choose k}n^{k-3} 2^{-k^2}$ so we will use this as an
assumption.  Putting this all together, we find that
$$
\mathbb{P}\left(D>{n \choose k}\right) \lesssim \frac{1}{2^k k^k n^3}.
$$
In practice, we have observed -- for small $k$ and $n$ -- that this bound
is quite weak.  Nevertheless, this does show that the tail probabilities
of our Delaporte distributions are quite negligible.

\section{Implications and Evidence}

\noindent
As stated before,
it was shown in \cite{G} that, loosely speaking, $X_k$ is asymptotically
Poisson.  This was done by showing that the total $\ell_1$-distance between
the distributions of
$X_k$ and a Poisson random variable with mean $\mathbb{E}(X_k)$ tends to
0 as $k$ tends to infinity (with $n$ bounded from above by a
function of $k$).  We will first
show that the proposed Delaporte distribution is consistent with this
fact.  We will be using the MOM estimates and
note that $\mathbb{E}(X_k) = \lambda + \alpha\beta$ under the method of
moments.  We use the following notation.

\begin{notation}  Let $X$ be a random variable.  We denote its moment generating 
function by mgf$(X)$; that is, mgf$(X) = \mathbb{E}(e^{tX})$.
\end{notation}

\begin{theorem} Let $n,k \in \mathbb{Z}^+$ with $k \geq 3$.  Define $D \sim$ Delaporte$(\lambda, \alpha, \beta)$, and $P \sim$ Poisson$(\lambda+\alpha\beta)$.  Then the $\mathrm{mgf}(D) \rightarrow \mathrm{mgf}(P)$ as $k \rightarrow \infty$
under the following assumptions:
\begin{center}
 $(1) \,\,\,\,\displaystyle\alpha   \sim \frac{{n \choose k}}{n^{k-1}}$;
\qquad $(2) \,\,\,\,\displaystyle\beta \sim \frac{n^{k-2}}{2^{{k \choose 2}}}$;
\qquad$(3) \,\,\,\,\displaystyle n = O\left(k^{1+\frac{1}{k-1}} \cdot2^{\frac{k}{2}}\right)$
\end{center} \label{asy}
\end{theorem}

\begin{proof}
Since $D$ is a convolution of a
Negative Binomial random variable with success probability $\frac{\beta}{1+\beta}$ and mean $\alpha \beta$
and a Poisson random variable with mean $\lambda$, using the moment generating functions
of these, we easily have
$$
\mathrm{mgf}(D) = \frac{e^{\lambda (e^t-1)}}{(1 - \beta(e^t-1))^\alpha}.
$$
Isolate the denominator and use $\ln(1+x) \approx x$ for small $x$.
Since   $\beta \rightarrow \infty$ as $k,n \rightarrow \infty$, with
the restriction $t \leq \ln\left(1+\frac{1}{\beta^2}\right)$
we have $\beta(e^t-1) \leq \frac{1}{\beta},$ so that $\beta(e^t-1)$ is small
for sufficiently large $\beta$.  Hence,
$\ln \left((1 - \beta(e^t-1))^\alpha\right) \approx -\alpha \beta (e^t-1)$.
 For large $n$ and $k$, this gives 
$
(1 - \beta(e^t-1))^\alpha \approx e^{-\alpha \beta (e^t-1)},
$
for $t \leq \ln\left(1+\frac{1}{\beta^2}\right)$.  Hence, we find that
mgf$(D) \approx e^{(\lambda +\alpha\beta)(e^t-1)} = $ mgf$(P)$ on a
small interval including $t=0$, which is enough to conclude the result.
\end{proof}

Having Theorem \ref{asy} and knowing that there is a one-to-one correspondence between
random variables and their moment generating functions, we can state that, loosely,
the Delaporte random variable is asymptotically Poisson.  Hence, we are not
violating the ``Poisson Paradigm," as noted in \cite{AS}.

We will now give evidence to suggest that the assumptions in Theorem \ref{asy} are
satisfied. We will be using MOM estimates.
We know that $\mu=\mathbb{E}(X_k) = {{n \choose k}}{2^{1-{k \choose 2}}}$.
Using Zeilberger's Maple package {\tt SMCramsey} that accompanies \cite{Z} we find
the leading terms for
the second and third moments about the mean for $X_k$ for small $k$:
\small
\begin{center}
\begin{tabular} { |c|l|l|} 
\hline
$k$& $\mathbb{E}((X-\mu)^2)$& $\mathbb{E}((X-\mu)^3)$  \\
\hline \hline
&&\\[-5pt]
3 & $\displaystyle {n \choose 3} \cdot \frac{3}{2^{4}}$ & $\displaystyle {n \choose 4} \cdot \frac{6n}{2^{6}}$ \\[12pt]
4 & $\displaystyle {n \choose 4} \cdot \frac{12n}{2^{10}}$ & $\displaystyle {n \choose 4} \cdot \frac{24n^3}{2^{15}}$ \\[12pt]
5 & $\displaystyle {n \choose 5} \cdot \frac{15n^2}{2^{18}}$ & $\displaystyle {n \choose 5} \cdot \frac{15n^5}{2^{27}}$ \\[12pt]
6 & $\displaystyle {n \choose 6} \cdot \frac{10n^3}{2^{28}}$ & $\displaystyle {n \choose 6} \cdot \frac{10n^7}{3\cdot 2^{42}}$ \\[12pt]
7 & $\displaystyle {n \choose 7} \cdot \frac{35n^4}{2 \cdot 2^{42}}$ & $\displaystyle {n \choose 7} \cdot \frac{35n^9}{6 \cdot 2^{64}}$ \\[12pt]
8 & $\displaystyle {n \choose 8} \cdot \frac{84n^5}{15\cdot 2^{56}}$ & $\displaystyle {n \choose 8} \cdot \frac{42n^{11}}{255 \cdot 2^{84}}$\\[12pt]
\hline
\end{tabular}
\vskip 5pt
{\small {\bf Table 2}: Second and third moment orders}
\end{center}

\normalsize
As noted in Definition \ref{DelDef}, for the Delaporte random variable $D$, we have
$$\begin{array}{c}
\mathbb{E}(D) = \lambda +\alpha\beta; \qquad
\mathbb{E}((D-\mu)^2) = \lambda + \alpha\beta(1+\beta);\qquad
\mathbb{E}((D-\mu)^3) = \lambda + \alpha\beta(1+3\beta+2\beta^2).
\end{array}
$$

By Lemma \ref{ExpVar} and the fact that 
$\mathbb{E}((D-\mu)^2) = \mathbb{E}(D)+\alpha\beta^2$,
we can deduce that
$\alpha\beta^2 \sim {n \choose k} \frac{n^{k-3}}{2^{2{k \choose 2}-2}}.$
Looking at the third moments in Table 2, we have evidence to suggest that
$2\alpha\beta^3 \sim {n \choose k} \frac{n^{2k-5}}{2^{3{k \choose 2}-3}}.$
Taking the ratio of these last two expressions yields
$$
\alpha \sim \frac{{n \choose k}}{n^{k-1}} \quad \mbox{and} \quad
\beta \sim \frac{n^{k-2}}{2^{{k \choose 2}}}.
$$

\noindent
{\bf Remark.} The interested reader can obtain more accurate results by
using the following MOM formulas calculated in Mathematica:
\small
$$\widehat{\lambda} = \dfrac{2 m^2 v+2 v f-m \left(2 f+v^2 \left(s^2 v+7\right)\right)+6 v^3}{4 m^2-12 m v+v^2 \left(9-s^2 v\right)};$$
\vskip 5pt
\noindent
$$\widehat{\alpha} = \dfrac{-4 (m-v)^4}{4 m^2(m-4v)\!-\!f(4m-6v)\!+\!s^2v^3(m-v)
\!+\! v^2(21m-9v)}; \qquad\widehat{\beta} = \dfrac{-2 m^2+f+5 m v-3 v^2}{2 (m-v)^2},$$
\vskip 5pt
\noindent
\normalsize
where $m$ is the sample mean, $v$ is the sample variance,  $s$ is the  sample skewness, and we use the notation
$f=\sqrt{s^2 v^3 (m-v)^2}$. Note that we calculate the sample skewness using the default {\tt R} command \texttt{skewness()}. A summary of sample skewness calculations can be found \cite{JG}.

\begin{remark} We are clearly extrapolating in our formulas for $\alpha$ and $\beta$, but it is
interesting to note that the order of $n$ required in Theorem \ref{asy} is actually
slightly better (by a factor of $k^{\frac{1}{k}}$)
than the best known lower bound on $R(k,k)$, while a larger order for $n$ would
void our proof of the asymptotic Poisson nature of $X_k$ via
the Delaporte distribution.  Might this be evidence that $k2^{\frac{k}{2}}$
is the correct order for the Ramsey number $R(k,k)$?
\end{remark}

\section{Simulation Study: MLE vs. MOM}

\noindent
Though the MOM estimators can be calculated in closed-form, these estimators are quite complicated and the derivation of their expected value is impractical. MLEs are more burdensome and so far have been calculated numerically using the {\tt optim} function in {\tt cran R} \cite{R}. 

Due to the nature of these estimators, a discussion on their accuracy and precision are difficult as closed-form expectations are impractical to calculate. To explore and compare the accuracy and precision of the MOM estimators and MLEs, consider their asymptotic behavior across three simulations motivated by estimators from the $k=4, n=14$; $k=5, n=20$; and $k=5, n=49$ graphs.

Tables 3 and 4 contain the results of these simulations of Delaporte data with parameter values as well as MOM estimates and MLEs across increasing sample sizes $n$. As $n$ increases, the MLEs quickly approach the true values, whereas the MOM estimators appear to require higher sample sizes. These simulations suggest that both the MLEs and MOM estimates might be asymptotically unbiased where, as expected, the MLEs are more efficient and have lower variability.

\begin{center}
\footnotesize
	\begin{tabular} { |l|l|c|c|c|c|c|c|} 
		\hline
		\multicolumn{2}{|c|}{}& \multicolumn{3}{c|}{MLEs}\\ \hline
		&&&&\\[-5pt]
		Parameters & $n$  & $\widehat{\alpha}$ & $\widehat{\beta}$ & $\widehat{\lambda}$\\
		\hline \hline
		\begin{tabular}{@{}l@{}}$\alpha=1.84$ \\ $\beta=7.89$ \\ $\lambda=16.75$\\\end{tabular} 
& 
\begin{tabular}{@{}l@{}}100\\1000\\10000\\100000
\end{tabular}
&
\begin{tabular}{@{}l@{}}
2.03 (1.40)\\1.87 (0.39)\\1.85 (0.10)\\1.84 (0.03)\\
\end{tabular}
& 
\begin{tabular}{@{}l@{}}8.90 (3.41)\\7.92 (0.91)\\7.88 (0.27)\\7.89 (0.09)\\\end{tabular}
& 
\begin{tabular}{@{}l@{}}16.96 (3.65)\\16.70 (1.22)\\16.74 (0.36)\\16.74 (0.11)\\
\end{tabular}
\\\hline
		
\begin{tabular}{@{}l@{}}
$\alpha=3.74$ \\ $\beta=15.46$ \\ $\lambda=93.57$\\\end{tabular} 
& \begin{tabular}{@{}l@{}}100\\1000\\10000\\100000\\\end{tabular}
& 
\begin{tabular}{@{}l@{}}4.55 (2.46)\\3.68 (0.75)\\ 3.74 (0.22)\\3.74 (0.08)\\\end{tabular}
& 
\begin{tabular}{@{}l@{}}15.48 (4.71)\\15.83 (1.88)\\15.46 (0.53)\\15.47 (0.18)\\\end{tabular}
& 
\begin{tabular}{@{}l@{}}90.66 (14.76)\\94.46 (5.28)\\ 93.60 (1.51)\\ 93.58 (0.55)\\\end{tabular}
\\\hline
		\begin{tabular}{@{}l@{}}$\alpha=9.45$ \\ $\beta=163.51$ \\$\lambda=2178.57$\\\end{tabular} 
& \begin{tabular}{@{}l@{}}100\\1000\\10000\\100000\\\end{tabular}&
		\begin{tabular}{@{}l@{}}11.39 (10.25)\\9.84 (2.20)\\9.56 (0.66)\\9.48 (0.18)\\\end{tabular}& 
		\begin{tabular}{@{}l@{}}183.02 (69.73)\\162.73 (19.11)\\162.66 (6.17)\\163.25 (1.67)\end{tabular}& 
		\begin{tabular}{@{}l@{}}2159.43 (560.00)\\2159.16 (154.19)\\2171.45 (49.41) \\2176.68 (14.19)\end{tabular}\\\hline
	\end{tabular}
	\vskip 5pt 
\noindent	{\small {\bf Table 3}: MLEs for Delaporte data simulated with previously
calculated parameters across various sample sizes}

\vskip 10pt
\footnotesize
	\begin{tabular} { |l|l|c|c|c|c|c|c|} 
		\hline
		\multicolumn{2}{|c|}{}& \multicolumn{3}{c|}{MOM Estimators}\\ \hline
		&&&&\\[-5pt]
		Parameters & $n$  & $\widehat{\alpha}$ & $\widehat{\beta}$ & $\widehat{\lambda}$\\
		\hline \hline
		\begin{tabular}{@{}l@{}}$\alpha=1.84$ \\ $\beta=7.89$ \\ $\lambda=16.75$\\\end{tabular} & \begin{tabular}{@{}l@{}}100\\1000\\10000\\100000\end{tabular}&
		\begin{tabular}{@{}l@{}}3.13 (2.46)\\1.98 (0.63)\\1.86 (0.17)\\1.85 (0.06)\\\end{tabular}& 
		\begin{tabular}{@{}l@{}}7.96 (4.10)\\7.89 (1.43)\\7.88 (0.44)\\7.87 (0.14)\\\end{tabular}& 
		\begin{tabular}{@{}l@{}}14.08 (5.87)\\16.41 (2.02)\\16.71 (0.63)\\16.72 (0.20)\\\end{tabular}\\\hline
		\begin{tabular}{@{}l@{}}$\alpha=3.74$ \\ $\beta=15.46$ \\ $\lambda=93.57$\\\end{tabular} & \begin{tabular}{@{}l@{}}100\\1000\\10000\\100000\\\end{tabular}&
		\begin{tabular}{@{}l@{}}6.10 (4.20)\\ 3.95 (1.23)\\3.77 (0.37)\\3.73 (0.10)\\\end{tabular}& 
		\begin{tabular}{@{}l@{}}14.06 (4.84)\\ 15.60 (2.71)\\15.44 (0.81)\\ 15.49 (0.24) \\\end{tabular}& 
		\begin{tabular}{@{}l@{}}82.30 (20.90)\\92.85 (8.31)\\93.42 (2.69)\\ 93.63 (0.70)\\\end{tabular}\\\hline
		\begin{tabular}{@{}l@{}}$\alpha=9.45$ \\ $\beta=163.51$ \\$\lambda=2178.57$\\\end{tabular} & \begin{tabular}{@{}l@{}}100\\1000\\10000\\100000\\\end{tabular}&
		\begin{tabular}{@{}l@{}}14.92 (12.81)\\10.49(4.11)\\9.62 (1.06) \\9.46 (0.28)\\\end{tabular}& 
		\begin{tabular}{@{}l@{}}173.22 (95.54)\\162.38 (29.14)\\162.60 (9.31)\\ 163.46 (2.55)\\\end{tabular}& 
		\begin{tabular}{@{}l@{}}1950.67 (709.17)\\2125.06 (275.23)\\2167.94 (79.22)\\ 2178.36 (21.67) \\\end{tabular}\\\hline
	\end{tabular}
	\vskip 5pt
\noindent {\small {\bf Table 4}: MOM estimators   for Delaporte data simulated with previously
calculated parameters across various sample sizes}

\end{center}

The results of this simulation study lend credence to the validity of using MOMs in the last section.

\section{Monochromatic Arithmetic Progressions}

\noindent
We follow a similar strategy for van der Waerden numbers. Let all possible 2-colorings of a given interval $[1,n]$ of
positive integers be equally likely and define $Y_k=Y_k(n)$ to be the random variable giving the total number of monochromatic arithmetic progressions of length $k$ in the interval.  To approximate the distribution of $Y_k$, the program {\tt APCount}\footnote[3]{Available at {\tt http://www.aaronrobertson.org}.} takes a user input of positive integers $n$, $k$, and $g$ and generates $g$ instances of the integers between $1$ and $n$ 
each having one of two colors. The color of any given integer is randomly
decided, again with Python's random module. The program counts the total number of monochromatic arithmetic progressions of length $k$
in each instance.  Over all $g$ instances, we then produce an
empirical probability mass function for $Y_k$.


Though {\tt APCount} is dealing with larger numbers than {\tt GraphCount}, it works much more quickly, given that arithmetic progressions are simpler than graphs. However, like {\tt GraphCount}, {\tt APCount} has
a run-time of $O(n^2)$ for $k=3$ and $O(n^k)$ for $k \ge 4$.

\begin{center}\small
\begin{tabular} { |c|c|c|} 
\hline
Input $k$ & Input $n$ & Time per progression (sec)  \\
\hline \hline
3 & 9 & 0.00016 \\
4 & 35 & 0.00049 \\ 
5 & 178 & 0.0226 \\
6 & 1132 & 4.9 \\
7 & 3703 & 167 \\
8 & 11495 & 4980 \\
\hline
\end{tabular}
\vskip 5pt 
{\small {\bf Table 5}: Run-times for {\tt APCount}}
\end{center}

In Figure 4, we present the empirical ps for small
$k$ and $n$, with $n$ near $w(k)$.

\begin{figure}[h!]\tiny
	\begin{center} \hspace*{-0.1in}
		\begin{tabular}{ccc}
			\includegraphics[scale=.21]{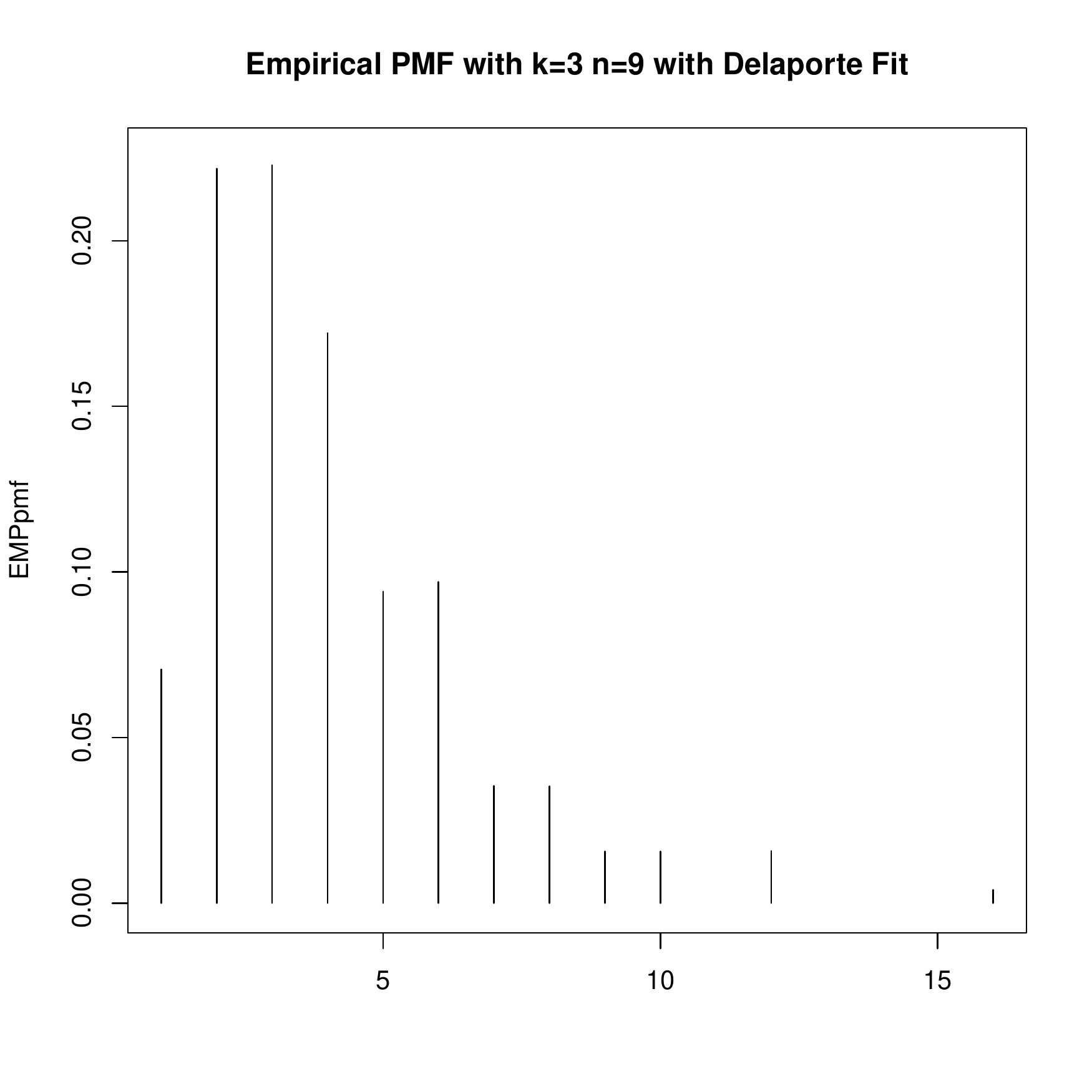} &  
			\includegraphics[scale=.21]{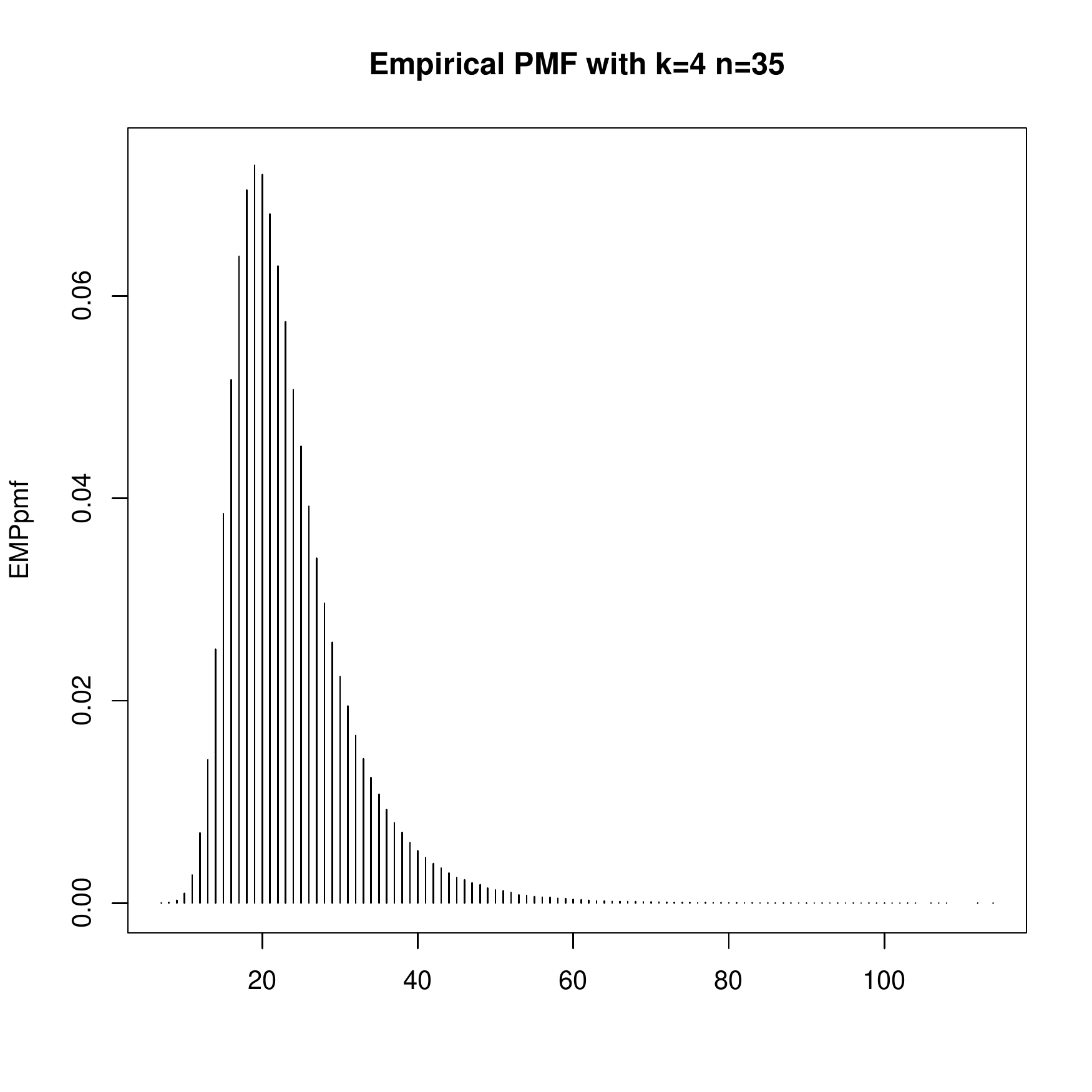} &   
			\includegraphics[scale=.21]{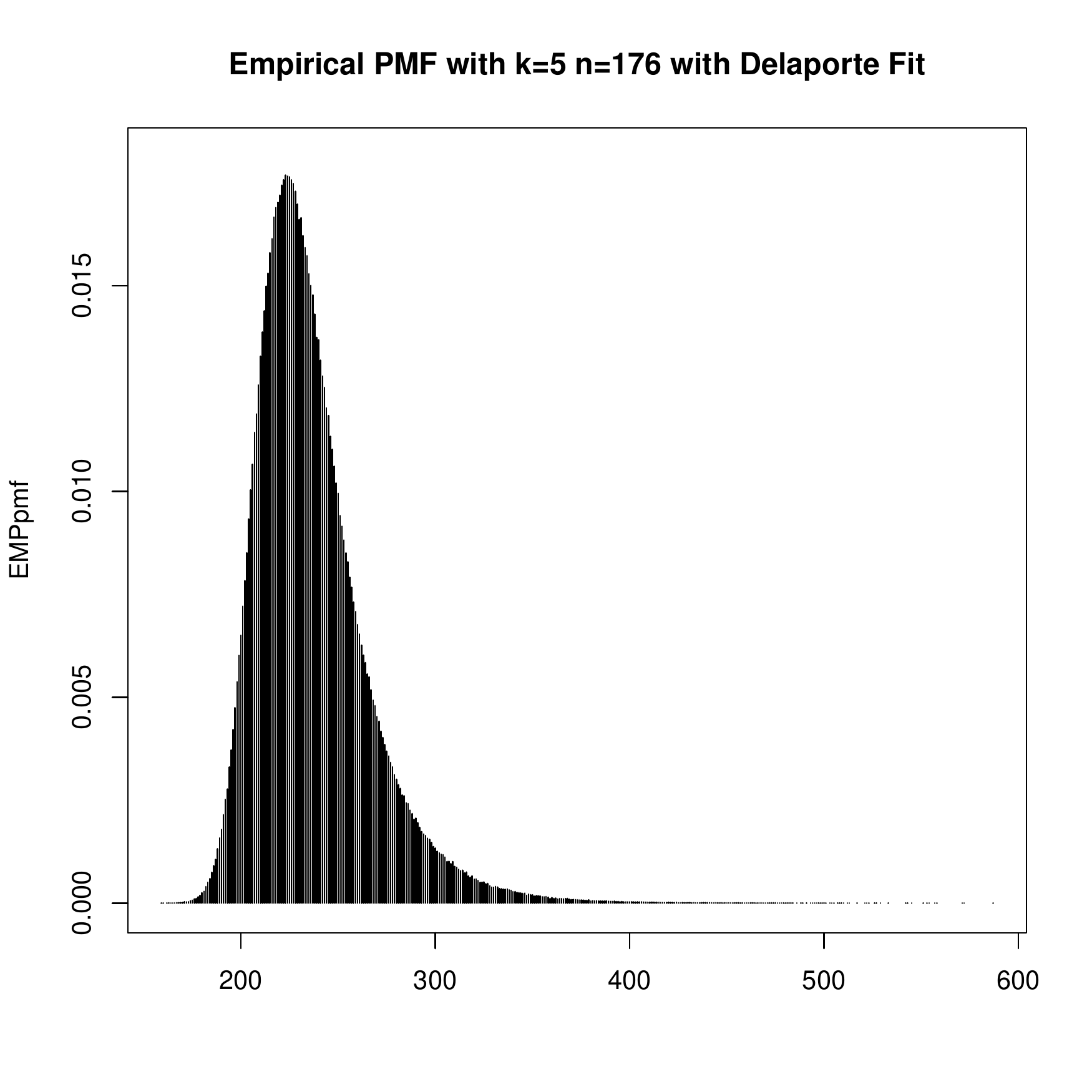} \vspace{-.4em}\\ 			
			Sample size = 1M & Sample size = 1M & Sample size = 1.7M\\
			\includegraphics[scale=.21]{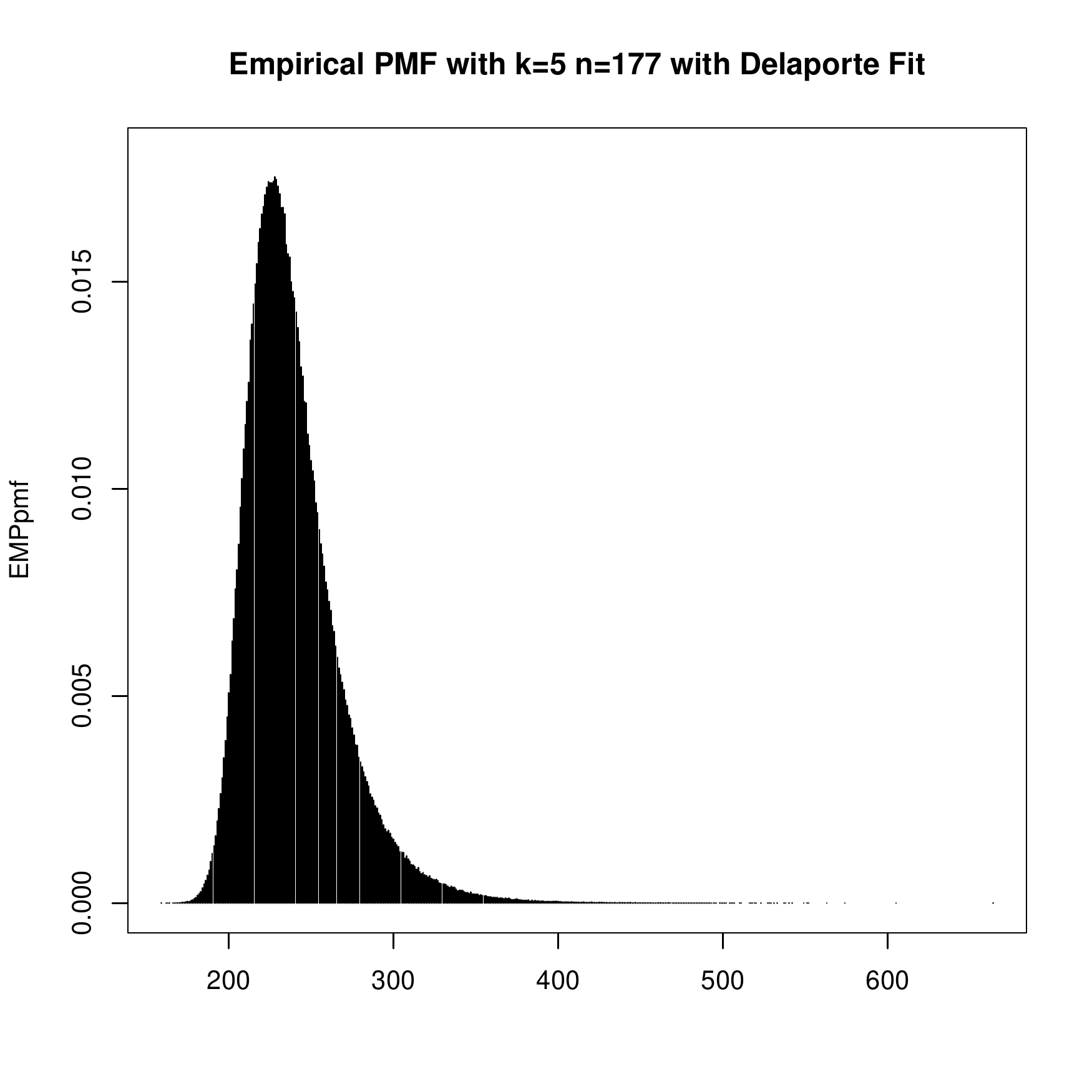} &  
			\includegraphics[scale=.21]{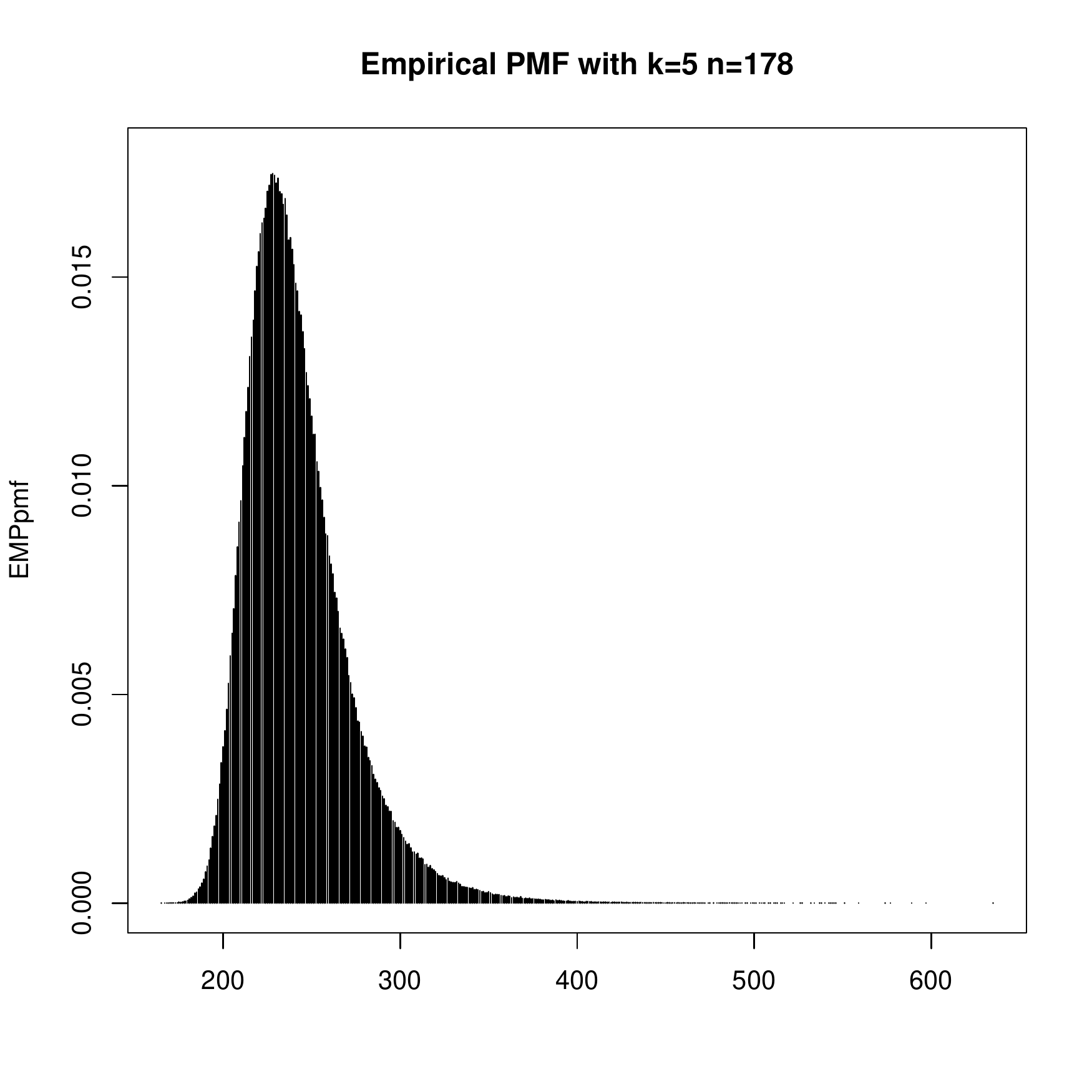} &   			
			\includegraphics[scale=.21]{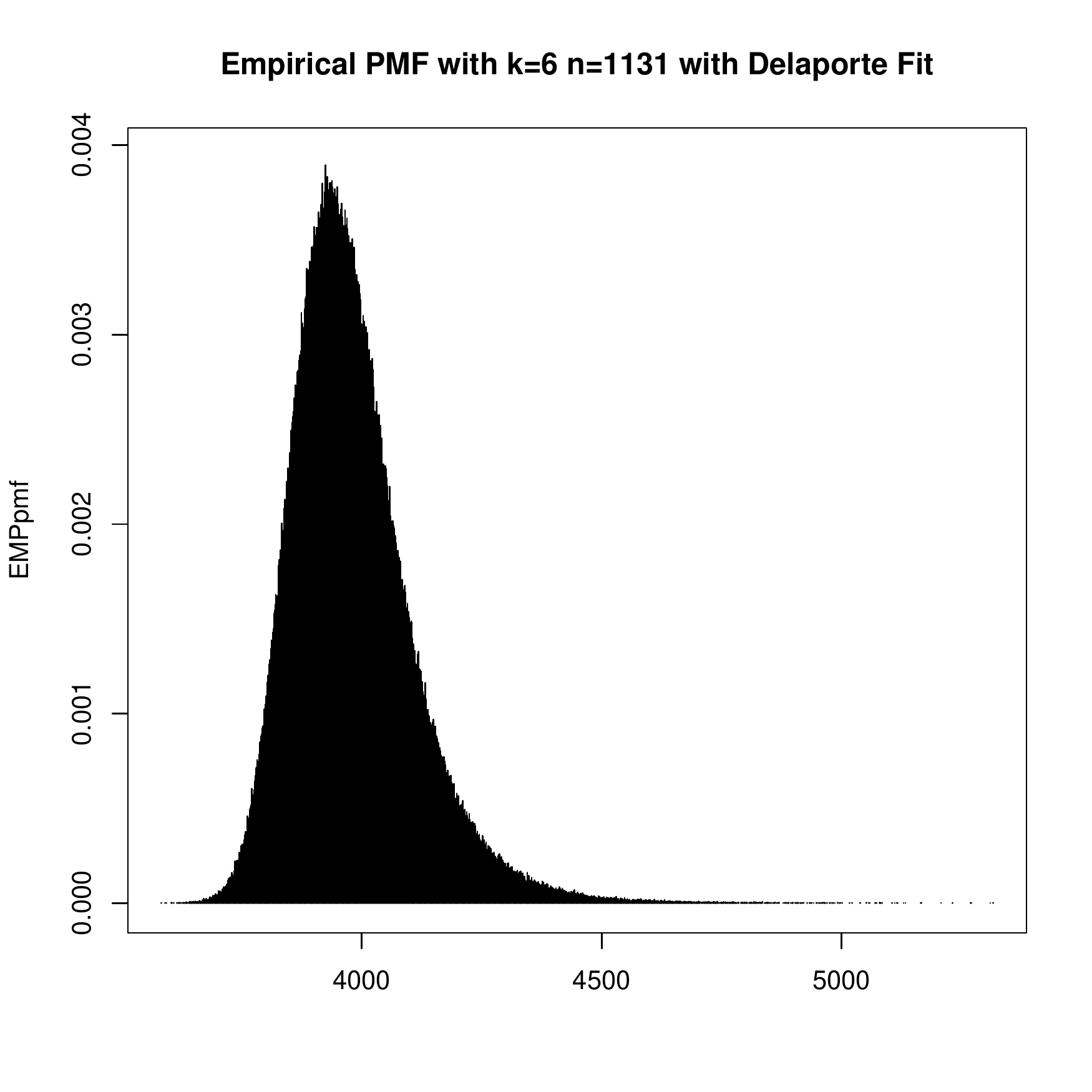} \vspace{-.4em}\\ 
			Sample size = 1.1M & Sample size = 1M & Sample size = 0.85M\\
			\includegraphics[scale=.21]{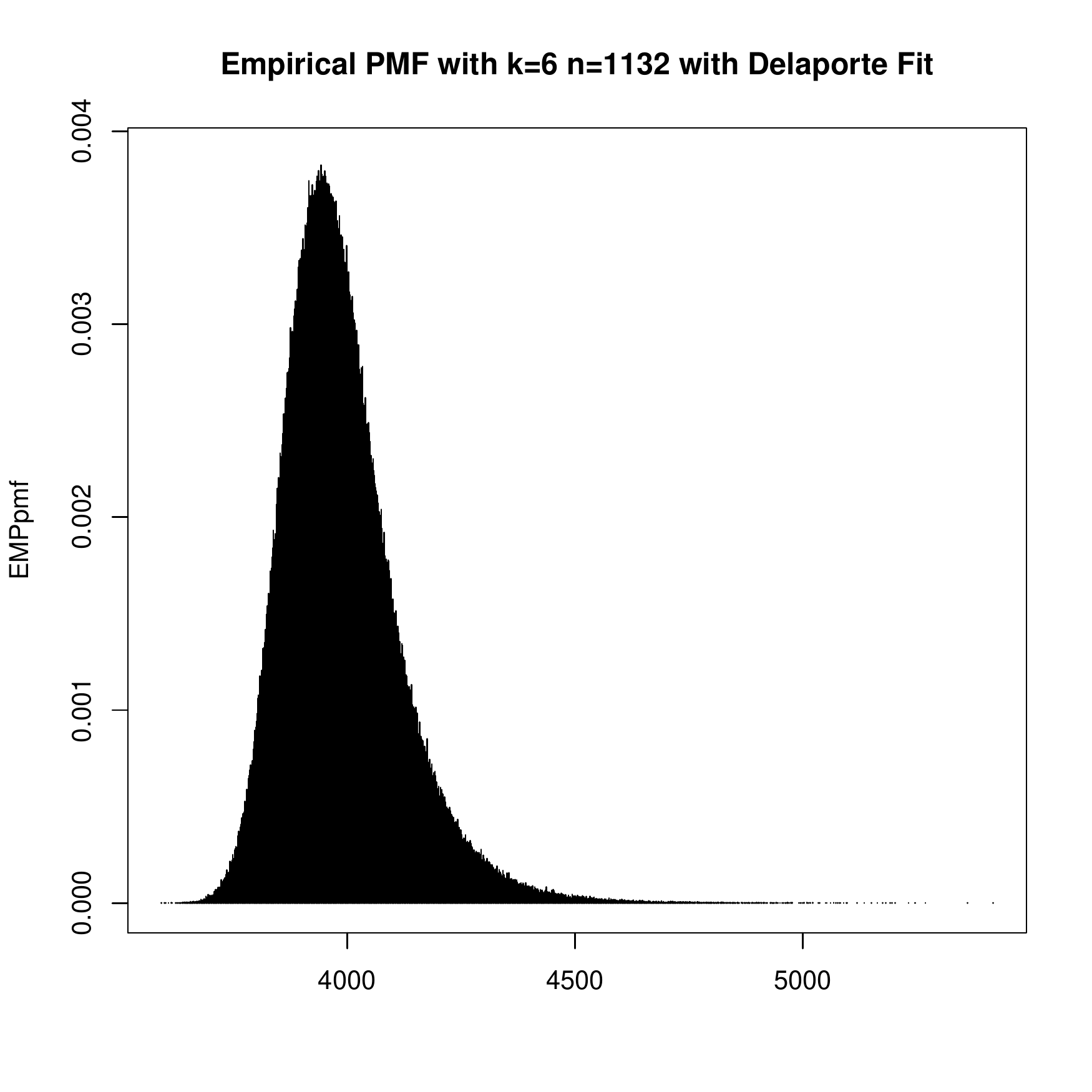} &  			
			\includegraphics[scale=.21]{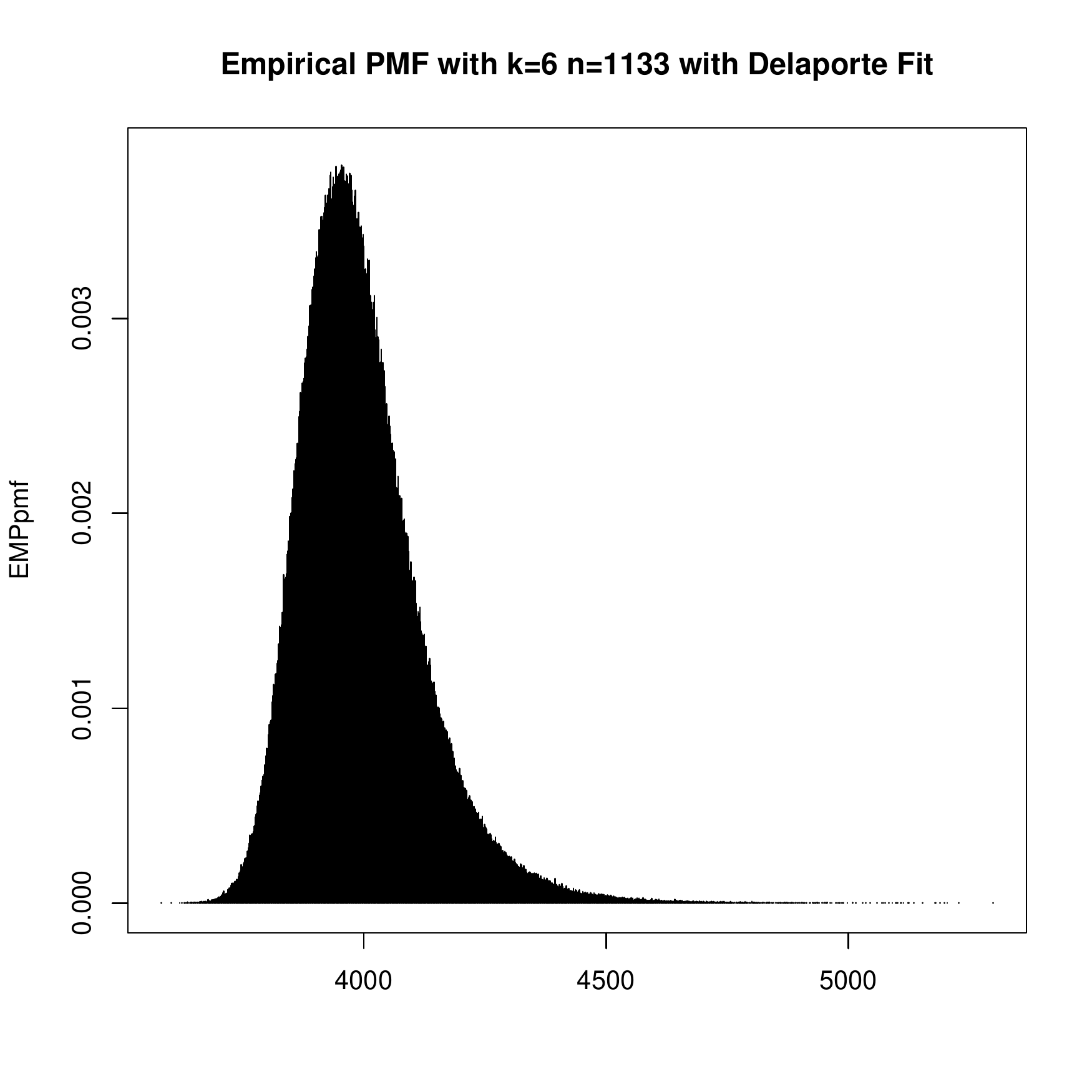} &   
			\includegraphics[scale=.21]{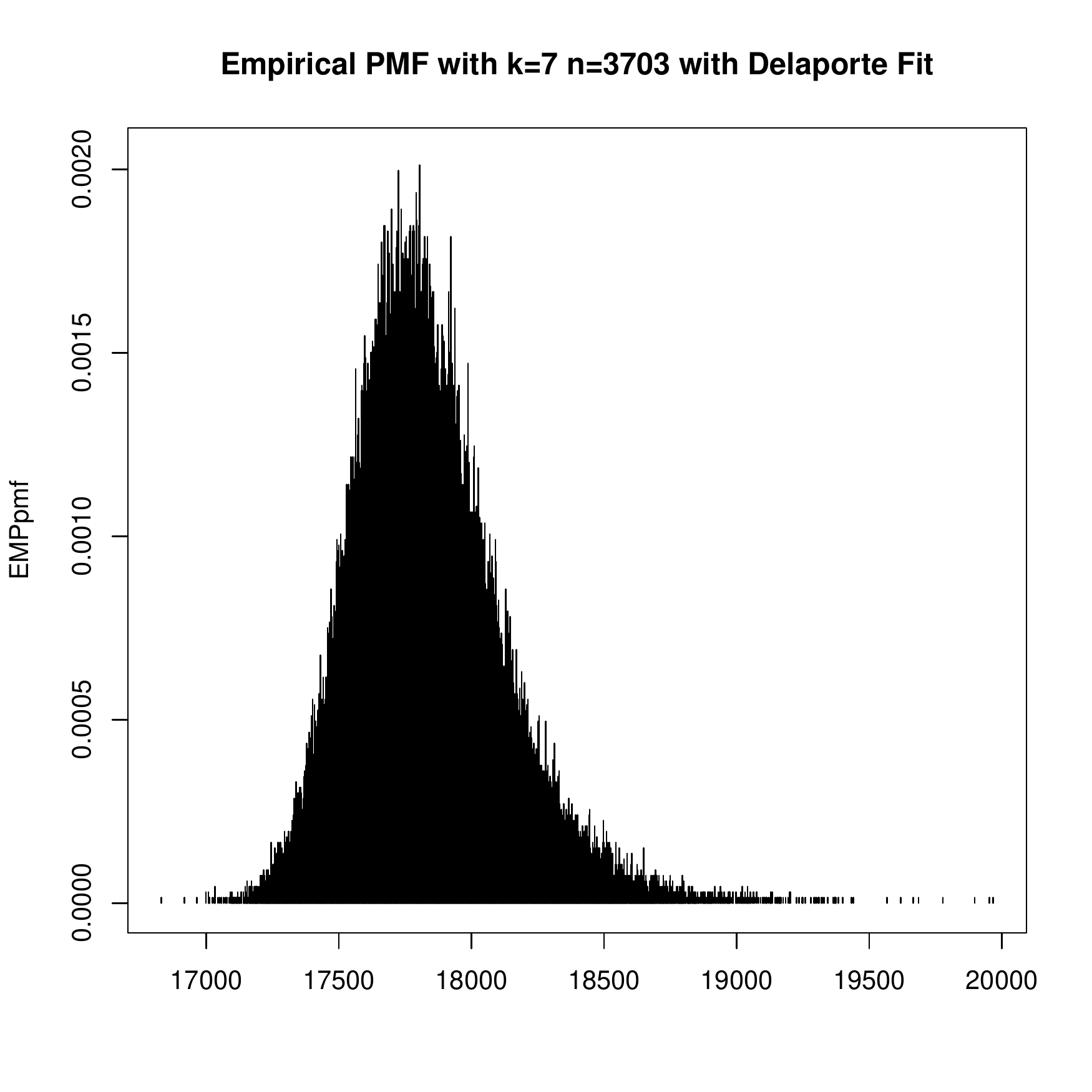} \vspace{-.4em}\\ 
		\end{tabular}
		\normalsize
		\caption{Empirical pmfs for the number of monochromatic arithmetic progressions}
	\end{center}
\end{figure}

At first blush, it is quite striking that the histograms for the number of monochromatic arithmetic progressions
have very similar shapes to the histograms for the number of monochromatic complete subgraphs.
However, our arithmetic progressions are mostly independently colored with some dependence between some pairs
of arithmetic progressions.  So, heuristically, the number of monochromatic
arithmetic progressions would be asymptotically Poisson by very similar
reasoning to the number of monochromatic subgraphs.  

We will now show that for fixed (small) $k$, the Poisson distribution is
under-dispersed for the distribution of the number of monochromatic
arithmetic progressions.  

\begin{lemma} Let $k \geq 3$.  We have $\displaystyle \lim_{n \rightarrow \infty}
\frac{\mathbb{E}(Y_k(n))}{\mathrm{Var}(Y_k(n))}<1$.
\end{lemma}

\noindent
{\it Proof.}  Let $Z_i$ be the indicator function for whether or not the $i^{\mathrm{th}}$
$k$-term arithmetic progression is monochromatic so that $Y_k = \sum_{i=1}^{\frac{n^2}{2(k-1)}}Z_i$.
By the linearity of expectation, we have $\mathbb{E}(Y_k) = \sum_{i=1}^{\frac{n^2}{2(k-1)}}\mathbb{E}(Z_i) = {\frac{n^2}{2(k-1)}} \cdot \frac{1}{2^{k-1}}
= \frac{n^2}{(k-1)2^k}$.  We also have $\mathrm{Var}(Y_k) = \mathrm{Var}(\sum_{i=1}^{\frac{n^2}{2(k-1)}} Z_i)
= \sum_{i=1}^{\frac{n^2}{2(k-1)}} \mathrm{Var}(Z_i) + \sum_{i \neq j}\mathrm{Cov}(Z_i,Z_j)
= \sum_{i=1}^{\frac{n^2}{2(k-1)}} \left(\mathbb{E}(Z_i^2) - \mathbb{E}^2(Z_i)\right) +  \sum_{i \neq j} \mathrm{Cov}(Z_i,Z_j)
$.  Since $Z_i^2 = Z_i$, this simplifies to
$\mathrm{Var}(Y_k) = \sum_{i=1}^{\frac{n^2}{2(k-1)}} \left(\frac{1}{2^{k-1}} - \frac{1}{2^{2k-2}}\right)+ \sum_{i \neq j}\mathrm{Cov}(Z_i,Z_j) = \frac{2^{k-1}n^2}{(k-1)2^{2k-1}}
+ \sum_{i \neq j}\mathrm{Cov}(Z_i,Z_j)$.

We now must look at when $\mathrm{Cov}(Z_i,Z_j) \neq 0$ so that we only need consider
when $Z_i$ and $Z_j$ correspond to dependent arithmetic progressions, meaning that they
share at least one term.  Next, we note that if two arithmetic progressions $A$ and $B$
share only one term, then $\mathrm{Cov}(A,B) =\mathbb{E}(AB) - \mathbb{E}(A)\mathbb{E}(B)
= \frac{1}{2^{2k-2}} - \frac{1}{2^{k-1}}\cdot\frac{1}{2^{k-1}}= 0$.

 We have 
$\mathrm{Cov}(Z_i,Z_j) = \mathbb{E}(Z_iZ_j) - \mathbb{E}(Z_i)\mathbb{E}(Z_j)
= \mathbb{E}(Z_iZ_j) - \frac{1}{2^{2k-2}}$.  We will now give a lower bound for 
$\mathrm{Cov}(Z_i,Z_j)$.    
Given $i$, for each of those $j$ values that correspond to an arithmetic
progression that shares (exactly) $t \geq 2$ values with the
$i^{\mathrm{th}}$ arithmetic progression we have $\mathbb{E}(Z_iZ_j)=
\frac{1}{2^{2k-t-1}} - \frac{1}{2^{2k-2}} \geq \frac{1}{2^{2k-t}}$.
For each $i$, we use the trivial lower bound of $1$ for the number of values of $j$ for which the
$i^{\mathrm{th}}$ and $j^{\mathrm{th}}$ arithmetic progressions share  $t$ terms.
To prove the lemma we only need a trivial bound of $1$ value of $j$. 

Putting the above together and noting that for each $i$ we are using
$k-1$ values of $j$ based on how many common terms each has with the $i^{\mathrm{th}}$ arithmetic progression,
we get
$\sum_{i \neq j}\mathrm{Cov}(Z_i,Z_j) \geq  {\frac{n^2}{2(k-1)}} \sum_{t=2}^{k-1} \frac{1}{2^{2k-t}}
= \frac{n^2}{2(k-1)} \left(\frac{1}{2^{k+2}} - \frac{1}{2^{2k}}\right) \geq \frac{n^2}{(k-1)2^{k+4}}$.
We now are done since
$$\displaystyle \lim_{n \rightarrow \infty}
\frac{\mathbb{E}(Y_k(n))}{\mathrm{Var}(Y_k(n))} \leq \displaystyle \lim_{n \rightarrow \infty}\frac{\frac{n^2}{(k-1)2^k}}{\frac{n^2}{(k-1)2^k} + \frac{n^2 }{(k-1)2^{k+4}}} = \frac{1}{1+2^{-4}} = \frac{16}{17}<1.$$
\hfill $\Box$

\vskip 5pt
\noindent
{\bf Remark.} We were unsuccessful in our attempts to show that the limit in the above
lemma is 0 as is the case with monochromatic complete subgraphs.
\vskip 5pt

Based on the under-dispersion of the Poisson distribution compared to $Y_k$ and the
weak dependence between monochromatic arithmetic progressions, we see that we are in a  similar situation to the monochromatic subgraphs.
Hence, relying on the same heuristics, we investigate  how well the Delaporte distribution approximates these new empirical histograms
(see Figure 5, below).

\begin{figure}[h!] \tiny
	\begin{center} \hspace*{-0.1in}
		\begin{tabular}{ccc}
			\includegraphics[scale=.21]{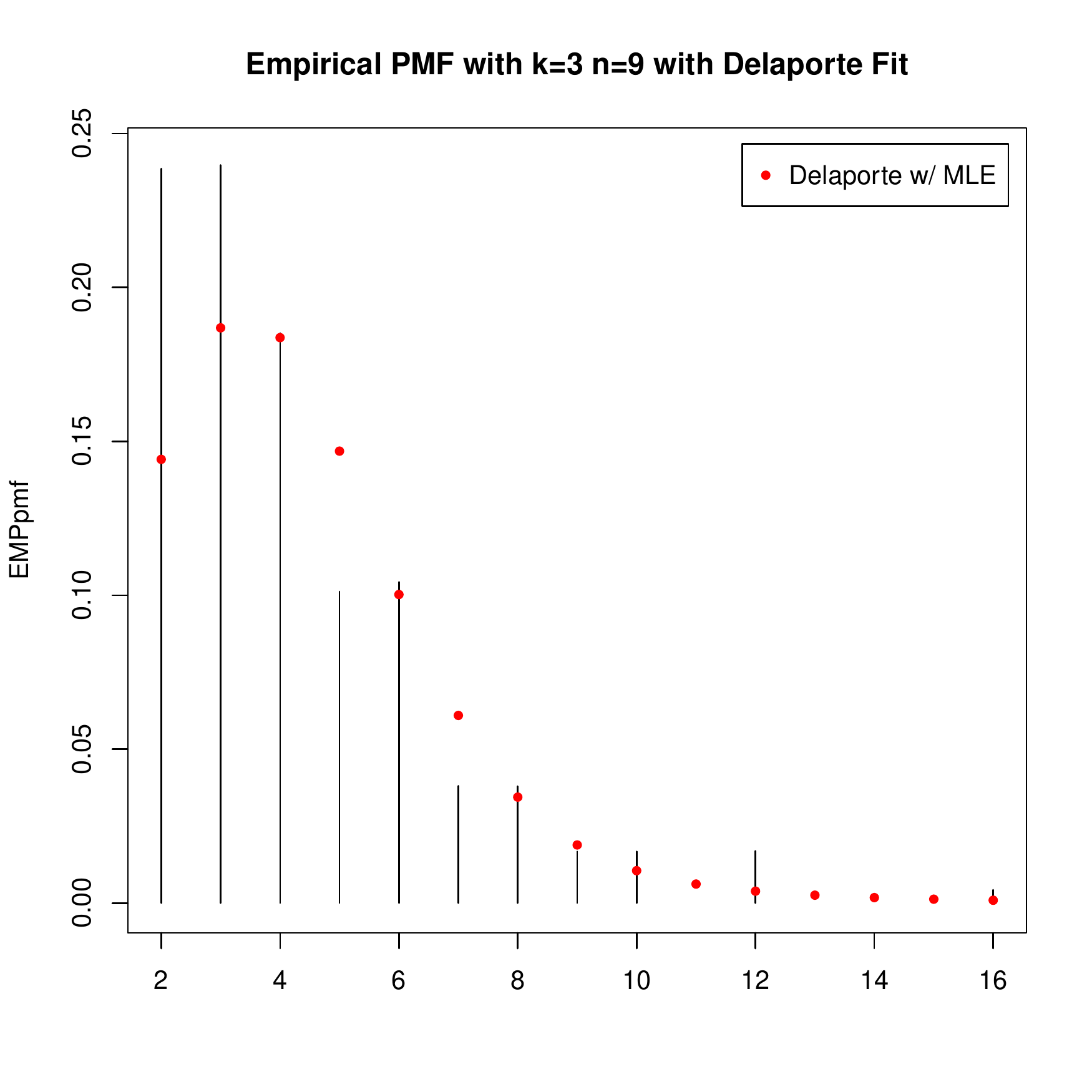} &  
			\includegraphics[scale=.21]{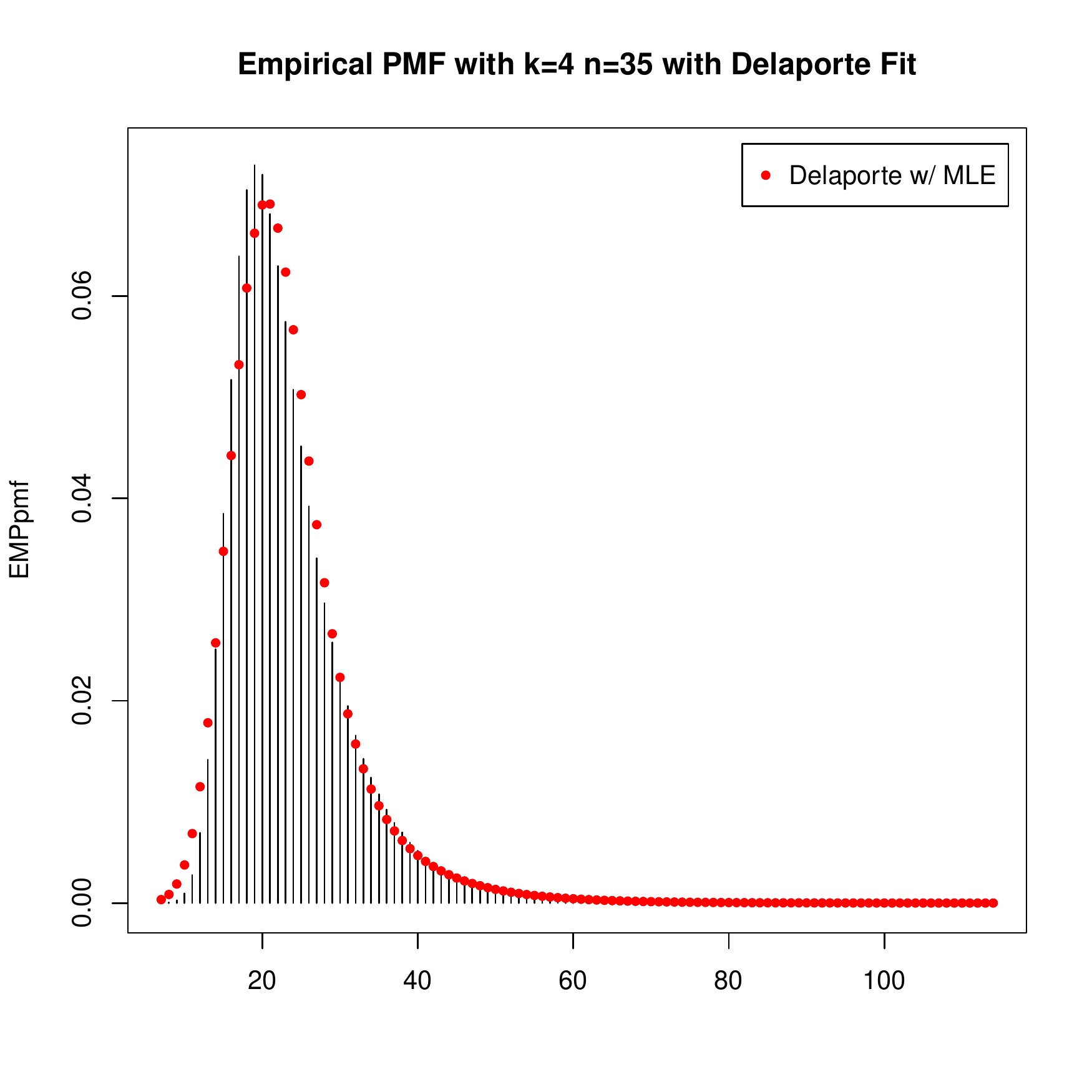} &   \includegraphics[scale=.21]{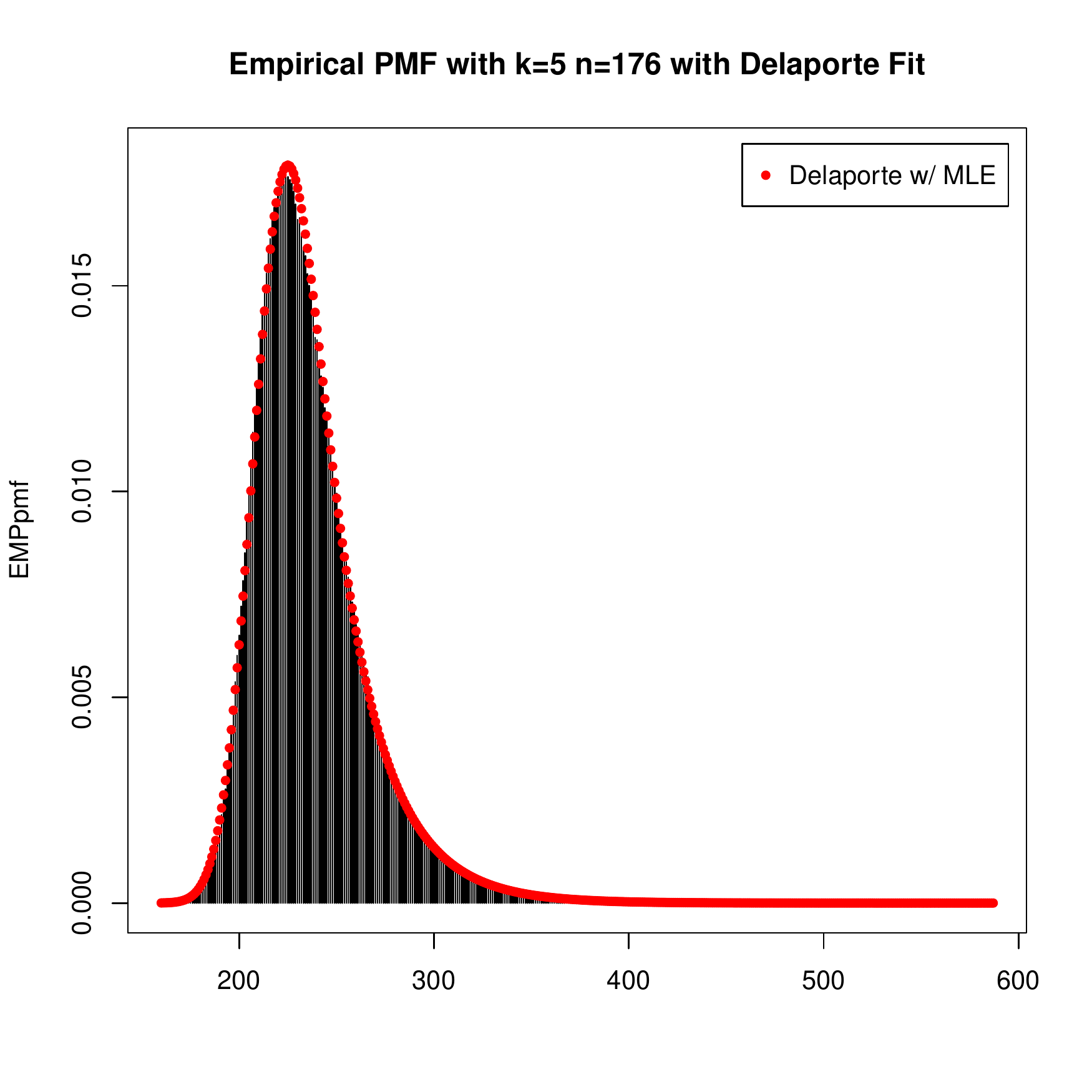} \vspace{-1em}\\ 			
			Sample size = 1M & Sample size = 1M & Sample size = 1.7M\\
			\includegraphics[scale=.21]{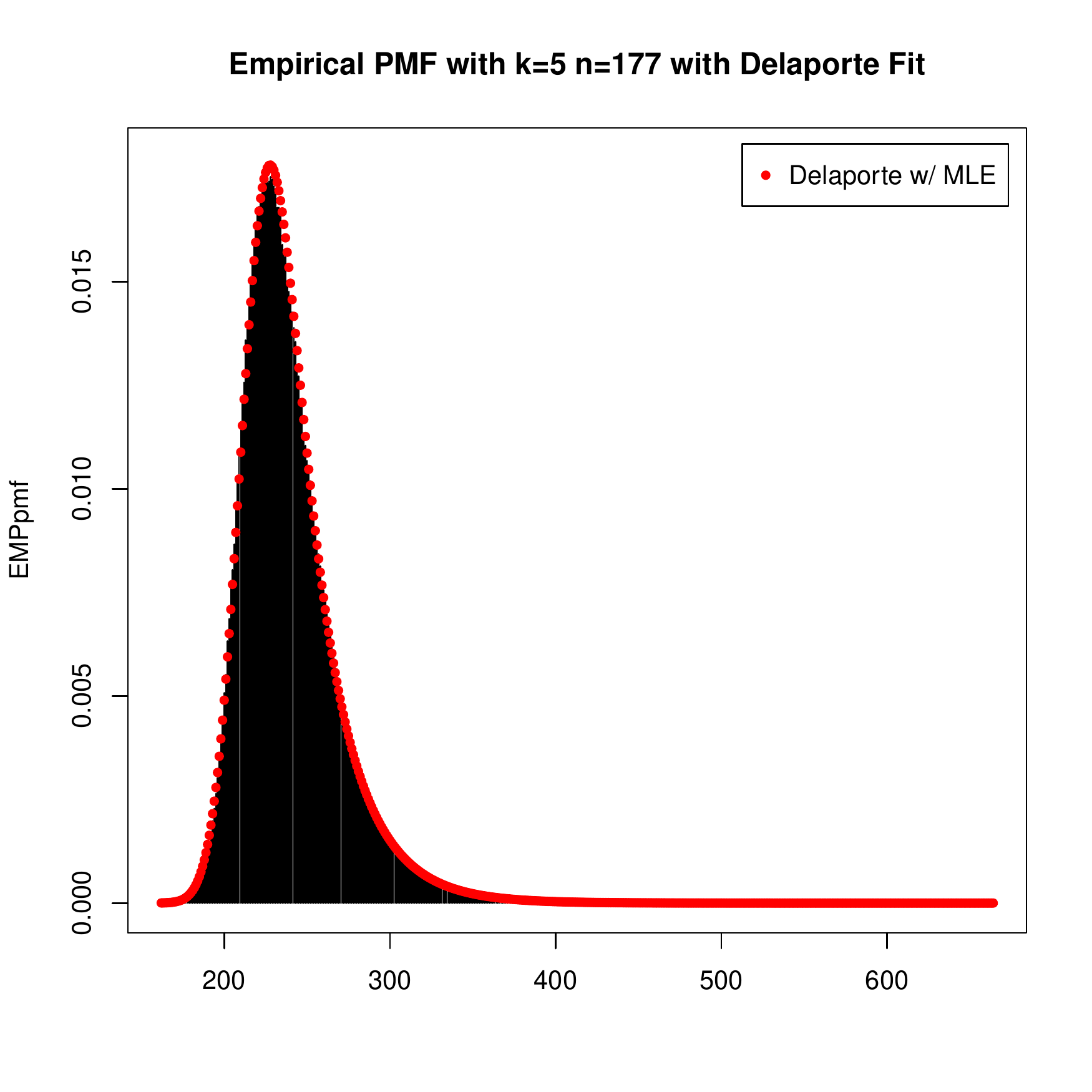} &  
			\includegraphics[scale=.21]{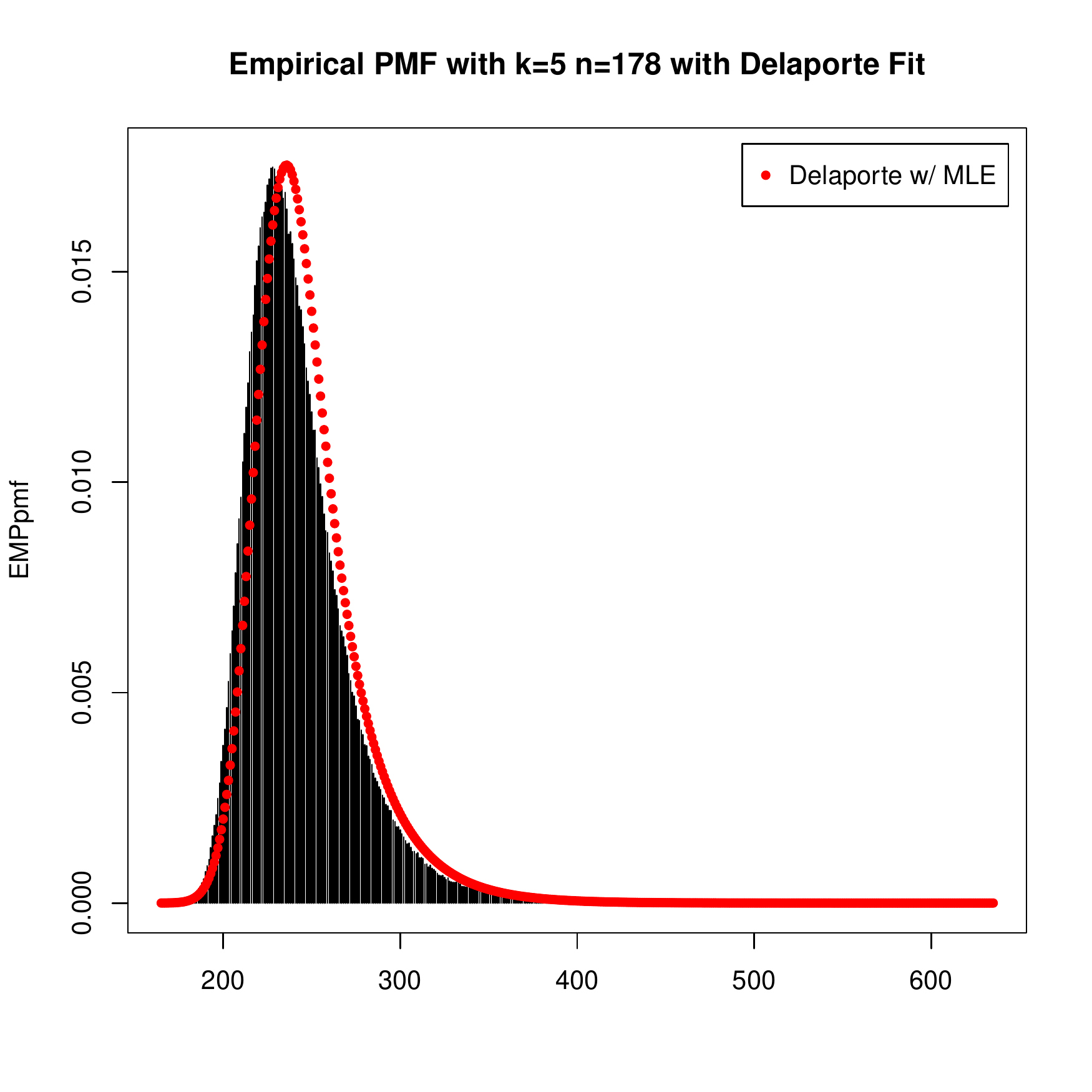} &   			\includegraphics[scale=.21]{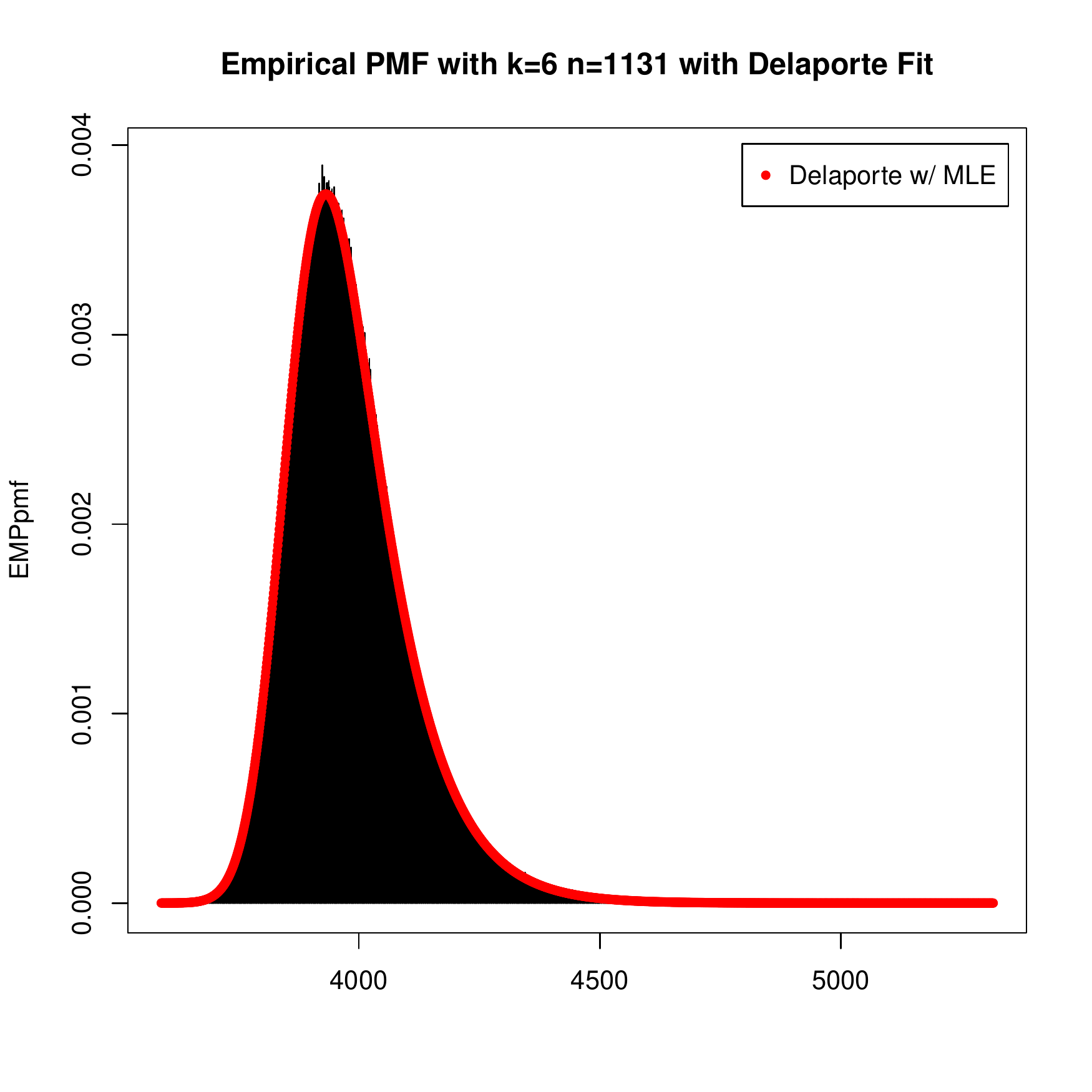} \vspace{-1em}\\ 
			Sample size = 1.1M & Sample size = 1M & Sample size = 0.85M\\
			\includegraphics[scale=.21]{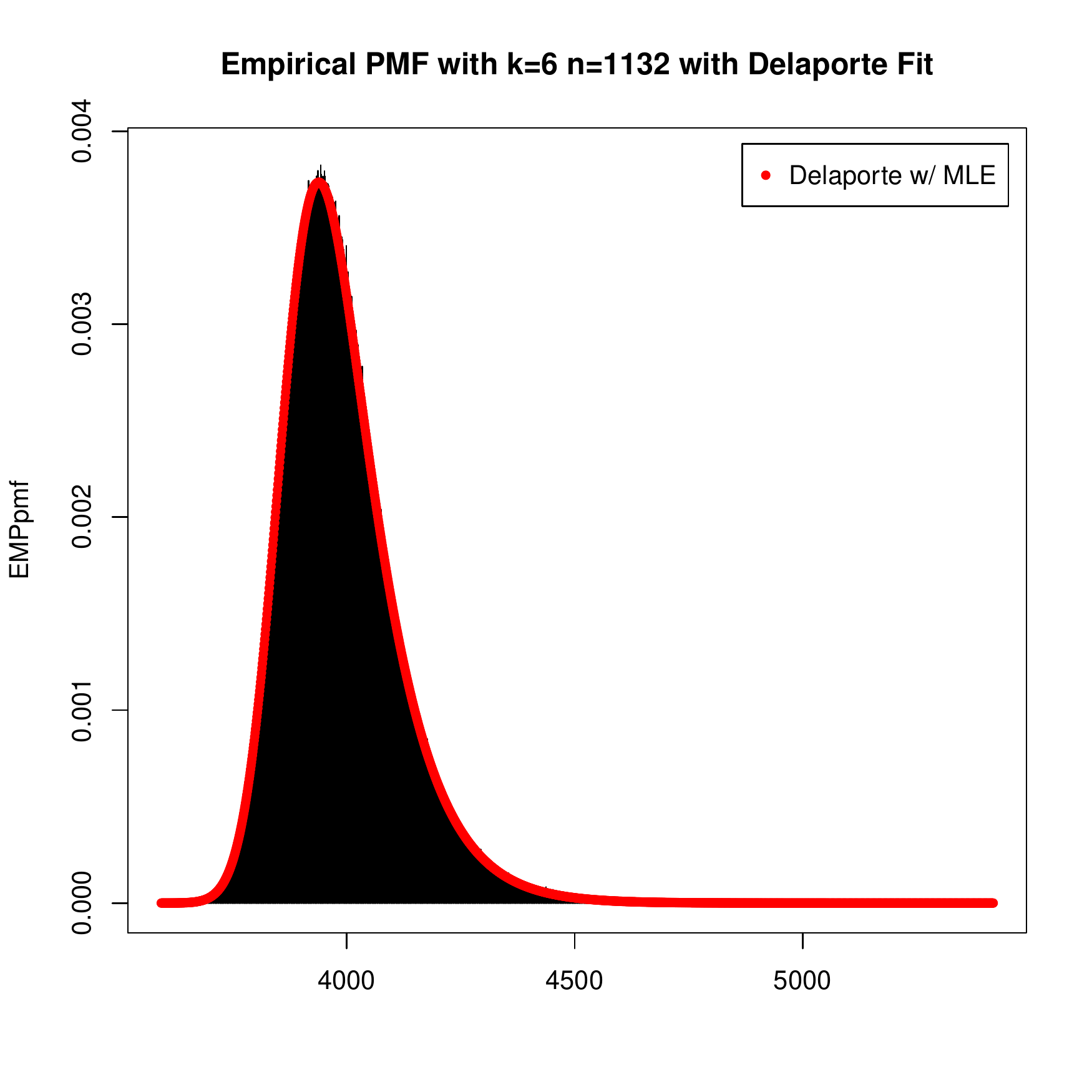} &  			\includegraphics[scale=.21]{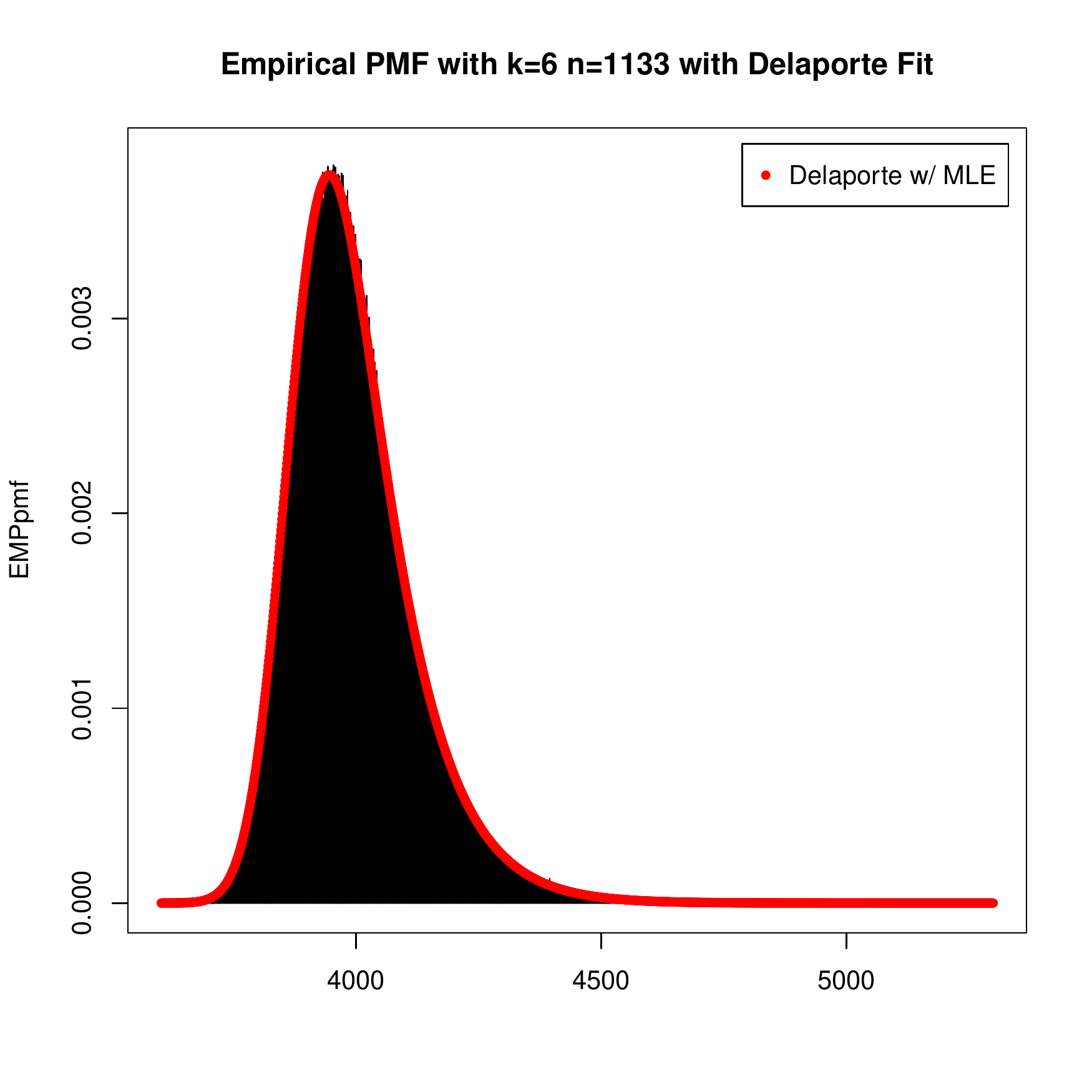} &   \includegraphics[scale=.21]{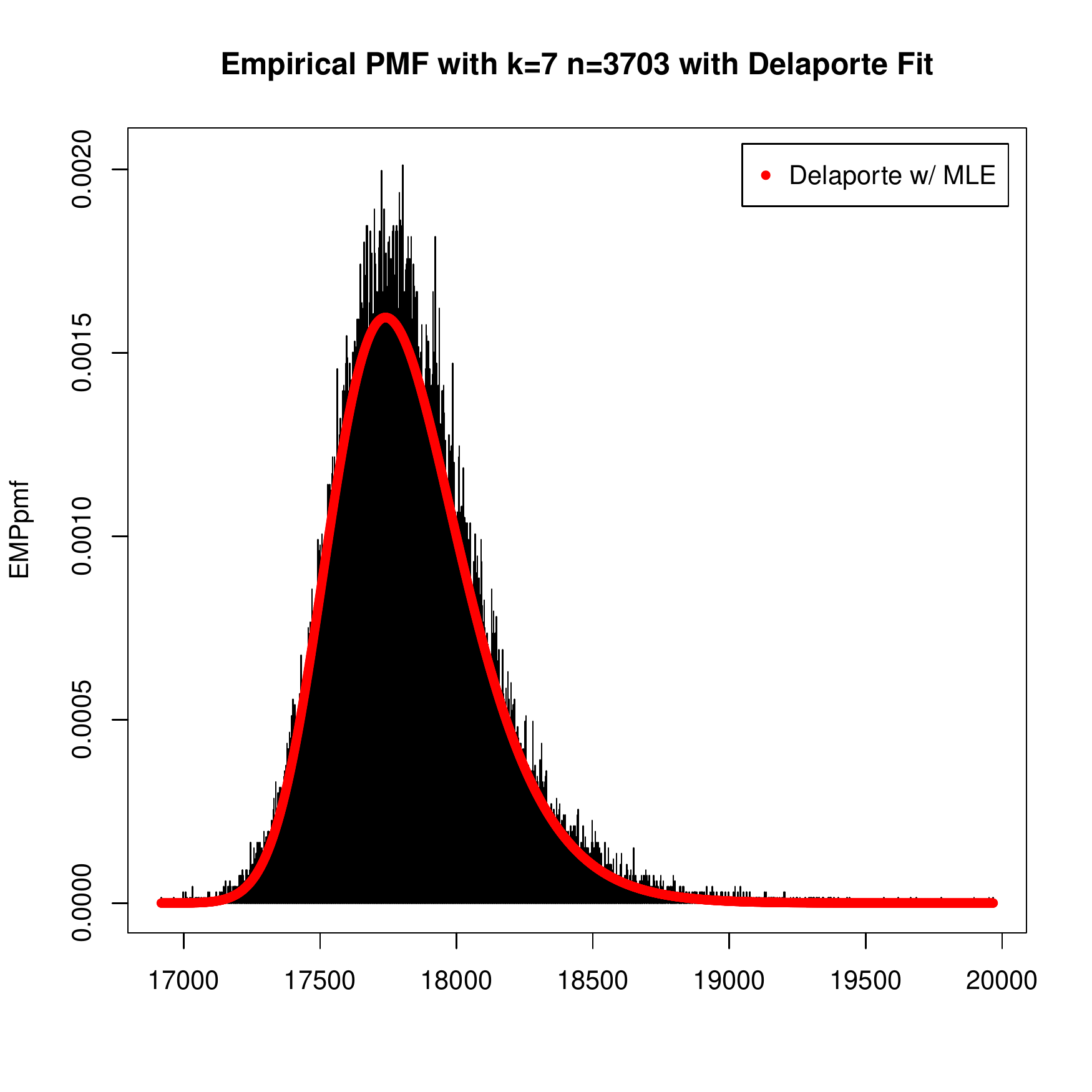} \vspace{-1em}\\ 
			Sample size = 0.89M & Sample size = 0.85M & Sample size = 67K\\ 
		\end{tabular} \normalsize
		\caption{Empirical pmfs for various scenarios with Delaporte Overlay}
	\end{center}
\end{figure}

The Delaporte distribution is, again, an unusually good
approximation, this time for the number of monochromatic arithmetic progressions.
The reader may notice that for the last histogram  in Figures 4 and 5
(with a sample
size of 67K), there are spikes at the peak and that the Delaporte
overlay misses these spikes.  Based on our many simulations (only a 
fraction of which are shown in this article), we find that
these spikes diminish as the sample size increases and that they settle
near the Delaporte peak.  We included this histogram to show what is
expected when the sample size is relatively ``small."

Having both the number of monochromatic complete subgraphs and the number of monochromatic
arithmetic progressions producing such similar empirical
histograms, 
we end with the following question:
\vskip 5pt
\centerline{
Is there a ``Delaporte Paradigm" for Ramsey objects?}


\begin{thebibliography}{30} \footnotesize \parskip=0pt

\bibitem{A} A. Adler, Delaporte: Statistical functions for the Delaporte
distribution, {\tt R} package version 3.0.0, 2016,
{\tt https://CRAN.R-project.org/package=Delaporte}.

\bibitem{Al} P. Allison, Estimation and testing for Markov model
of reinforcement, {\it Sociological Methods and Research} {\bf 8} (1980), 434-453.

\bibitem{AS} N. Alon and J. Spencer, The Probabilistic Method, fourth edition, Wiley, New Jersey, 2015.

\bibitem{G} A. Godbole, D. Skipper, and R. Sunley, The asymptotic lower bound of diagonal Ramsey numbers: a closer look,
{\it Disc. Prob.  Algorithms} {\bf 72} (1995), 81-94.

\bibitem{J} S. Janson, Poisson convergence and poisson processes with applications to random graphs,
{\it Stochastic Processes and their Applications} {\bf 26} (1987), 1-30.

\bibitem{JLR}  S. Janson, T. Luczak, and A. Rucinski, Random Graphs, Wiley-Interscience, New
York, 2000.

\bibitem{JG} D. Joanes  and C. Gill, Comparing measures of sample skewness and kurtosis,
{\it The Statistician} {\bf 47} (1998), 183-189.

\bibitem{K} M. Kouril, The van der Waerden number $W(2,6)$ is 1132, {\it Experiment. Math.} {\bf 17} (2008), 53-61.

\bibitem{Lawless} J. Lawless, Negative binomial and mixed Poisson regression,
{\it Canad. J. Stat.} {\bf 15} (1987), 209-225.

\bibitem{M} A. Marshall and I. Olkin, Bivariate distributions generated
from P\'olya-Eggenberger urn models, {\it J. Multivariate Anal.} {\bf 35} (1990), 48-65.

\bibitem{MR} B. D. McKay and S.P. Radziszowski, R(4,5) = 25, {\it Journal of Graph Theory} {\bf 19} (1995), 309-322.



\bibitem{R}   {\tt R} Core Team (2016). {\tt R}: A language and environment for statistical computing, {\tt R} Foundation for Statistical Computing, Vienna, Austria, {\tt https://www.R-project.org}.

\bibitem{V} G. Venter,   Effects of variations from Gamma-Poisson assumptions, {\it CAS Proceedings} {\bf LXXVIII} (1991), 41-55.

\bibitem{Z} D. Zeilberger, Symbolic moment calculus II: why is Ramsey theory sooooo eeeenormously hard,
{\it Integers} {\bf 7(2)} (2007), \#A34.

\end{thebibliography}
\end{document}